\documentclass[11pt]{article}

\usepackage[a4paper, margin=3cm]{geometry}

\usepackage{amsmath, amsthm, amsfonts, amssymb}
\usepackage{mathtools}

\numberwithin{equation}{section} 
\numberwithin{table}{section} 
\numberwithin{figure}{section} 

\usepackage[table,xcdraw]{xcolor} 

\usepackage[utf8]{inputenc}
\usepackage{import} 

\usepackage[title, titletoc]{appendix} 

\usepackage[shortlabels]{enumitem} 
\setlist[enumerate, 1]{label={(\roman*)}} 

\usepackage{algpseudocode}  
\usepackage{algorithm}  

\usepackage{wrapfig}
\usepackage{multicol}

\usepackage{graphicx} 

\usepackage[labelfont=bf]{caption} 
\usepackage{subcaption} 

\usepackage[backend=biber, style=alphabetic, sorting=nyt, maxbibnames=6, giveninits=true]{biblatex}

\renewbibmacro{in:}{%
  \ifboolexpr{%
     test {\ifentrytype{article}}%
     or
     test {\ifentrytype{inproceedings}}%
  }{}{\printtext{\bibstring{in}\intitlepunct}}%
}

\usepackage{hyperref} 
\hypersetup{
    colorlinks   = true, 
    urlcolor     = blue, 
    linkcolor    = blue, 
    citecolor    = red   
}
\usepackage{url}
\theoremstyle{plain} 
\newtheorem{theorem}{Theorem}[section] 
\newtheorem{corollary}[theorem]{Corollary} 
\newtheorem{lemma}[theorem]{Lemma}

\theoremstyle{definition} 

\newtheorem*{definition*}{Definition}
\newtheorem*{theorem*}{Theorem} 
\newtheorem*{corollary*}{Corollary}
\newtheorem*{lemma*}{Lemma}
\newtheorem*{proposition*}{Proposition}

\newtheorem*{note*}{Note}
\newtheorem{example}[theorem]{Example}
\newtheorem*{example*}{Example}

\newtheorem*{exercise*}{Exercise}
\newtheorem{remark}[theorem]{Remark}
\newtheorem*{remark*}{Remark}

\newcommand{\prob}[2][]{\mathbb{P}_{#1}\left(#2\right)} 
\newcommand{\ev}[2][]{\mathbb{E}_{#1}\left[#2\right]} 

\newcommand{\reals}{\mathbb{R}} 

\newcommand{\innprod}[2]{\left<#1 , #2\right>} 
\newcommand{\norm}[1]{\lVert #1 \rVert} 
\newcommand{\normbig}[1]{\left\lVert #1 \right\rVert} 

\newcommand{\rank}{\operatorname{rank}} 
\newcommand{\Null}{\operatorname{Null}} 
\newcommand{\Range}{\operatorname{Range}} 

\newcommand*{\tran}{{\mkern-1.5mu\mathsf{T}}} 

\DeclareMathOperator*{\argmin}{arg\,min}

\newcommand{\rvline}{\hspace*{-\arraycolsep}\vline\hspace*{-\arraycolsep}} 
\newcommand\Tstrut{\rule{0pt}{2.9ex}} 
\newcommand\Bstrut{\rule[-1.2ex]{0pt}{0pt}} 

\let\epsilon\varepsilon

\let\phi\varphi

\title{A subspace constrained randomized Kaczmarz method\\for structure or external knowledge exploitation}

\author{
    Jackie Lok\thanks{
    Department of Operations Research and Financial Engineering, Princeton University, \texttt{jackie.lok@princeton.edu}, \texttt{elre@princeton.edu}}
    \and
    Elizaveta Rebrova\footnotemark[1]
}

\date{}

\begin{document}

\maketitle

\begin{abstract}
We study a version of the randomized Kaczmarz algorithm for solving systems of linear equations where the iterates are confined to the solution space of a selected subsystem. We show that the subspace constraint leads to an accelerated convergence rate, especially when the system has approximately low-rank structure. On Gaussian-like random data, we show that it results in a form of dimension reduction that effectively increases the aspect ratio of the system. Furthermore, this method serves as a building block for a second, quantile-based algorithm for solving linear systems with arbitrary sparse corruptions, which is able to efficiently utilize external knowledge about corruption-free equations and achieve convergence in difficult settings. Numerical experiments on synthetic and realistic data support our theoretical results and demonstrate the validity of the proposed methods for even more general data models than guaranteed by the theory.
\end{abstract}



\section{Introduction} \label{sec:intro}

A ubiquitous problem across the sciences is solving large-scale systems of linear equations $\mathbf{A} \mathbf{x} = \mathbf{b}$, for which scalable and efficient iterative methods are useful when it is too slow or infeasible to solve the system directly. Instead of solving such problems obliviously, it is natural to have insights into the structural properties of the linear system of interest, such as being approximately low-rank. Moreover, external knowledge about trustworthy observations in the presence of corrupted measurements could be available.

Now, additional information about the structure of a linear system influences the choice of the most suitable solver. In this work, we take an adaptive, problem-aware approach to account for various types of auxiliary information by augmenting the iterations of a generic iterative linear solver based on a distinguished subsystem of equations.

The generic solver that we consider is the \emph{Kaczmarz algorithm}~\cite{Kaczmarz1937}, which is an iterative, row-action method for solving large-scale, typically overdetermined systems of linear equations.
It is a special case of the alternating projection method that has low computational cost and storage per iteration, and can be used in the streaming setting where a single row (or block of rows) of the system can be accessed at a time. Besides its traditional applications in areas such as image reconstruction~\cite{Natterer2001, Herman2009} and signal processing~\cite{CenkerEtAl1992}, the Kaczmarz algorithm has recently been used as a building block for more sophisticated methods to design linear solvers~\cite{derezinski2023solving, du2021kaczmarz}, and to address problems such as phase retrieval~\cite{TanVershynin2019} and tensor recovery~\cite{chen2021regularized}.

In each iteration of the Kaczmarz algorithm, a row $\mathbf{a}_j$ of the matrix $\mathbf{A}$ is selected, and the current iterate $\mathbf{x}^k$ is projected onto the hyperplane $\mathbf{a}_j^{\tran} \mathbf{x} = b_j$ by
\begin{equation} \label{eq:rk_update}
    \mathbf{x}^{k+1} = \mathbf{x}^k + \frac{b_j - \mathbf{a}_j^{\tran} \mathbf{x}^k}{\norm{\mathbf{a}_j}} \cdot \frac{\mathbf{a}_j}{\norm{\mathbf{a}_j}} .
\end{equation}
In their seminal paper, Strohmer and Vershynin~\cite{StrohmerVershynin2009} show that if the system of linear equations is consistent and has unique solution $\mathbf{x}^*$, then the \emph{randomized Kaczmarz (RK) algorithm}, which samples each row independently with probability $\norm{\mathbf{a}_j}^2 / \norm{\mathbf{A}}_F^2$ at each iteration, converges to $\mathbf{x}^*$ in expectation with an exponential rate (i.e.\ linearly) that depends on the geometric properties of $\mathbf{A}$ (or more precisely, its scaled condition number $\norm{\mathbf{A}}_F / \sigma_{\mathrm{min}}(\mathbf{A})$):
\begin{equation} \label{eq:rk_convergence_bound}
    \mathbb{E} \norm{\mathbf{x}^k - \mathbf{x}^*}^2 \leq \left( 1 - \frac{\sigma_{\mathrm{min}}(\mathbf{A})^2}{\norm{\mathbf{A}}_F^2} \right)^k \cdot \norm{\mathbf{x}^0 - \mathbf{x}^*} .
\end{equation}
Subsequently, many variants of the randomized Kaczmarz method have been analyzed; we defer a detailed discussion of related works to Section~\ref{sec:related_work} after presenting our results.

In this paper, we propose two Kaczmarz-based algorithms that can exploit (a) \emph{approximately low-rank structure and geometric properties of the matrix in a system of linear equations to accelerate convergence}; and (b) \emph{external knowledge about corruption-free equations for linear systems with arbitrary sparse corruptions to enable convergence even in the highly corrupted regime}.

\subsection{Setup and notation} \label{subsec:setup}

We consider a consistent, overdetermined (i.e.\ tall) system of linear equations $\mathbf{A} \mathbf{x} = \mathbf{b}$, or \emph{linear system} for short, where the rows of $\mathbf{A} \in \reals^{m \times n}$ are denoted by $\mathbf{a}_1, \mathbf{a}_2, \dots, \mathbf{a}_m \in \reals^n$, $\mathbf{b} \in \reals^{m}$, and $m \geq n$. We assume throughout that $\mathbf{A}$ has full rank, and denote the unique solution of the linear system by $\mathbf{x}^* \in \reals^n$. We work in the real setting for simplicity, but everything can be generalized to the complex setting.

Vectors, oriented as columns by default,  and matrices are written in boldface. The vector $\ell_2$-norm is denoted by $\norm{\cdot}$, and the matrix spectral and Frobenius norms are denoted by $\norm{\cdot}$ and $\norm{\cdot}_F$. The singular values of a matrix $\mathbf{A} \in \reals^{m \times n}$ are denoted by $\sigma_{\mathrm{max}}(\mathbf{A}) = \sigma_1(\mathbf{A}) \geq \sigma_2(\mathbf{A}) \geq \dots \geq \sigma_{\min\{m, n\}}(\mathbf{A}) = \sigma_{\mathrm{min}}(\mathbf{A})$, and the smallest non-zero singular value is denoted by $\sigma_{\mathrm{min}}^+(\mathbf{A})$. The Moore-Penrose pseudoinverse of $\mathbf{A}$ is denoted by $\mathbf{A}^{\dagger}$. We refer to the row submatrix of $\mathbf{A}$ (resp. subvector of $\mathbf{b}$) indexed by $I \subseteq [m] \coloneqq \{  1, 2, \dots, m \}$ by $\mathbf{A}_I$ (resp. $\mathbf{b}_I$). The solution space $\mathbf{A}_I \mathbf{x} = \mathbf{b}_I$ of a linear system refers to the affine subspace $\{ \mathbf{x} \in \reals^n : \mathbf{A}_I \mathbf{x} = \mathbf{b}_I \}$.

\subsection{Methods and main results} \label{subsec:main_results}

\subsubsection{The SCRK method}

Fix a subset $I_0 \subset [m]$ of indices of rows of $\mathbf{A}$ with $m_0 \coloneqq |I_0| < n$, and denote the remaining indices by $I_1 \coloneqq [m] \setminus I_0$. We define a variant of the RK algorithm that confines the iterates within the solution space $\mathbf{A}_{I_0} \mathbf{x} = \mathbf{b}_{I_0}$, which we will refer to as the \emph{subspace constrained randomized Kaczmarz (SCRK) method}. Each update of the SCRK algorithm consists of a projection of the current iterate $\mathbf{x}^k$ onto the solution space $\mathbf{A}_{I_0 \cup \{ j \}} \mathbf{x} = \mathbf{b}_{I_0 \cup \{ j \}}$, where the row corresponding to $j \in I_1$ is sampled according to an input probability distribution, and can be algebraically expressed by
\begin{equation} \label{eq:scrk_block_update}
    \mathbf{x}^{k+1} = \mathbf{x}^k + \mathbf{A}_{I_0 \cup \{j \}}^{\dagger} (\mathbf{b}_{I_0 \cup \{ j \}} - \mathbf{A}_{I_0 \cup \{ j \}} \mathbf{x}^k) .
\end{equation}
This is essentially a block Kaczmarz update~\cite{Elfving1980, NeedellTropp2014}, but with $I_0$ fixed throughout the iterations so that the iterates are confined within the selected solution space $\mathbf{A}_{I_0} \mathbf{x} = \mathbf{b}_{I_0}$. Reusing the same block allows for properties of the distinguished subsystem $\mathbf{A}_{I_0} \mathbf{x} = \mathbf{b}_{I_0}$ to be exploited, and also leads to a more efficient update formula: in Lemma~\ref{lem:scrk_update_formula}, we prove that as long as $\mathbf{x}^k$ satisfies $\mathbf{A}_{I_0} \mathbf{x}^k = \mathbf{b}_{I_0}$ and $\mathbf{a}_j \notin \Range(\mathbf{A}_{I_0}^{\tran})$, \eqref{eq:scrk_block_update} simplifies to
\begin{equation} \label{eq:scrk_update_main}
    \mathbf{x}^{k+1} = \mathbf{x}^k + \frac{b_j - \mathbf{a}_j^{\tran} \mathbf{x}^k}{\norm{\mathbf{P} \mathbf{a}_j}} \cdot \frac{\mathbf{P} \mathbf{a}_j}{{\norm{\mathbf{P} \mathbf{a}_j}}} ,
\end{equation}
where $\mathbf{P} \coloneqq \mathbf{I} - \mathbf{A}_{I_0}^{\dagger} \mathbf{A}_{I_0}$ is the orthogonal projector onto $\Null(\mathbf{A}_{I_0}) = \Range(\mathbf{A}_{I_0}^{\tran})^{\perp}$.
Unlike the block update~\eqref{eq:scrk_block_update}, this does not require a new pseudoinverse to be computed at every iteration and thus can be performed faster. The SCRK method, which leverages~\eqref{eq:scrk_update_main}, is summarized in Algorithm~\ref{alg:scrk_method}. For concreteness, we fix a particular sampling distribution for the rows of $\mathbf{A}_{I_1}$ that leads to an especially simple and interpretable analysis. By varying the distribution, better convergence rates may be possible~\cite{GowerRichtarik2015, AgWaLu2014}.

On a conceptual level, the SCRK update~\eqref{eq:scrk_update_main} is reminiscent of the usual Kaczmarz update~\eqref{eq:rk_update}, with the new direction $\mathbf{P} \mathbf{a}_j$ representing the ``extra information'' offered by $\mathbf{a}_j$ beyond that which is already known from being in the solution space $\mathbf{A}_{I_0} \mathbf{x} = \mathbf{b}_{I_0}$.

\begin{algorithm}[!htb]
\caption{Subspace Constrained Randomized Kaczmarz (SCRK)} \label{alg:scrk_method}
\begin{algorithmic}[1]
    \Procedure{SCRK}{$\mathbf{A}, \mathbf{b}, I_0, K$}
        \State $\mathbf{P} = \mathbf{I} - \mathbf{A}_{I_0}^{\dagger} \mathbf{A}_{I_0}$ \Comment{Orthogonal projector onto $\Null(\mathbf{A}_{I_0})$}
        \State \textbf{initialize} $\mathbf{x}^0 = \mathbf{A}_{I_0}^{\dagger} \mathbf{b}_{I_0}$ \Comment{Initial iterate $\mathbf{x}^0$ solves $\mathbf{A}_{I_0} \mathbf{x}^0 = \mathbf{b}_{I_0}$}
        \For{$k = 1, \dots, K$}
            \State sample $j \in [m] \setminus I_0$ with prob. $\norm{\mathbf{P} \mathbf{a}_j}^2 / \norm{\mathbf{A}_{I_1} \mathbf{P}}_F^2$  \Comment{Sample row in $\mathbf{A}_{I_1}$}
            \State $\mathbf{x}^{k} = \mathbf{x}^{k-1} + \frac{b_j - \mathbf{a}_j^{\tran} \mathbf{x}^{k-1}}{\norm{\mathbf{P} \mathbf{a}_j}} \cdot \frac{\mathbf{P} \mathbf{a}_j}{\norm{\mathbf{P} \mathbf{a}_j}}$ \Comment{Project onto $\mathbf{A}_{{I_0} \cup \{ j \}} \mathbf{x} = \mathbf{b}_{I_0 \cup \{ j \}}$}
        \EndFor
        \State \textbf{return} $\mathbf{x}^K$
    \EndProcedure
\end{algorithmic}
\end{algorithm}

The following result, proved in Section~\ref{subsec:update_formula}, shows that the SCRK method converges linearly in expectation to the solution $\mathbf{x}^*$ of $\mathbf{A} \mathbf{x} = \mathbf{b}$ under minimal assumptions.

\begin{theorem} \label{thm:scrk_convergence}
Suppose that the rows of $\mathbf{A}$ are partitioned into two blocks $\mathbf{A}_{I_0}$ and $\mathbf{A}_{I_1}$ of sizes $m_0$ and $m - m_0$ respectively. Let $\mathbf{P} = \mathbf{I} - \mathbf{A}_{I_0}^{\dagger} \mathbf{A}_{I_0}$ be the orthogonal projector onto $\Null(\mathbf{A}_{I_0})$, and $\sigma_{\mathrm{min}}^+(\mathbf{A}_{I_1} \mathbf{P})$ be the smallest non-zero singular value of $\mathbf{A}_{I_1} \mathbf{P}$. Then the SCRK iterates $\mathbf{x}^k$ from Algorithm~\ref{alg:scrk_method} converge to the solution $\mathbf{x}^*$ in expectation with
\begin{equation} \label{eq:scrk_convergence_1}
    \mathbb{E} \norm{\mathbf{x}^k - \mathbf{x}^*}^2 \leq \left( 1 - \frac{\sigma_{\mathrm{min}}^+(\mathbf{A}_{I_1} \mathbf{P})^2}{\norm{\mathbf{A}_{I_1} \mathbf{P}}_F^2} \right)^k \cdot \norm{\mathbf{x}^0 - \mathbf{x}^*}^2 .
\end{equation}
\end{theorem}

In the special case that the row spaces of $\mathbf{A}_{I_0}$ and $\mathbf{A}_{I_1}$ are orthogonal (i.e.\ $\mathbf{A}_{I_0} \mathbf{A}_{I_1}^{\tran} = \mathbf{0}$), the SCRK updates~\eqref{eq:scrk_update_main} reduce to the usual Kaczmarz updates~\eqref{eq:rk_update} since $\mathbf{P} \mathbf{a}_j = \mathbf{a}_j$ for all $j \in I_1$, and hence we immediately deduce the following:

\begin{corollary} \label{cor:scrk_convergence_orthog}
Consider the same setup as Theorem~\ref{thm:scrk_convergence}. If $\mathbf{A}_{I_0} \mathbf{A}_{I_1}^{\tran} = \mathbf{0}$, then
\begin{equation} \label{eq:scrk_convergence_orthog_1}
    \mathbb{E} \norm{\mathbf{x}^k - \mathbf{x}^*}^2 \leq \left( 1 - \frac{\sigma_{\mathrm{min}}^+(\mathbf{A}_{I_1})^2}{\norm{\mathbf{A}_{I_1}}_F^2} \right)^k \cdot \norm{\mathbf{x}^0 - \mathbf{x}^*}^2 .
\end{equation}
\end{corollary}

\textbf{Noisy linear systems.}
In Section~\ref{sec:noisy_case}, we prove that for inconsistent systems of linear equations where a noisy measurement vector $\widehat{\mathbf{b}} \ne \mathbf{b}$ is observed, the SCRK method converges at the same rate up to an error horizon around the solution $\mathbf{x}^*$ with a radius that depends on the noise in $I_0$ and $I_1$, as well as the geometries of $\mathbf{A}_{I_0}$ and $\mathbf{A}_{I_1} \mathbf{P}$ (Theorem~\ref{thm:noisy_scrk_convergence}). This expands on a phenomenon that is known from previous analyses of Kaczmarz methods~\cite{Needell2010, NeedellTropp2014}.

The analysis of inconsistent linear systems requires developing technical results involving a two-step decomposition of the block update~\eqref{eq:scrk_block_update}, which provides a partial generalization of the two-subspace Kaczmarz method in~\cite{NeedellWard2013} (see Remark~\ref{rmk:scrk_block_update_general}).

\begin{remark}[Per-iteration complexity] \label{rmk:complexity}
Each SCRK iteration can be computed in $O(m_0 n)$ flops, with calculating $\mathbf{P} \mathbf{a}_j = \mathbf{a}_j - \mathbf{A}_{I_0}^{\dagger} \mathbf{A}_{I_0} \mathbf{a}_j$ being the most expensive step. This requires directly computing $\mathbf{A}_{I_0}^{\dagger}$ or an orthonormal basis for $\Range(\mathbf{A}_{I_0}^{\tran})$\footnote{The projector $\mathbf{A}_{I_0}^{\dagger} \mathbf{A}_{I_0}$ can be written as $\mathbf{Q} \mathbf{Q}^{\tran}$ where $\mathbf{Q} \in \reals^{n \times m_0}$ is a matrix whose columns form an orthonormal basis of $\Range(\mathbf{A}_{I_0}^{\tran})$, which can be computed in $O(m_0^2 n)$ flops.} only once, using a method based on QR decomposition or SVD. This per-iteration cost is comparable to the $O(n)$ flops per iteration of RK if $m_0$ is not too large. The overall complexity is then determined by multiplying the per-iteration cost by the number of iterations required to reach a desired error, which we will elaborate upon below.

\end{remark}

\textbf{Exploiting low-rank structure with the SCRK method.}
The convergence rate of the SCRK algorithm depends on the geometric properties of $\mathbf{A}$ and $\mathbf{P}$: Theorem~\ref{thm:scrk_convergence} shows that $k_\epsilon \coloneqq \kappa(\mathbf{A}_{I_1} \mathbf{P})^2  \log\left( 1 / \epsilon \right)$ iterations suffice to achieve the relative error guarantee $\mathbb{E} \norm{\mathbf{x}^k - \mathbf{x}^*}^2 \leq \epsilon \norm{\mathbf{x}^0 - \mathbf{x}^*}^2$, where $\kappa(\mathbf{A}_{I_1} \mathbf{P}) \coloneqq \norm{\mathbf{A}_{I_1} \mathbf{P}}_F / \sigma_{\mathrm{min}}^+(\mathbf{A}_{I_1} \mathbf{P})$ is a scaled condition number of $\mathbf{A}_{I_1} \mathbf{P}$. For the same guarantee using RK, from~\eqref{eq:rk_convergence_bound}, $\kappa(\mathbf{A})^2 \log(1 / \epsilon)$ iterations are required, where $\kappa(\mathbf{A}) \coloneqq \norm{\mathbf{A}}_F / \sigma_{\mathrm{min}}(\mathbf{A})$.

Since $\kappa(\mathbf{A}_{I_1} \mathbf{P}) \leq \kappa(\mathbf{A})$, we see that the projector $\mathbf{P}$ acts as a right preconditioner for $\mathbf{A}$, improving the convergence rate of SCRK compared to RK. In particular, we can expect a significant per-iteration improvement (and hence overall advantage) if $\norm{\mathbf{A}_{I_1} \mathbf{P}}_F \ll \norm{\mathbf{A}}_F$ or $\sigma_{\mathrm{min}}^+(\mathbf{A}_{I_1} \mathbf{P}) \gg \sigma_{\mathrm{min}}(\mathbf{A})$. We examine the connection between the geometry of $\mathbf{A}$ and convergence rates in more detail in Section~\ref{subsec:scrk_geometry} to show that the SCRK method is able to exploit approximately low-rank structure and geometric properties of $\mathbf{A}$ to accelerate convergence.

We also describe how a good subset $I_0$ of rows, if not explicitly known, can actually be \emph{efficiently found} via a connection to low-rank matrix approximation in Section~\ref{subsec:sampling_subspace}.

\medskip

\textbf{SCRK on random data and dimension reduction.}
In a somewhat complementary setting, we show that for ``unstructured'' matrices, the subspace constraint imposed by the projector $\mathbf{P}$ \emph{acts as a form of dimension reduction that effectively increases the aspect ratio of the system to reflect the dimensionality of the solution that remains unsolved} in Section~\ref{sec:random_matrix_analysis}. More precisely, we prove that when $\mathbf{A} \in \reals^{m \times n}$ is drawn from a generic class of ``Gaussian-like'' random matrices, the SCRK method typically converges with a rate that is approximately $1 - 1 / (n - m_0)$ as long as the ``effective aspect ratio'' $(m - m_0) / (n - m_0)$ of the system is sufficiently large (Theorem~\ref{thm:subgauss_scrk_convergence}). Note that $1 - 1 /(n - m_0)$ is the best possible rate that can be achieved by the RK algorithm (with any sampling distribution) on a consistent $(m - m_0) \times (n - m_0)$ linear system (see~\cite{GowerRichtarik2015}).

\subsubsection{The QuantileSCRK method}

We also propose a modification of the SCRK method for solving corrupted systems of linear equations. This setting models applications where some measurements are corrupted by arbitrarily large errors, which may occur during the data collection, transmission, and storage process due to malfunctioning sensors or faulty components (for more examples, see~\cite{StKuPoBo2012, HaddockNeedell2019}). Unlike the noisy setting above, the error horizon is not very meaningful since significant outliers can be introduced. Hence, the aim is to converge to the solution $\mathbf{x}^*$ \emph{exactly} by identifying and avoiding corruptions, which may be possible if the number of corruptions is relatively small and the system is highly overdetermined.

Our model for \emph{corrupted linear systems} is defined as follows. Let $\mathcal{C} \subseteq [m]$ and $\mathbf{b}_{\mathcal{C}} \in \reals^n$ be a \emph{sparse} vector of \emph{arbitrary} (possibly adversarial) corruptions supported on $\mathcal{C}$. Moreover, suppose that we possess external knowledge in the form of a corruption-free subset $I_0 \subset [m]$ of size $m_0$ such that $(\mathbf{b}_{\mathcal{C}})_{I_0} = \mathbf{0}$; for example, this could reflect a set of trustworthy measurements by a reliable source, or infallible equations arising from physical laws. The goal is to reconstruct the solution $\mathbf{x}^*$ of $\mathbf{A} \mathbf{x} = \mathbf{b}$ given $\mathbf{A}$, $I_0$, and the corrupted measurements $\widetilde{\mathbf{b}} \coloneqq \mathbf{b} + \mathbf{b}_{\mathcal{C}}$.

To achieve convergence, we take inspiration from the QuantileRK method proposed in~\cite{HaNeReSw2022}, which modifies the RK algorithm so that each projection is sampled from a set of admissible rows whose residuals $|b_j - \mathbf{a}_j^{\tran} \mathbf{x}^k|$ are smaller than the $q$\textsuperscript{th} quantile of residual sizes at each iteration for some parameter $q \in (0, 1]$.
This modification is based on the heuristic that large residuals should be indicative of corrupted measurements, and small residuals lead to small steps that cannot divert the iterate too far away from the solution. We propose to exploit the auxiliary information by confining the iterates of QuantileRK within the ``trusted'' solution space $\mathbf{A}_{I_0} \mathbf{x} = \mathbf{b}_{I_0}$. We will refer to this procedure, summarized in Algorithm~\ref{alg:quantilescrk_method}, as the \emph{QuantileSCRK method}.

\begin{algorithm}[!htb]
\caption{Quantile Subspace Constrained Randomized Kaczmarz (QuantileSCRK)} \label{alg:quantilescrk_method}
\begin{algorithmic}[1]
    \Procedure{QuantileSCRK}{$\mathbf{A}, \widetilde{\mathbf{b}}, I_0, q, K$}
        \State $\mathbf{P} = \mathbf{I} - \mathbf{A}_{I_0}^{\dagger} \mathbf{A}_{I_0}$ \Comment{Orthogonal projector onto $\Null(\mathbf{A}_{I_0})$}
        \State \textbf{initialize} $\mathbf{x}^0 = \mathbf{A}_{I_0}^{\dagger} \widetilde{\mathbf{b}}_{I_0}$ \Comment{Initial iterate $\mathbf{x}^0$ solves $\mathbf{A}_{I_0} \mathbf{x}^0 = \widetilde{\mathbf{b}}_{I_0}$}
        \For{$k = 1, \dots, K$}
            \State $\gamma_q = q\text{-quantile}\left\{ |\widetilde{b}_j - \mathbf{a}_j^{\tran} \mathbf{x}^k| ,\, j \in [m] \setminus I_0 \right\}$ \Comment{Threshold based on residuals}
            \State $J = \left\{ j \in [m] \setminus I_0 : |\widetilde{b}_j - \mathbf{a}_j^{\tran} \mathbf{x}^k| \leq \gamma_q \right\}$ \Comment{Set of admissible rows}
            \State sample $j \in J$ with prob. proportional to $\norm{\mathbf{P} \mathbf{a}_j}^2$ \Comment{Sample admissible row}
            \State $\mathbf{x}^{k} = \mathbf{x}^{k-1} + \frac{\widetilde{b}_j - \mathbf{a}_j^{\tran} \mathbf{x}^{k-1}}{\norm{\mathbf{P} \mathbf{a}_j}} \cdot \frac{\mathbf{P} \mathbf{a}_j}{\norm{\mathbf{P} \mathbf{a}_j}}$ \Comment{Project onto $\mathbf{A}_{{I_0} \cup \{ j \}} \mathbf{x} = \widetilde{\mathbf{b}}_{I_0 \cup \{ j \}}$}
        \EndFor
        \State \textbf{return} $\mathbf{x}^K$
    \EndProcedure
\end{algorithmic}
\end{algorithm}

\vspace{1mm} 

We prove the following result for QuantileSCRK, a simplified version of Theorem~\ref{thm:quantilescrk_convergence} that we defer the precise statement of to Section~\ref{sec:corruptions_analysis}. It shows that for unstructured matrices modelled by continuous ``Gaussian-like'' random matrices, the QuantileSCRK method robustly and efficiently converges, provided that there is enough external knowledge (in terms of $m_0$) to make the \emph{effective aspect ratio} $(m - m_0) / (n - m_0)$ large enough, and the proportion of corrupted measurements, $\beta \coloneqq |\mathcal{C}| / (m - m_0)$, is not too large:

\begin{theorem}[Simplified version of Theorem~\ref{thm:quantilescrk_convergence}] \label{thm:quantilescrk_convergence_simplified}
Assuming that $\mathbf{A}$ is a continuous ``Gaussian-like'' random matrix, there exist positive constants $R \geq 1$, $\beta_0 < 1$, $c_1$ and $c_2$, which are independent of $m$ and $n$, such that if $(m - m_0) / (n - m_0) \geq R$ and $\beta \leq \beta_0$, then with probability at least $1 - e^{-c_1 (m - m_0)}$ over the randomness in $\mathbf{A}$, the QuantileSCRK iterates $\mathbf{x}^k$ from Algorithm~\ref{alg:quantilescrk_method} converge to the solution $\mathbf{x}^*$ with
\begin{equation} \label{eq:quantilescrk_convergence_simplified_2}
    \mathbb{E} \norm{\mathbf{x}^k - \mathbf{x}^*}^2 \leq \left( 1 - \frac{c_2}{n - m_0} \right)^k \cdot \norm{\mathbf{x}^0 - \mathbf{x}^*}^2 .
\end{equation}
\end{theorem}

Since we are interested in large-scale systems with $m, n \gg 1$, the values of the constants $c_1$ and $c_2$ are dominated by $m$ and $n$ (e.g.\ the probability guarantee is exponentially close to one for large $m$).
Note that this result applies to almost-square matrices with $m = (1 + o(1)) n$ rows provided $m_0$ is big enough, which lies outside the scope of existing QuantileRK theory.
Experimentally, we found that the QuantileSCRK method works well for more general data models than described by the theory, such as when $\mathbf{A}$ is a structured sparse matrix in an image reconstruction problem (see Section~\ref{subsec:experiments_ct}).

\begin{remark}
\begin{enumerate}[label=(\roman*), leftmargin=*]
    \item (Rejection sampling).
    To avoid recomputing the normalizing constant $Z_J = \sum_{j \in J} \norm{\mathbf{P} \mathbf{a}_j}^2$ in every iteration of Algorithm~\ref{alg:quantilescrk_method}
    for sampling a row from the admissible set $J$, which depends on $\mathbf{x}^k$, rejection sampling (as originally proposed in~\cite{HaNeReSw2022}) can be used: i.e.\ in each iteration, a row $j \in I_1$ is sampled with probability $\norm{\mathbf{P} \mathbf{a}_j} / \norm{\mathbf{A}_{I_1} \mathbf{P}}_F^2$, and the projection is made if and only if $|b_j - \mathbf{a}_j^{\tran} \mathbf{x}^k| \leq \gamma_q$.

    \item (Uniform sampling).
    It is computationally more efficient to sample rows $\mathbf{a}_j$ uniformly at random from $I_1$, together with rejection sampling, in Algorithm~\ref{alg:quantilescrk_method}. By using the threshold $\widetilde{\gamma}_q = q\text{-quantile}\left\{ |b_j - \mathbf{a}_j^{\tran} \mathbf{x}^k| / \norm{\mathbf{P} \mathbf{a}_j} ,\, j \in I_1 \right\}$, which has been modified to capture the heterogeneity of the projected row norms, instead of $\gamma_q$, it can be shown that analogues of our results (e.g.~Theorem~\ref{thm:quantilescrk_convergence} and Lemma~\ref{lem:quantilescrk_convergence_condition}) still hold, except that the relevant spectral quantities come from the matrix $\mathbf{D} \mathbf{A}_{I_1} \mathbf{P}$, where $\mathbf{D}$ is the diagonal matrix with entries $\norm{\mathbf{P} \mathbf{a}_j}^{-1}$, $j \in I_1$, instead of $\mathbf{A}_{I_1} \mathbf{P}$.
\end{enumerate}
\end{remark}

\subsection{Organization} \label{subsec:organization}

Section~\ref{sec:related_work} discusses related works. Section~\ref{sec:scrk_analysis} analyzes the SCRK algorithm: we prove the convergence result (Theorem~\ref{thm:scrk_convergence}) in Section~\ref{subsec:update_formula} and generalize it to the noisy setting in Section~\ref{sec:noisy_case} (Theorem~\ref{thm:noisy_scrk_convergence}). We provide several results on using the SCRK method to exploit low-rank structure and geometric properties of $\mathbf{A}$ in Section~\ref{subsec:scrk_structure_discussion}. Furthermore, we show that the subspace constraint acts as a form of dimension reduction when $\mathbf{A}$ is a Gaussian-like random matrix in Section~\ref{sec:random_matrix_analysis}. Section~\ref{sec:corruptions_analysis} analyzes the QuantileSCRK algorithm for solving corrupted linear systems. We provide various numerical experiments in Section~\ref{sec:experiments} to complement our theoretical results, and conclude in Section~\ref{sec:conclusion}.

\section{Related works} \label{sec:related_work}

\paragraph{Kaczmarz-type methods}
Kaczmarz-type algorithms are related to a variety of modern algorithms in (randomized) numerical linear algebra and (stochastic) optimization.
Randomized Kaczmarz (RK) can be viewed as an instance of the stochastic gradient descent (SGD) algorithm with a particular step size, which, based on this connection, has led to new insights into both methods, such as highlighting the role of weighted sampling for SGD~\cite{NeSrWa2016}.
Furthermore, the RK method is one of the basic representatives of the sketch-and-project method~\cite{GowerRichtarik2015}, which provides a unified framework for iteratively solving linear systems -- including the randomized coordinate descent method, related block variants, and the randomized Newton method -- and can also be directly extended to non-linear optimization problems~\cite{GKLR2019}.

Recently, methods based on the Kaczmarz algorithm have been used in the design of more sophisticated linear solvers.
A GMRES-type solver preconditioned by randomized and greedy Kaczmarz inner-iterations is studied in~\cite{du2021kaczmarz}.
Within the sketch-and-project framework, a randomized block Kaczmarz algorithm that uses the preconditioned conjugate gradient method to perform an inexact projection in each iteration is analyzed in~\cite{derezinski2023solving}, and the method is proven to be especially computationally efficient when the matrix $\mathbf{A}$ has a flat-tailed spectrum.
An iterative solver that further combines these ideas with momentum and sparse sketching matrices is analyzed in~\cite{derezinski2024fine}.
In this paper, we focus on the randomized Kaczmarz algorithm for solving systems of linear equations.

\paragraph{Randomized Kaczmarz}
The analysis of the RK algorithm in~\cite{StrohmerVershynin2009} spurred many developments and variants, including randomized block Kaczmarz methods~\cite{NeedellTropp2014, Necoara2019} and Kaczmarz-Motzkin methods that combine sampling and greedy row selection~\cite{LoHaNe2017, HaddockMa2021}.
It is shown that the Kaczmarz method can be extended to solve least squares problems in~\cite{ZouziasFreris2013, NeZhZo2015, MaNeRa2015}, and systems of linear inequalities in~\cite{LeventhalLewis2010, BriskmanNeedell2015}.
The duality between RK and randomized coordinate descent (also called randomized Gauss-Seidel) has motivated a unified description of the two methods and their extended versions~\cite{MaNeRa2015}.
By using ideas from optimization, RK methods with varying step sizes~\cite{NeSrWa2016}, acceleration~\cite{LiuWright2016}, and that promote sparsity~\cite{LorenzEtAl2014, SchopferLorenz2019, SchopferEtAl2022} have also been studied.

The SCRK method resembles the block Kaczmarz method studied in~\cite{NeedellTropp2014}: the difference is that the blocks differ by one row between iterations and thus the update simplifies so that computing new pseudoinverses is not required. Therefore, the SCRK method can offer a similar advantage from using blocks in an efficient manner if a ``good'' block $\mathbf{A}_{I_0}$ can be found (see Section~\ref{subsec:scrk_structure_discussion} for a discussion of what a good block is, and how one might be found). In~\cite{NeedellWard2013, Wu2022}, a two-subspace Kaczmarz method that iteratively projects onto the solution space associated with two rows is shown to significantly outperform the RK method when the system has correlated rows. The SCRK method can be considered as an extension of this idea to higher dimensional subspaces (see Remark~\ref{rmk:scrk_block_update_general}).

The idea of constraining the iterates of the RK algorithm has also implicitly appeared in the design of a fast solver for Laplacian systems in the theoretical computer science literature in~\cite{KOSZ2013}, where the row spaces of the blocks $\mathbf{A}_{I_0}$ and $\mathbf{A}_{I_1}$ are orthogonal by construction and hence Corollary~\ref{cor:scrk_convergence_orthog} applies. The SCRK method offers a general framework for analyzing convergence when $\mathbf{A}_{I_0}$ and $\mathbf{A}_{I_1}$ are not orthogonal.
Randomized sketch descent methods for solving optimization problems subject to linear constraints are also studied in~\cite{NecoaraTakac2021}, where in each iteration the coordinate space (corresponding to $\mathbf{x}$) is randomly sketched for dimensionality reduction and (random) projection matrices, analogous to $\mathbf{P}$, enforce the linear constraints.

\paragraph{Corrupted linear systems}
The literature on solving linear systems with arbitrary sparse corruptions is abundant: see, e.g.,~\cite{amaldi1995complexity, amaldi2005randomized, dalalyan2019outlier, candes2005error, StKuPoBo2012}. Such problems are often tackled within the compressed sensing and robust statistics literature using methods based on linear or SDP relaxations. The closest line of work to our approach is on iterative, row-action, corruption-avoiding algorithms. The first Kaczmarz-type method was proposed in~\cite{HaddockNeedell2019}, which introduced the idea that large residuals should be indicative of corrupted equations, but makes strong restrictions on the number of corrupted measurements (scaling sublinearly with $m$).
The QuantileRK method, introduced in~\cite{HaNeReSw2022}, utilizes quantile-based steps based on this residual heuristic. An important bottleneck of this method is that the linear system generally needs to be sufficiently overdetermined (i.e.\ $m \geq Cn$ for a large constant $C$) to guarantee convergence. We show that with enough external knowledge (i.e.\ $m_0$ large enough), the QuantileSCRK method works even for almost-square systems. Other works studying the QuantileRK method include~\cite{Steinerberger2023, JarmanNeedell2021, ChJaNeRe2022}; in particular, we adapt a deterministic sufficient condition for convergence from~\cite{Steinerberger2023} for QuantileSCRK (see Lemma~\ref{lem:quantilescrk_convergence_condition}).
Another Kaczmarz-type method based on obtaining sparse least squares solutions is analyzed in~\cite{SchopferEtAl2022} and demonstrated to be able to solve linear systems corrupted by impulsive noise.

\section{Analysis of the SCRK method} \label{sec:scrk_analysis}

In this section, we provide theoretical analysis of the SCRK method (Algorithm~\ref{alg:scrk_method}). Recall that $\mathbf{P} = \mathbf{I} - \mathbf{A}_{I_0}^{\dagger} \mathbf{A}_{I_0}$ is the orthogonal projector onto $\Null(\mathbf{A}_{I_0})$, which is equal to $\Range(\mathbf{A}_{I_0}^{\tran})^{\perp}$, the orthogonal complement of the row space of $\mathbf{A}_{I_0}$.

\subsection{Simplified SCRK update formula and proof of Theorem~\ref{thm:scrk_convergence}} \label{subsec:update_formula}

First, we provide a proof of how the block update~\eqref{eq:scrk_block_update} simplifies to the more interpretable and computationally efficient formula~\eqref{eq:scrk_update_main}.

\begin{lemma} \label{lem:scrk_update_formula}
Let $\mathbf{x}^{k+1} = \mathbf{x}^k + \mathbf{A}_{I_0 \cup \{j \}}^{\dagger} (\mathbf{b}_{I_0 \cup \{ j \}} - \mathbf{A}_{I_0 \cup \{ j \}} \mathbf{x}^k)$ and $\mathbf{P} = \mathbf{I} - \mathbf{A}_{I_0}^{\dagger} \mathbf{A}_{I_0}$. If $\mathbf{x}^k$ solves $\mathbf{A}_{I_0} \mathbf{x}^k = \mathbf{b}_{I_0}$ and $\mathbf{P} \mathbf{a}_j \ne \mathbf{0}$, then
\[
    \mathbf{x}^{k+1} = \mathbf{x}^k + \frac{b_j - \mathbf{a}_j^{\tran} \mathbf{x}^k}{\norm{\mathbf{P} \mathbf{a}_j}} \cdot \frac{\mathbf{P} \mathbf{a}_j}{{\norm{\mathbf{P} \mathbf{a}_j}}} .
\]
\end{lemma}

\begin{proof}
We may assume $b_j - \mathbf{a}_j^{\tran} \mathbf{x}^k \ne 0$, otherwise $\mathbf{x}^{k+1} = \mathbf{x}^k$. Since $\mathbf{x}^{k+1}$ is the orthogonal projection of $\mathbf{x}^k$ onto the solution space $\mathbf{A}_{I_0 \cup \{ j \}} \mathbf{x} = \mathbf{b}_{I_0 \cup \{ j \}}$, the increment $\mathbf{z} \coloneqq \mathbf{x}^{k+1} - \mathbf{x}^k$ is the solution of the following optimization problem:
\begin{equation} \label{lem:scrk_update_formula_pf1}
    \min\limits_{\mathbf{z} \in \reals^n} \frac{1}{2} \norm{\mathbf{z}}^2 \quad \text{subject to} \quad
    \mathbf{A}_{I_0} (\mathbf{x}^{k} + \mathbf{z}) = \mathbf{b}_{I_0} ,\;
    \mathbf{a}_j^{\tran} (\mathbf{x}^k + \mathbf{z}) = b_j .
\end{equation}
This can be solved by introducing the Lagrange multipliers $\boldsymbol{\lambda} \in \reals^n$ and $\tau \in \reals$ for the two constraints. Since $\mathbf{A}_{I_0} \mathbf{x}^k = \mathbf{b}_{I_0}$, the first constraint is equivalent to $\mathbf{A}_{I_0} \mathbf{z} = \mathbf{0}$, and thus $\mathbf{z}$ solves
\begin{equation} \label{lem:scrk_update_formula_pf2}
    \mathbf{z} + \mathbf{A}_{I_0}^{\tran} \boldsymbol{\lambda} + \tau \mathbf{a}_j = \mathbf{0}
\end{equation}
whilst satisfying $\mathbf{A}_{I_0} \mathbf{z} = \mathbf{0}$ and $\mathbf{a}_j^{\tran} \mathbf{z} = b_j - \mathbf{a}_j^{\tran} \mathbf{x}^k$. Since $\mathbf{P}$ is the orthogonal projector onto $\Null(\mathbf{A}_{I_0})$, $\mathbf{P} \mathbf{A}_{I_0}^{\tran} = \mathbf{0}$ and the first constraint is equivalent to $\mathbf{P} \mathbf{z} = \mathbf{z}$. Hence, pre-multiplying~\eqref{lem:scrk_update_formula_pf2} by $\mathbf{P}$ implies that $\mathbf{z} = -\tau \mathbf{P} \mathbf{a}_j$. Furthermore, pre-multiplying~\eqref{lem:scrk_update_formula_pf2} by $\mathbf{z}^{\tran}$ and using the constraints implies that $\tau = -\norm{\mathbf{z}}^2 / (b_j - \mathbf{a}_j^{\tran} \mathbf{x}^k) = -\tau^2 \norm{\mathbf{P} \mathbf{a}_j}^2 / (b_j - \mathbf{a}_j^{\tran} \mathbf{x}^k)$.
Solving for $\tau$ yields $\tau = (b_j - \mathbf{a}_j^{\tran} \mathbf{x}^k) / \norm{\mathbf{P} \mathbf{a}_j}^2$,
which completes the proof.
\end{proof}

\begin{remark}
From the optimization formulation~\eqref{lem:scrk_update_formula_pf1}, it can also be shown that the unit direction $\mathbf{P} \mathbf{a}_j / \norm{\mathbf{P} \mathbf{a}_j}$ taken from $\mathbf{x}^k$ to reach $\mathbf{x}^{k+1}$ maximizes $|\mathbf{a}_j^{\tran} \tilde{\mathbf{z}}|^2$ over all unit vectors $\tilde{\mathbf{z}} \in \Null(\mathbf{A}_{I_0})$. This provides a nice geometric interpretation of the SCRK update: \emph{the direction $\mathbf{P} \mathbf{a}_j$ taken to reach the solution space $\mathbf{A}_{I_0 \cup \{j\}} \mathbf{x} = \mathbf{b}_{I_0 \cup \{j\}}$ minimizes the angle from the optimal direction $\mathbf{a}_j$ for reaching the solution space $\mathbf{a}_j^{\tran} \mathbf{x} = b_j$ within the subspace $\Null(\mathbf{A}_{I_0})$}; see Figure~\ref{fig:subspacerk_picture} for an illustration.
For an alternative algebraic proof of a more general version of Lemma~\ref{lem:scrk_update_formula}, see Remark~\ref{rmk:scrk_block_update_general} later.
\end{remark}

\begin{figure}[!htb]
    \centering
    \includegraphics[width=0.55\textwidth]{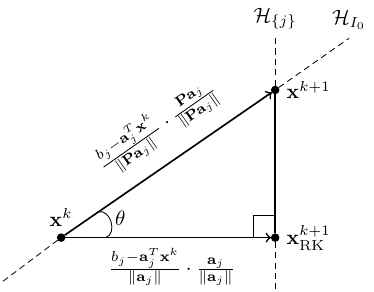}
    \caption{SCRK update from the current iterate $\mathbf{x}^k$ for reaching the vector $\mathbf{x}^{k+1}$ in the solution space $\mathcal{H}_{\{j\}} = \{ \mathbf{x} \in \reals^n : \mathbf{a}_j^{\tran} \mathbf{x} = b_j \}$ whilst remaining within
    $\mathcal{H}_{I_0} = \{ \mathbf{x} \in \reals^n : \mathbf{A}_{I_0} \mathbf{x} = \mathbf{b}_{I_0} \}$, compared to the RK update for reaching $\mathbf{x}^{k+1}_{\mathrm{RK}}$ alone.}
    \label{fig:subspacerk_picture}
\end{figure}

We will now use the simplified update formula in Lemma~\ref{lem:scrk_update_formula} to prove Theorem~\ref{thm:scrk_convergence}.

\begin{proof}[Proof of Theorem~\ref{thm:scrk_convergence}]
Consider the $k$\textsuperscript{th} iterate $\mathbf{x}^k$. Suppose that $\mathbf{a}_j$ is sampled in the next iteration (with $\mathbf{P} \mathbf{a}_j \ne \mathbf{0}$).
By subtracting $\mathbf{x}^*$ from both sides of~\eqref{eq:scrk_update_main} and noting that $\mathbf{a}_j^{\tran}(\mathbf{x}^k - \mathbf{x}^*) = \mathbf{a}_j^{\tran}\mathbf{P}(\mathbf{x}^k - \mathbf{x}^*)$ for any $j \in I_1$ since $\mathbf{x}^k - \mathbf{x}^* \in \Null(\mathbf{A}_{I_0})$, we have
\[
    \mathbf{x}^{k+1} - \mathbf{x}^* = \left( \mathbf{I} - \mathbf{v} \mathbf{v}^{\tran} \right) (\mathbf{x}^k - \mathbf{x}^*)
    \quad \text{for } \mathbf{v} \coloneqq \frac{\mathbf{P} \mathbf{a}_j}{\norm{\mathbf{P} \mathbf{a}_j}} .
\]
Since $\mathbf{v} \mathbf{v}^{\tran}$ is an orthogonal projector, $(\mathbf{x}^{k+1} - \mathbf{x}^*) \perp \mathbf{v} \mathbf{v}^{\tran}(\mathbf{x}^{k} - \mathbf{x}^*)$. Thus, by Pythagoras' theorem,
\[
    \norm{\mathbf{x}^{k+1} - \mathbf{x}^*}^2 = \norm{\mathbf{x}^k - \mathbf{x}^*}^2 - \| \mathbf{v} \mathbf{v}^{\tran} (\mathbf{x}^{k} - \mathbf{x}^*) \|^2 = \norm{\mathbf{x}^k - \mathbf{x}^*}^2 - \left|  \mathbf{v}^{\tran} (\mathbf{x}^k - \mathbf{x}^*) \right|^2 .
\]
By taking expectation (where each row $\mathbf{a}_j$ is sampled with probability $\norm{\mathbf{P} \mathbf{a}_j}^2 / \norm{\mathbf{A}_{I_1} \mathbf{P}}_F^2$), conditional on all the choices up to the $k$\textsuperscript{th} iteration, we obtain
\begin{align}
    \mathbb{E}_k \norm{\mathbf{x}^{k+1} - \mathbf{x}^*}^2
    &= \norm{\mathbf{x}^k - \mathbf{x}^*}^2 - \sum_{j \in I_1: \mathbf{P} \mathbf{a}_j \ne \mathbf{0}} \frac{\norm{\mathbf{P} \mathbf{a}_j}^2}{\norm{\mathbf{A}_{I_1} \mathbf{P}}_F^2} \cdot \left| \left( \frac{\mathbf{P} \mathbf{a}_j}{\norm{\mathbf{P} \mathbf{a}_j}} \right)^{\tran} (\mathbf{x}^k - \mathbf{x}^*) \right|^2 \nonumber\\
    &= \norm{\mathbf{x}^k - \mathbf{x}^*}^2 - \frac{1}{\norm{\mathbf{A}_{I_1} \mathbf{P}}_F^2} \norm{\mathbf{A}_{I_1} \mathbf{P} (\mathbf{x}^k - \mathbf{x}^*)}^2 \nonumber\\
    &= \left( 1 - \frac{\theta^2}{\norm{\mathbf{A}_{I_1} \mathbf{P}}_F^2} \right) \cdot \norm{\mathbf{x}^k - \mathbf{x}^*}^2 , \quad \text{where } \theta \coloneqq \normbig{\mathbf{A}_{I_1} \mathbf{P} \left( \frac{\mathbf{x}^k - \mathbf{x}^*}{\norm{\mathbf{x}^k - \mathbf{x}^*}}\right)} . \nonumber 
\end{align}

The next step is to estimate $\theta$ from below, which requires more care than a similar estimate used to prove convergence of the RK method~\cite{StrohmerVershynin2009} since $\mathbf{A}_{I_1} \mathbf{P}$ has a nontrivial nullspace.
A similar case where the system matrix has a nontrivial nullspace was also treated in~\cite{LorenzEtAl2018}.
First, observe that $\Null(\mathbf{A}_{I_1} \mathbf{P}) = \Null(\mathbf{P})$. Indeed, the nontrivial inclusion $\Null(\mathbf{A}_{I_1} \mathbf{P}) \subseteq \Null(\mathbf{P})$ follows from the observation that $\mathbf{A}_{I_1} \mathbf{P} \mathbf{y} = \mathbf{0}$ implies that $\mathbf{P} \mathbf{y} \in \Null(\mathbf{A}_{I_0}) \cap \Null(\mathbf{A}_{I_1}) = \Null(\mathbf{A}) = \{ \mathbf{0} \}$, since $\mathbf{A}$ has full rank. Therefore, since $\mathbf{x}^k - \mathbf{x}^* \in \Null(\mathbf{A}_{I_0})$ is orthogonal to $\Null(\mathbf{P}) = \Null(\mathbf{A}_{I_1}\mathbf{P})$,
\begin{equation} \label{eq:scrk_convergence_pf1}
    \theta = \normbig{\mathbf{A}_{I_1} \mathbf{P} \left( \frac{\mathbf{x}^k - \mathbf{x}^*}{\norm{\mathbf{x}^k - \mathbf{x}^*}} \right)}^2
    \geq \min_{\substack{\mathbf{z} \in \Null(\mathbf{A}_{I_0}) \\ \norm{\mathbf{z}} = 1}} \norm{\mathbf{A}_{I_1} \mathbf{z}}^2
    = \sigma_{\mathrm{min}}^+(\mathbf{A}_{I_1} \mathbf{P})^2 .
\end{equation}
This implies the following bound for one step of the SCRK method:
\begin{equation*}
    \mathbb{E}_k \norm{\mathbf{x}^{k+1} - \mathbf{x}^*}^2 \leq \left( 1 - \frac{\sigma_{\mathrm{min}}^+(\mathbf{A}_{I_1} \mathbf{P})^2}{\norm{\mathbf{A}_{I_1} \mathbf{P}}_F^2} \right) \cdot \norm{\mathbf{x}^k - \mathbf{x}^*}^2 .
\end{equation*}
By iterating and taking the full expectation, this concludes the proof of Theorem~\ref{thm:scrk_convergence}.
\end{proof}

\begin{remark} \label{rmk:singular_vector_directions}
If $\mathbf{v}_\ell$ is a right singular vector of $\mathbf{A}_{I_1} \mathbf{P}$ corresponding to the $\ell$\textsuperscript{th} largest singular value $\sigma_{\ell}(\mathbf{A}_{I_1} \mathbf{P})$, it can be shown that
\[
    \mathbb{E} \innprod{\mathbf{x}^{k} - \mathbf{x}^*}{\mathbf{v}_{\ell}}
    = \left( 1 - \frac{\sigma_{\ell}(\mathbf{A}_{I_1} \mathbf{P})^2}{\norm{\mathbf{A}_{I_1} \mathbf{P}}_F^2} \right)^k \cdot \innprod{\mathbf{x}^0 - \mathbf{x}^*}{\mathbf{v}_{\ell}} .
\]
This shows that the residual vector $\mathbf{x}^k - \mathbf{x}^*$ decays fastest in the directions corresponding to the largest singular values of $\mathbf{A}_{I_1}$ restricted to $\Null(\mathbf{A}_{I_0})$. This phenomenon was proved by Steinerberger~\cite{Steinerberger2021} for the RK method.
\end{remark}

\subsection{SCRK convergence on inconsistent linear systems} \label{sec:noisy_case}

In the general case, the measurement vector $\mathbf{b}$ might not be known exactly, but only accessible through a set of noisy observations $\widehat{\mathbf{b}} \coloneqq \mathbf{b} + \mathbf{r}$, where $\mathbf{r}$ is an arbitrary error vector (which is considered to be small). Similar to previous analyses of Kaczmarz methods~\cite{Needell2010, NeedellTropp2014, rebrova2021block}, we prove that if the SCRK method is used with the noisy measurements $\widehat{\mathbf{b}}$, then the iterates converge to the solution $\mathbf{x}^*$ up to an error horizon:

\begin{theorem} \label{thm:noisy_scrk_convergence}
Suppose that the rows of $\mathbf{A}$ are partitioned into two blocks $\mathbf{A}_{I_0}$ and $\mathbf{A}_{I_1}$ of sizes $m_0$ and $m - m_0$ respectively, and assume that $\mathbf{A}_{I_0}$ has full row rank.
If $\widehat{\mathbf{x}}^k$ denotes the sequence of SCRK iterates from Algorithm~\ref{alg:scrk_method} where the noisy measurement vector $\widehat{\mathbf{b}} = \mathbf{b} + \mathbf{r}$ is used in place of $\mathbf{b}$, and the initial iterate $\widehat{\mathbf{x}}^0$ solves $\mathbf{A}_{I_0} \widehat{\mathbf{x}}^0 = \widehat{\mathbf{b}}_{I_0}$, then
\begin{equation*}
    \mathbb{E} \norm{\widehat{\mathbf{x}}^k - \mathbf{x}^*}^2 \leq \left( 1 - \frac{\sigma_{\mathrm{min}}^+(\mathbf{A}_{I_1} \mathbf{P})^2}{\norm{\mathbf{A}_{I_1} \mathbf{P}}_F^2} \right)^k \cdot \norm{\widehat{\mathbf{x}}^{0} - \mathbf{x}^*}^2 + \gamma_0 + \gamma_1 ,
\end{equation*}
where $\gamma_0, \gamma_1 \geq 0$ are given by
\begin{equation*}
    \gamma_0 = \frac{2 \norm{\mathbf{r}_{I_0}}^2}{\sigma_{\mathrm{min}}(\mathbf{A}_{I_0})^2} - \norm{\mathbf{A}_{I_0}^{\dagger} \mathbf{r}_{I_0}}^2
    \quad \text{and} \quad
    \gamma_1 = \frac{\norm{\mathbf{r}_{I_1} - \mathbf{A}_{I_1} \mathbf{A}_{I_0}^{\dagger} \mathbf{r}_{I_0}}^2}{\sigma_{\mathrm{min}}^+(\mathbf{A}_{I_1} \mathbf{P})^2} .
\end{equation*}
\end{theorem}

\begin{remark}[Error horizon]
Note that $\gamma_0$ only depends on the noise in the measurements corresponding to the fixed block $I_0$. In particular, if $\mathbf{r}_{I_0} = 0$, then $\gamma_0 = 0$ and $\gamma_1$ only depends on $\norm{\mathbf{r}_{I_1}}^2$. On the other hand, if $\mathbf{r}_{I_1} = 0$, then $\gamma_1$ only depends on $\norm{\mathbf{A}_{I_1} \mathbf{A}_{I_0}^{\dagger} \mathbf{r}_{I_0}}^2 = \sum_{j \in I_1} |\mathbf{a}_j^{\tran} \mathbf{A}_{I_0}^{\dagger} \mathbf{r}_{I_0}|^2$. Note that $|\mathbf{a}_j^{\tran} \mathbf{A}_{I_0}^{\dagger} \mathbf{r}_{I_0}|$ corresponds to the angle between the row $\mathbf{a}_j$ and $\mathbf{A}_{I_0}^{\dagger} \mathbf{r}_{I_0}$; the vector $\mathbf{A}_{I_0}^{\dagger} \mathbf{r}_{I_0}$ accounts for how noise in the measurements corresponding to the fixed block $\mathbf{A}_{I_0}$ shifts the solution space (see Lemma~\ref{lem:noisy_solution_space_geometry}\ref{lem:noisy_solution_space_geometry_a}). Finally, if $\mathbf{r} = 0$, then $\gamma_0 = \gamma_1 = 0$ and we recover Theorem~\ref{thm:scrk_convergence}.
\end{remark}


\begin{remark}[Least squares]
Given a set of noisy measurements $\widehat{\mathbf{b}}$, our setup can easily be translated to the problem of solving the \emph{inconsistent} system of linear equations $\mathbf{A} \mathbf{x} \approx \widehat{\mathbf{b}}$ in a least squares sense. In this setting, by defining $\mathbf{x}^*$ to be the least squares solution (i.e.\ $\mathbf{x}^* \coloneqq \argmin_{\mathbf{x} \in \reals^n} \norm{\mathbf{A} \mathbf{x} - \widehat{\mathbf{b}}}^2 = \mathbf{A}^{\dagger} \widehat{\mathbf{b}}$), and setting $\mathbf{b} \coloneqq \mathbf{A} \mathbf{x}^*$ and $\mathbf{r} \coloneqq \widehat{\mathbf{b}} - \mathbf{A} \mathbf{x}^*$, Theorem~\ref{thm:noisy_scrk_convergence} can be applied to deduce that the SCRK iterates converge to the least squares solution up to the same error horizon.
\end{remark}

We develop some technical results before proving Theorem~\ref{thm:noisy_scrk_convergence}.
Note that with noisy measurements $\widehat{\mathbf{b}}$, the relationship $\widehat{\mathbf{x}}^k - \mathbf{x}^* \in \Null(\mathbf{A}_{I_0})$ does not necessarily hold anymore.
Thus, it will be more convenient to work directly with the block update~\eqref{eq:scrk_block_update} instead. First, we present a decomposition of the pseudoinverse $\mathbf{A}_{I \cup J}^{\dagger}$ in terms of $\mathbf{A}_{I}$ and $\mathbf{A}_J$.

\begin{lemma} \label{lem:pseudoinverse_block_formula}
Let $\mathbf{A} \in \reals^{m \times n}$ and $I, J \subseteq [m]$ be two disjoint subsets of row indices. If $\mathbf{A}_{I \cup J}$ has full row rank, then the pseudoinverse $\mathbf{A}_{I \cup J}^{\dagger}$ admits the block representation
\begin{equation} \label{eq:pseudoinverse_block_formula_1}
    \mathbf{A}_{I \cup J}^{\dagger} =
    \begin{pmatrix}
        \mathbf{A}_{I}^{\dagger} - (\mathbf{A}_J \mathbf{P})^{\dagger} \mathbf{A}_J \mathbf{A}_{I}^{\dagger}
        & \rvline &
        (\mathbf{A}_J \mathbf{P})^{\dagger}
    \end{pmatrix} ,
\end{equation}
where $\mathbf{P} = \mathbf{I} - \mathbf{A}_{I}^{\dagger} \mathbf{A}_{I}$ is the orthogonal projection operator onto $\Null(\mathbf{A}_{I})$.
\end{lemma}

\begin{proof}[Proof of Lemma~\ref{lem:pseudoinverse_block_formula}]
First, we record the key algebraic property that will be used repeatedly:
\begin{equation} \label{eq:full-rank-pinv}
    \mathbf{X}^{\dagger} = \mathbf{X}^{\tran} (\mathbf{X}\mathbf{X}^{\tran})^{-1}
    \quad \text{for } \mathbf{X} = \mathbf{A}_{I \cup J} ,\, \mathbf{A}_{I} ,\, \text{or } \mathbf{A}_J \mathbf{P} .
\end{equation}
This follows since all three matrices have full row rank: $\mathbf{A}_{I \cup J}$ by assumption, $\mathbf{A}_{I}$ as its row subset, and $\mathbf{A}_J \mathbf{P}$ from the observation that if its rows were linearly dependent then there would exist some nonzero $\boldsymbol{\alpha} \in \reals^{|J|}$ such that $\sum_{j \in J} \alpha_j \mathbf{a}_j^{\tran} \mathbf{P} = \mathbf{0}$, which would imply that $\sum_{j \in J} \alpha_j \mathbf{a}_j \in \Null(\mathbf{P}) = \Range(\mathbf{A}_I^{\tran})$ and thus contradict the assumption that $\mathbf{A}_{I \cup J}$ has full row rank. We have
\[
    \mathbf{A}_{I \cup J} \mathbf{A}_{I \cup J}^{\tran}
    = \left( \begin{array}{ccc}
        \mathbf{A}_{I} \mathbf{A}_{I}^{\tran} & \rvline & \mathbf{A}_{I} \mathbf{A}_J^{\tran} \Bstrut \\
        \hline
        (\mathbf{A}_{I} \mathbf{A}_J^{\tran})^{\tran} & \rvline & \mathbf{A}_J \mathbf{A}_J^{\tran} \Tstrut \\
    \end{array} \right)
    \eqqcolon \left( \begin{array}{ccc}
        \mathbf{A}_{II} & \rvline & \mathbf{A}_{IJ} \Bstrut \\
        \hline
        \mathbf{A}_{IJ}^{\tran} & \rvline & \mathbf{A}_{JJ} \Tstrut \\
    \end{array} \right) .
\]
Since $\mathbf{A}_I$ has full row rank, $\mathbf{A}_{II}$ is invertible and
\begin{equation} \label{eq:schur}
\begin{aligned}
    \mathbf{A}_{I \cup J}^{\dagger} &\overset{\eqref{eq:full-rank-pinv}}{=} \mathbf{A}_{I \cup J}^{\tran} \left( \mathbf{A}_{I \cup J} \mathbf{A}_{I \cup J}^{\tran} \right)^{-1} \\
    &\overset{\phantom{\eqref{eq:full-rank-pinv}}}{=}
    \begin{pmatrix}
        \mathbf{A}_{I}^{\tran} & \rvline & \mathbf{A}_J^{\tran}
    \end{pmatrix}
    \left( \begin{array}{ccc}
        \mathbf{A}_{II}^{-1} + \mathbf{A}_{II}^{-1}\mathbf{A}_{IJ}\mathbf{R}^{-1}\mathbf{A}_{IJ}^{\tran} \mathbf{A}_{II}^{-1} & \rvline & -\mathbf{A}_{II}^{-1} \mathbf{A}_{IJ} \mathbf{R}^{-1} \Bstrut \\
        \hline
        - \mathbf{R}^{-1}\mathbf{A}_{IJ}^{\tran} \mathbf{A}_{II}^{-1} & \rvline & \mathbf{R}^{-1} \Tstrut \\
    \end{array} \right) ,
\end{aligned}
\end{equation}
where $\mathbf{R} \coloneqq \mathbf{A}_{JJ} - \mathbf{A}_{IJ}^{\tran} \mathbf{A}_{II}^{-1} \mathbf{A}_{IJ}$ is the Schur complement~\cite{HornJohnson2012} of the block $\mathbf{A}_{II}$. Recall that $\mathbf{P} = \mathbf{I} - \mathbf{A}_{I}^{\dagger} \mathbf{A}_{I}$, and so
\begin{equation*}
    \mathbf{R}
    = \mathbf{A}_J \mathbf{A}_J^{\tran} - \mathbf{A}_J \mathbf{A}_{I}^{\tran} (\mathbf{A}_I \mathbf{A}_I^{\tran})^{-1} \mathbf{A}_{I} \mathbf{A}_{J}^{\tran}
    \overset{\eqref{eq:full-rank-pinv}}{=} \mathbf{A}_J [\mathbf{I} - \mathbf{A}_{I}^{\dagger} \mathbf{A}_{I}] \mathbf{A}_J^{\tran}
    = \mathbf{A}_J \mathbf{P} \mathbf{A}_J^{\tran} ,
\end{equation*}
which implies that
\begin{equation}\label{eq:pseudoinverse_block_formula_pf2}
   \mathbf{P}\mathbf{A}_J^{\tran}\mathbf{R}^{-1} = \mathbf{P}\mathbf{A}_J^{\tran}(\mathbf{A}_J \mathbf{P}\mathbf{A}_J^{\tran})^{-1} = (\mathbf{A}_J \mathbf{P})^{\tran} ((\mathbf{A}_J \mathbf{P})(\mathbf{A}_J \mathbf{P})^{\tran})^{-1} \overset{\eqref{eq:full-rank-pinv}}{=}  (\mathbf{A}_J \mathbf{P})^{\dagger} .
\end{equation}
Next we compute expressions for the two blocks of $\mathbf{A}_{I \cup J}^\dagger$ in~\eqref{eq:schur}. The first block, $\mathbf{A}_{I}^{\tran}[\mathbf{A}_{II}^{-1} + \mathbf{A}_{II}^{-1}\mathbf{A}_{IJ}\mathbf{R}^{-1}\mathbf{A}_{IJ}^{\tran} \mathbf{A}_{II}^{-1} ] + \mathbf{A}_{J}^{\tran}[- \mathbf{R}^{-1}\mathbf{A}_{IJ}^{\tran} \mathbf{A}_{II}^{-1}]$, simplifies to
\begin{align*}
    \mathbf{A}_{I}^{\dagger} - [-\mathbf{A}_{I}^{\dagger} \mathbf{A}_{I} + \mathbf{I}] \mathbf{A}_J^{\tran} \mathbf{R}^{-1} \mathbf{A}_J \mathbf{A}_{I}^{\dagger}
    &= \mathbf{A}_{I}^{\dagger} - \mathbf{P}\mathbf{A}_J^{\tran} \mathbf{R}^{-1}\mathbf{A}_J \mathbf{A}_{I}^{\dagger}
    \overset{\eqref{eq:pseudoinverse_block_formula_pf2}}{=} \mathbf{A}_{I}^{\dagger} - (\mathbf{A}_J \mathbf{P})^{\dagger} \mathbf{A}_J \mathbf{A}_{I}^{\dagger} .
\end{align*}
The second block, $\mathbf{A}_{I}^{\tran}[-\mathbf{A}_{II}^{-1} \mathbf{A}_{IJ} \mathbf{R}^{-1}] + \mathbf{A}_{J}^{\tran}[ \mathbf{R}^{-1}]$, simplifies to
\begin{align*}
    [-\mathbf{A}_{I}^{\dagger} \mathbf{A}_{I} + \mathbf{I}] \mathbf{A}_J^{\tran} \mathbf{R}^{-1}
    = \mathbf{P} \mathbf{A}_J^{\tran} \mathbf{R}^{-1} \overset{\eqref{eq:pseudoinverse_block_formula_pf2}}{=} (\mathbf{A}_J \mathbf{P})^{\dagger} . \label{eq:pseudoinverse_block_formula_pf4}
\end{align*}
Combining the two preceding displayed equations completes the proof.
\end{proof}

Next, we describe how the noise affects the geometry of the solution spaces.

\begin{lemma} \label{lem:noisy_solution_space_geometry}
Denote the true and noisy solution spaces associated with $I \subset [m]$ by
\begin{equation}
    \mathcal{H}_{I} = \{ \mathbf{x} \in \reals^n : \mathbf{A}_{I} \mathbf{x} = \mathbf{b}_{I} \} \quad \text{ and } \quad
    \widehat{\mathcal{H}}_{I} = \{ \mathbf{x} \in \reals^n : \mathbf{A}_{I} \mathbf{x} = \mathbf{b}_{I} + \mathbf{r}_{I} \}
\end{equation}
respectively. If $\mathbf{A}_I$ has full row rank, then $\mathcal{H}_{I}$ and $\widehat{\mathcal{H}}_{I}$ satisfy the following:
\begin{enumerate}[label=\textnormal{(\alph*)}, ref=\textnormal{\alph*}, leftmargin=*]
    \item \label{lem:noisy_solution_space_geometry_a} $\widehat{\mathcal{H}}_{I} = \mathcal{H}_{I} + \mathbf{A}_{I}^{\dagger} \mathbf{r}_{I}$.

    \item \label{lem:noisy_solution_space_geometry_b} $\widehat{\mathcal{H}}_{I} - \mathbf{x}^* = \Null(\mathbf{A}_{I}) + \mathbf{A}_{I}^{\dagger} \mathbf{r}_{I}$.

    \item \label{lem:noisy_solution_space_geometry_c} The vector $\mathbf{A}_{I}^{\dagger} \mathbf{r}_{I}$ is orthogonal to $\Null(\mathbf{A}_{I})$.
\end{enumerate}
\end{lemma}
\begin{proof}
(\ref{lem:noisy_solution_space_geometry_a}): Since $\mathbf{A}_{I}$ has full row rank, $\mathbf{A}_{I} \mathbf{A}_{I}^{\dagger} = \mathbf{I}$. Therefore, for any $\mathbf{x} \in \mathcal{H}_{I}$, we have $\mathbf{A}_{I} (\mathbf{x} + \mathbf{A}_{I}^{\dagger} \mathbf{r}_{I}) = \mathbf{b}_{I} + \mathbf{r}_{I}$ and so $\mathbf{x} + \mathbf{A}_{I}^{\dagger} \mathbf{r}_{I} \in \widehat{\mathcal{H}}_{I}$. Conversely, for any $\widehat{\mathbf{x}} \in \widehat{\mathcal{H}}_{I}$, we have $\mathbf{A}_{I} (\widehat{\mathbf{x}} - \mathbf{A}_{I}^{\dagger} \mathbf{r}_{I}) = \mathbf{b}_{I}$. Thus, $(\widehat{\mathbf{x}} - \mathbf{A}_{I}^{\dagger} \mathbf{r}_{I}) \in \mathcal{H}_{I}$, and so
\[
    \widehat{\mathbf{x}} = (\widehat{\mathbf{x}} - \mathbf{A}_{I}^{\dagger} \mathbf{r}_{I}) + \mathbf{A}_{I}^{\dagger} \mathbf{r}_{I} \in \mathcal{H}_{I} + \mathbf{A}_{I}^{\dagger} \mathbf{r}_{I} .
\]
(\ref{lem:noisy_solution_space_geometry_b}): Since $\mathbf{x}^* \in \mathcal{H}_{I}$, we have $\mathcal{H}_{I} - \mathbf{x}^* = \Null(\mathbf{A}_{I})$. Together with part~(\ref{lem:noisy_solution_space_geometry_a}), this implies~(\ref{lem:noisy_solution_space_geometry_b}). Finally, (\ref{lem:noisy_solution_space_geometry_c}) follows from the fact $\Range(\mathbf{A}_{I}^{\dagger}) = \Range(\mathbf{A}_{I}^{\tran}) = \Null(\mathbf{A}_{I})^{\perp}$.
\end{proof}

We will now prove Theorem~\ref{thm:noisy_scrk_convergence} by using the expression for $\mathbf{A}_{I_0 \cup \{ j \}}^{\dagger}$ given in Lemma~\ref{lem:pseudoinverse_block_formula}, as well as the geometry of the shifted solution spaces described by Lemma~\ref{lem:noisy_solution_space_geometry}.

\begin{proof}[Proof of Theorem~\ref{thm:noisy_scrk_convergence}]
Consider the $k$\textsuperscript{th} iterate $\widehat{\mathbf{x}}^k$. Suppose that $\mathbf{a}_j$ is sampled in the next iteration (with $\mathbf{P} \mathbf{a}_j \ne \mathbf{0}$ and hence $\mathbf{A}_{I_0 \cup \{ j \}}$ has full row rank). Then one step of the SCRK algorithm with noisy measurements corresponds to the projection of $\widehat{\mathbf{x}}^k$ onto the noisy solution space $\widehat{\mathcal{H}}_{I_0 \cup \{ j \}}$, namely,
\[
    \widehat{\mathbf{x}}^{k+1} = \widehat{\mathbf{x}}^k + \mathbf{A}_{I_0 \cup \{ j \}}^{\dagger} (\mathbf{b}_{I_0 \cup \{ j \}} + \mathbf{r}_{I_0 \cup \{ j \}} - \mathbf{A}_{I_0 \cup \{ j \}} \widehat{\mathbf{x}}^k).
\]
We will compare $\widehat{\mathbf{x}}^{k+1}$ with the projection of $\widehat{\mathbf{x}}^k$ onto the true solution space $\mathcal{H}_{I_0 \cup \{ j \}}$, denoted by
\[
    \mathbf{x}^{k+1} \coloneqq \widehat{\mathbf{x}}^k + \mathbf{A}_{I_0 \cup \{ j \}}^{\dagger} (\mathbf{b}_{I_0 \cup \{ j \}} - \mathbf{A}_{I_0 \cup \{ j \}} \widehat{\mathbf{x}}^k).
\]

\textbf{Step 1. Exact computations.}
Note that
\[
    \widehat{\mathbf{x}}^{k+1} - \mathbf{x}^* = (\mathbf{x}^{k+1} - \mathbf{x}^*) + \mathbf{A}_{I_0 \cup \{ j \}}^{\dagger} \mathbf{r}_{I_0 \cup \{ j \}} ,
\]
and $\mathbf{A}_{I_0 \cup \{ j \}}^{\dagger} \mathbf{r}_{I_0 \cup \{ j \}} \perp (\mathbf{x}^{k+1} - \mathbf{x}^*) \in \Null(\mathbf{A}_{I_0 \cup \{ j \}})$ by  Lemma~\ref{lem:noisy_solution_space_geometry}\ref{lem:noisy_solution_space_geometry_c}. By using Pythagoras' theorem twice (and orthogonality of the true Kaczmarz projections), we have
\begin{align} \label{eq:noisy_scrk_exact_pf1}
    \norm{\widehat{\mathbf{x}}^{k+1} - \mathbf{x}^*}^2 &= \norm{\mathbf{x}^{k+1} - \mathbf{x}^*}^2 + \norm{\mathbf{A}_{I_0 \cup \{ j \}}^{\dagger} \mathbf{r}_{I_0 \cup \{ j \}}}^2 \nonumber\\
    &= \norm{\widehat{\mathbf{x}}^k - \mathbf{x}^*}^2 - \norm{\mathbf{A}_{I_0 \cup \{ j \}}^{\dagger} \mathbf{A}_{I_0 \cup \{ j \}} (\widehat{\mathbf{x}}^k - \mathbf{x}^*)}^2  +  \norm{\mathbf{A}_{I_0 \cup \{ j \}}^{\dagger} \mathbf{r}_{I_0 \cup \{ j \}}}^2.
\end{align}
By using Lemma~\ref{lem:pseudoinverse_block_formula} with $I = I_0$ and $J = \{j\}$, we can simplify the last two terms: firstly,
\begin{align}
    \mathbf{A}_{I_0 \cup \{ j \}}^{\dagger} \mathbf{r}_{I_0 \cup \{ j \}}
    &= \begin{pmatrix}
        \mathbf{A}_{I_0}^{\dagger} - \frac{\mathbf{P} \mathbf{a}_j \mathbf{a}_j^{\tran} \mathbf{A}_{I_0}^{\dagger}}{\norm{\mathbf{P} \mathbf{a}_j}^2} & \rvline & \frac{\mathbf{P} \mathbf{a}_j}{\norm{\mathbf{P} \mathbf{a}_j}^2}
    \end{pmatrix}
    \left( \begin{array}{c}
        \mathbf{r}_{I_0} \\
        \hline
        r_j \\
    \end{array} \right) \nonumber \\
    &= \mathbf{A}_{I_0}^{\dagger} \mathbf{r}_{I_0} + (r_j - \mathbf{a}_j^{\tran} \mathbf{A}_{I_0}^{\dagger} \mathbf{r}_{I_0}) \cdot \frac{\mathbf{P} \mathbf{a}_j}{\norm{\mathbf{P} \mathbf{a}_j}^2} . \label{eq:noisy_scrk_exact_pf3}
\end{align}
Next, since $\widehat{\mathbf{x}}^k - \mathbf{x}^* - \mathbf{A}_{I_0}^{\dagger} \mathbf{r}_{I_0} \in \Null(\mathbf{A}_{I_0})$ (from Lemma~\ref{lem:noisy_solution_space_geometry}\ref{lem:noisy_solution_space_geometry_b}) is a fixed point for $\mathbf{P}$,
\begin{align}\label{eq:noisy_scrk_exact_pf2}
    \mathbf{A}_{I_0 \cup \{ j \}}^{\dagger} \mathbf{A}_{I_0 \cup \{ j \}} (\widehat{\mathbf{x}}^k - \mathbf{x}^*)
    &= \begin{pmatrix}
        \mathbf{A}_{I_0}^{\dagger} - \frac{\mathbf{P} \mathbf{a}_j \mathbf{a}_j^{\tran} \mathbf{A}_{I_0}^{\dagger}}{\norm{\mathbf{P} \mathbf{a}_j}^2} & \rvline & \frac{\mathbf{P} \mathbf{a}_j}{\norm{\mathbf{P} \mathbf{a}_j}^2}
    \end{pmatrix}
    \left( \begin{array}{c}
        \mathbf{r}_{I_0} \Bstrut \\
        \hline
        \mathbf{a}_j^{\tran} (\widehat{\mathbf{x}}^k - \mathbf{x}^*) \Tstrut \\
    \end{array} \right) \nonumber\\
    &= \mathbf{A}_{I_0}^{\dagger} \mathbf{r}_{I_0} + \frac{(\mathbf{P} \mathbf{a}_j)^{\tran} (\widehat{\mathbf{x}}^k - \mathbf{x}^* - \mathbf{A}_{I_0}^{\dagger} \mathbf{r}_{I_0})}{\norm{\mathbf{P} \mathbf{a}_j}} \cdot \frac{\mathbf{P} \mathbf{a}_j}{\norm{\mathbf{P} \mathbf{a}_j}} ,
\end{align}
Furthermore, by Lemma~\ref{lem:noisy_solution_space_geometry}\ref{lem:noisy_solution_space_geometry_c}, $\mathbf{A}_{I_0}^{\dagger} \mathbf{r}_{I_0} \perp \mathbf{P}\mathbf{a}_j \in \Null(\mathbf{A}_{I_0})$, which implies that the two summands in both~\eqref{eq:noisy_scrk_exact_pf3} and~\eqref{eq:noisy_scrk_exact_pf2} are orthogonal. Hence, we can further expand~\eqref{eq:noisy_scrk_exact_pf1} to show that $\norm{\widehat{\mathbf{x}}^{k+1} - \mathbf{x}^*}^2$ is equal to
\begin{align*}
    \norm{\widehat{\mathbf{x}}^k - \mathbf{x}^*}^2 - \norm{\mathbf{A}_{I_0}^{\dagger} \mathbf{r}_{I_0}}^2  - \left| \frac{(\mathbf{P} \mathbf{a}_j)^{\tran} (\widehat{\mathbf{x}}^k - \mathbf{x}^* - \mathbf{A}_{I_0}^{\dagger} \mathbf{r}_{I_0})}{\norm{\mathbf{P} \mathbf{a}_j}} \right|^2 + \norm{\mathbf{A}_{I_0}^{\dagger} \mathbf{r}_{I_0}}^2 + \frac{|r_j - \mathbf{a}_j^{\tran} \mathbf{A}_{I_0}^{\dagger} \mathbf{r}_{I_0}|^2}{\norm{\mathbf{P} \mathbf{a}_j}^2} .
\end{align*}
By cancelling identical terms and taking the expectation, conditional on all the choices of the algorithm up to the $k$\textsuperscript{th} iteration (similar to the proof of Theorem~\ref{thm:scrk_convergence}), we obtain
\begin{equation} \label{eq:noisy_scrk_convergence_pf1}
    \mathbb{E}_k \norm{\widehat{\mathbf{x}}^{k+1} - \mathbf{x}^*}^2
    = \norm{\widehat{\mathbf{x}}^k - \mathbf{x}^*}^2
    - \frac{\norm{\mathbf{A}_{I_1} \mathbf{P} (\widehat{\mathbf{x}}^k - \mathbf{x}^* - \mathbf{A}_{I_0}^{\dagger} \mathbf{r}_{I_0})}^2}{\norm{\mathbf{A}_{I_1} \mathbf{P}}_F^2}
    + \frac{\norm{\mathbf{r}_{I_1} - \mathbf{A}_{I_1} \mathbf{A}_{I_0}^{\dagger} \mathbf{r}_{I_0}}^2}{\norm{\mathbf{A}_{I_1} \mathbf{P}}_F^2} .
\end{equation}

\textbf{Step 2. Spectral bounds.}
Recall that $\Null(\mathbf{A}_{I_1}\mathbf{P}) = \Null(\mathbf{P})$ from the proof of Theorem~\ref{thm:scrk_convergence}. Since $\widehat{\mathbf{x}}^k - \mathbf{x}^* - \mathbf{A}_{I_0}^{\dagger} \mathbf{r}_{I_0} \in \Null(\mathbf{A}_{I_0})$, which is orthogonal to $\Null(\mathbf{P}) = \Null(\mathbf{A}_{I_1}\mathbf{P})$, arguing as in~\eqref{eq:scrk_convergence_pf1} shows that
\begin{align*}
    \norm{\mathbf{A}_{I_1} \mathbf{P} (\widehat{\mathbf{x}}^k - \mathbf{x}^* - \mathbf{A}_{I_0}^{\dagger} \mathbf{r}_{I_0})}^2
    &\geq \sigma_{\mathrm{min}}^+(\mathbf{A}_{I_1} \mathbf{P})^2 \cdot \norm{\widehat{\mathbf{x}}^k - \mathbf{x}^* - \mathbf{A}_{I_0}^{\dagger} \mathbf{r}_{I_0}}^2 \\
\end{align*}
By expanding the square, $\norm{\widehat{\mathbf{x}}^k - \mathbf{x}^* - \mathbf{A}_{I_0}^{\dagger} \mathbf{r}_{I_0}}^2$ is equal to $\norm{\widehat{\mathbf{x}}^k - \mathbf{x}^*}^2 + \norm{\mathbf{A}_{I_0}^{\dagger} \mathbf{r}_{I_0}}^2 - 2 (\widehat{\mathbf{x}}^k - \mathbf{x}^*)^{\tran} \mathbf{A}_{I_0}^{\dagger} \mathbf{r}_{I_0}$.
Since $\mathbf{A}_{I_0}^{\dagger} = \mathbf{A}_{I_0}^{\tran} (\mathbf{A}_{I_0}\mathbf{A}_{I_0}^{\tran})^{-1}$ and $\mathbf{A}_{I_0} (\widehat{\mathbf{x}}^k - \mathbf{x}^*) = \mathbf{r}_{I_0}$,
\[
    (\widehat{\mathbf{x}}^k - \mathbf{x}^*)^{\tran} \mathbf{A}_{I_0}^{\dagger} \mathbf{r}_{I_0} = \mathbf{r}_{I_0}^{\tran} (\mathbf{A}_{I_0}\mathbf{A}_{I_0}^{\tran})^{-1} \mathbf{r}_{I_0} \leq \frac{\norm{\mathbf{r}_{I_0}}^2}{\sigma_{\mathrm{min}}(\mathbf{A}_{I_0})^2} .
\]
Hence, we can bound the second term of~\eqref{eq:noisy_scrk_convergence_pf1} from below by
\[
    \frac{\norm{\mathbf{A}_{I_1} \mathbf{P} (\widehat{\mathbf{x}}^k - \mathbf{x}^* - \mathbf{A}_{I_0}^{\dagger} \mathbf{r}_{I_0})}^2}{\norm{\mathbf{A}_{I_1} \mathbf{P}}_F^2}
    \geq \frac{\sigma_{\mathrm{min}}^+(\mathbf{A}_{I_1} \mathbf{P})^2}{\norm{\mathbf{A}_{I_1} \mathbf{P}}_F^2} \left( \norm{\widehat{\mathbf{x}}^k - \mathbf{x}^*}^2
    + \norm{\mathbf{A}_{I_0}^{\dagger} \mathbf{r}_{I_0}}^2 - \frac{2 \norm{\mathbf{r}_{I_0}}^2}{\sigma_{\mathrm{min}}(\mathbf{A}_{I_0})^2} \right) .
\]
By instating the definitions of $\gamma_0$ and $\gamma_1$, we have shown that
\begin{equation} \label{eq:noisy_scrk_convergence_pf2}
    \mathbb{E}_k \norm{\widehat{\mathbf{x}}^{k+1} - \mathbf{x}^*}^2
    \leq \left( 1 - \frac{\sigma_{\mathrm{min}}^+(\mathbf{A}_{I_1} \mathbf{P})^2}{\norm{\mathbf{A}_{I_1} \mathbf{P}}_F^2} \right) \cdot \norm{\widehat{\mathbf{x}}^k - \mathbf{x}^*}^2
    + \frac{\sigma_{\mathrm{min}}^+(\mathbf{A}_{I_1} \mathbf{P})^2}{\norm{\mathbf{A}_{I_1} \mathbf{P}}_F^2}
    \left( \gamma_0 + \gamma_1 \right) .
\end{equation}
By iterating~\eqref{eq:noisy_scrk_convergence_pf2}, we deduce that $\mathbb{E} \norm{\widehat{\mathbf{x}}^{k} - \mathbf{x}^*}^2$ is upper bounded by
\begin{align*}
    \left( 1 - \frac{\sigma_{\mathrm{min}}^+(\mathbf{A}_{I_1} \mathbf{P})^2}{\norm{\mathbf{A_{I_1} \mathbf{P}}}_F^2} \right)^k \cdot \norm{\widehat{\mathbf{x}}^0 - \mathbf{x}^*}^2
    + \sum_{i=0}^{k-1} \left( 1 - \frac{\sigma_{\mathrm{min}}^+(\mathbf{A}_{I_1} \mathbf{P})^2}{\norm{\mathbf{A}_{I_1} \mathbf{P}}_F^2} \right)^i \cdot \frac{\sigma_{\mathrm{min}}^+(\mathbf{A}_{I_1} \mathbf{P})^2}{\norm{\mathbf{A}_{I_1} \mathbf{P}}_F^2} (\gamma_1 + \gamma_2) .
\end{align*}
We conclude by bounding the geometric series by $\norm{\mathbf{A}_{I_1} \mathbf{P}}_F^2 / \sigma_{\mathrm{min}}^+(\mathbf{A}_{I_1} \mathbf{P})^2$.
\end{proof}

\begin{remark} \label{rmk:scrk_block_update_general}
A natural generalization of the SCRK update~\eqref{eq:scrk_block_update} is to project onto the solution space $\mathbf{A}_{I_0 \cup J} \mathbf{x} = \mathbf{b}_{I_0 \cup J}$, where $J \subseteq [m] \setminus I_0$ is a block of row indices disjoint from $I_0$ with $|J| \geq 1$:
\begin{equation} \label{eq:scrk_block_update_general}
    \mathbf{x}^{k+1} = \mathbf{x}^k + \mathbf{A}_{I_0 \cup J}^{\dagger} (\mathbf{b}_{I_0 \cup J} - \mathbf{A}_{I_0 \cup J} \mathbf{x}^k) .
\end{equation}
Assuming that $\mathbf{A}_{I_0 \cup J}$ has full row rank, Lemma~\ref{lem:pseudoinverse_block_formula} implies that~\eqref{eq:scrk_block_update_general} can be computed by the following two-step procedure (which does not require $\mathbf{x}^k$ to satisfy $\mathbf{A}_{I_0} \mathbf{x}^k = \mathbf{b}_{I_0}$):
\begin{enumerate}[label=(\arabic*), leftmargin=*]
    \item Project $\mathbf{x}^k$ onto the solution space $\mathbf{A}_{I_0} \mathbf{x} = \mathbf{b}_{I_0}$ to obtain $\mathbf{y}^k$:
    \[
        \mathbf{y}^k = \mathbf{x}^k + \mathbf{A}_{I_0}^{\dagger} (\mathbf{b}_{I_0} - \mathbf{A}_{I_0} \mathbf{x}^k) .
    \]

    \item Compute the new measurements $\boldsymbol{\beta}_J \coloneqq \mathbf{b}_J - \mathbf{A}_J \mathbf{A}_{I_0}^{\dagger} \mathbf{b}_{I_0} \in \reals^{|J|}$, then project $\mathbf{y}^k$ onto the solution space $\mathbf{A}_{J} \mathbf{x} = \mathbf{b}_{J}$ whilst remaining in the solution space of $\mathbf{A}_{I_0} \mathbf{x} = \mathbf{b}_{I_0}$ for $\mathbf{x}^{k+1}$:
    \[
        \mathbf{x}^{k+1} = \mathbf{y}^k + (\mathbf{A}_J \mathbf{P})^{\dagger} (\boldsymbol{\beta}_J - (\mathbf{A}_J \mathbf{P}) \mathbf{y}^k) .
    \]
\end{enumerate}
In particular, by restricting to a single row $J = \{ j \}$ and imposing the condition $\mathbf{A}_{I_0} \mathbf{x}^k = \mathbf{b}_{I_0}$, we recover the simplified update formula~\eqref{eq:scrk_update_main}, which provides an alternative algebraic proof of Lemma~\ref{lem:scrk_update_formula}.

By further restricting to the special case $I = \{ i \}$, the two-step procedure above reduces to an update of the two-subspace Kaczmarz method of~\cite{NeedellWard2013}. Thus, the SCRK method can be seen as a partial generalization of the two-subspace Kaczmarz method, except that the subset $I_0$ is fixed throughout the iterations to exploit specific features of the block $\mathbf{A}_{I_0}$, and similar results concerning coherence with respect to more general subsets of equations can be obtained (see Remark~\ref{rem:coherence_remark}).

Finally, while all the convergence results in this paper are stated for the case $|J| = 1$, we believe that similar techniques can be extended to the case of $|J| > 1$.
\end{remark}

\subsection{Exploiting structure with the SCRK method} \label{subsec:scrk_structure_discussion}

In this section, we discuss how the SCRK method can exploit approximately low-rank structure and geometric properties of the data matrix $\mathbf{A}$ to accelerate convergence. For simplicity, we will restrict our attention to the noiseless case. Our goal is to study the per-iteration convergence rate (i.e.\ with $k = 1$). First, note that the SCRK rate~\eqref{eq:scrk_convergence_1} is as good as the RK rate~\eqref{eq:rk_convergence_bound}. Indeed,
\begin{equation} \label{eq:restricted_singval_ineq}
    \sigma_{\mathrm{min}}^+(\mathbf{A}_{I_1} \mathbf{P})
    = \min_{\substack{\mathbf{z} \in \Null(\mathbf{A}_{I_0}) \\ \norm{\mathbf{z}} = 1}} \norm{\mathbf{A}_{I_1} \mathbf{z}}
    = \min_{\substack{\mathbf{z} \in \Null(\mathbf{A}_{I_0}) \\ \norm{\mathbf{z}} = 1}} \norm{\mathbf{A} \mathbf{z}} \\
    \geq \min_{\substack{\mathbf{z} \in \reals^n \\ \norm{\mathbf{z}} = 1}} \norm{\mathbf{A} \mathbf{z}}
    = \sigma_{\mathrm{min}}(\mathbf{A}) ,
\end{equation}
and $\norm{\mathbf{A}_{I_1} \mathbf{P}}_F^2 \leq \norm{\mathbf{P}}^2 \norm{\mathbf{A}_{I_1}}_F^2 \leq \norm{\mathbf{A}}_F^2$. Therefore,
\begin{equation} \label{eq:scrk_vs_rk_rate}
    1 - \frac{\sigma_{\mathrm{min}}^+(\mathbf{A}_{I_1} \mathbf{P})^2}{\norm{\mathbf{A}_{I_1} \mathbf{P}}_F^2}
    \leq 1 - \frac{\sigma_{\mathrm{min}}(\mathbf{A})^2}{\norm{\mathbf{A}}_F^2} .
\end{equation}
However, since each SCRK iteration requires more computation (as discussed in Remark~\ref{rmk:complexity}), we would like to understand when the SCRK method is advantageous to RK overall. In the following, we first examine what features of the matrices $\mathbf{A}$ and $\mathbf{A}_{I_0}$ lead to such an advantage in Section~\ref{subsec:scrk_geometry}. Furthermore, we discuss how a good subset $I_0$ of rows, if not explicitly given, can actually be efficiently found when $\mathbf{A}$ has approximately low-rank structure in Section~\ref{subsec:sampling_subspace}.

\subsubsection{Geometry of the matrix and convergence rates} \label{subsec:scrk_geometry}

As highlighted by~\eqref{eq:scrk_vs_rk_rate}, either $\norm{\mathbf{A}_{I_1} \mathbf{P}}_F^2 \ll \norm{\mathbf{A}}_F^2$ or $\sigma_{\mathrm{min}}^+(\mathbf{A}_{I_1} \mathbf{P}) \gg \sigma_{\mathrm{min}}(\mathbf{A})$ leads to significant per-iteration advantage of SCRK over RK. We describe two specific motivating examples of systems with such structure before generalizing our observations in Corollary~\ref{cor:scrk_coherence} below. For these examples, consider an arbitrary $(m - m_0)$-dimensional subspace $\mathcal{U}$ of $\reals^n$. Let $\{ \mathbf{u}_1, \dots, \mathbf{u}_{m_0} \}$ and $\{ \mathbf{c}_{m_0 + 1}, \mathbf{c}_{m_0 + 2}, \dots, \mathbf{c}_n \}$ be orthonormal bases for $\mathcal{U}$ and $\mathcal{U}^{\perp}$ respectively, and $\epsilon \approx 0$ be a small positive constant.

\begin{example}[$\sigma_{\mathrm{min}}^+(\mathbf{A}_{I_1} \mathbf{P}) \gg \sigma_{\mathrm{min}}(\mathbf{A})$] \label{eg:scrk_structure_smin}
This can happen if the equations in the selected block $\mathbf{A}_{I_0}$ are almost collinear, but the system $\mathbf{A}_{I_1} \mathbf{P}$ with projected rows is well-conditioned. Let $\overline{\mathbf{u}} \coloneqq \frac{1}{m_0} \sum_{i=1}^{m_0} \mathbf{u}_i$, and $\mathbf{A} \in \reals^{n \times n}$ be the matrix where the first $m_0$ rows are given by $\mathbf{a}_j \coloneqq (1 - \epsilon) \overline{\mathbf{u}} + \epsilon \mathbf{u}_i$, $j = 1, \dots, m_0$, and the remaining rows are $\mathbf{c}_{m_0+1}, \dots, \mathbf{c}_{n}$. Choose $I_0 = [m_0]$ so that $\mathbf{P}$ is the orthogonal projection onto $\mathcal{U}^{\perp}$. Then $\sigma^+_{\mathrm{min}}(\mathbf{A}_{I_1} \mathbf{P}) = 1 \gg \epsilon \geq \sigma_{\mathrm{min}}(\mathbf{A})$.
Indeed, if $\mathbf{e}_i$ is the $i$\textsuperscript{th} standard basis vector, then
\[
    \sigma_{\mathrm{min}}(\mathbf{A})
    = \min_{\mathbf{x} \in \reals^n : \norm{\mathbf{x}} = 1} \norm{\mathbf{A}^{\tran} \mathbf{x}}
    \leq \normbig{\mathbf{A}^{\tran} \frac{(\mathbf{e}_1 - \mathbf{e}_2)}{\sqrt{2}}}
    = \frac{\epsilon}{\sqrt{2}} \norm{\mathbf{u}_1 - \mathbf{u}_2} = \epsilon.
\]
Furthermore, since the rows of $\mathbf{A}_{I_1} \mathbf{P}$ form an orthonormal basis for $\mathcal{U}^{\perp}$, for any unit vector $\mathbf{x} = \sum_{i = m_0 + 1}^n \alpha_i \mathbf{c}_i$ in $\mathcal{U}^{\perp}$, we have $\norm{\mathbf{A}_{I_1} \mathbf{x}}^2 = \sum_{i = m_0 + 1}^n \alpha_i^2 = 1$, and hence $\sigma^+_{\mathrm{min}}(\mathbf{A}_{I_1} \mathbf{P}) = \min_{\mathbf{x} \in \mathcal{U}^{\perp} : \norm{\mathbf{x}} = 1} \norm{\mathbf{A}_{I_1} \mathbf{x}} = 1$.
\end{example}

\begin{example}[$\norm{\mathbf{A}_{I_1} \mathbf{P}}_F^2 \ll \norm{\mathbf{A}}_F^2$] \label{eg:scrk_structure_frobnorm}
This can happen if the block $\mathbf{A}_{I_0}$ is highly correlated with the remaining rows $(\mathbf{a}_j)_{j \in I_1}$. Let $\mathbf{A} \in \reals^{n \times n}$ be the matrix where the first $m_0$ rows are $\mathbf{u}_1, \dots, \mathbf{u}_{m_0}$, and the remaining rows are $\mathbf{a}_j \coloneqq (1 - \epsilon) \mathbf{v}_j + \epsilon \mathbf{c}_j$, where $\mathbf{v}_j$ is any unit vector in $\mathcal{U}$, for $j = m_0 + 1, \dots, n$. Choose $I_0 = [m_0]$, so that $\Range(\mathbf{A}_{I_0}^{\tran}) = \mathcal{U}$ and $\mathbf{P}$ is the orthogonal projection onto $\mathcal{U}^{\perp}$. Then $\mathbf{P} \mathbf{a}_j = \epsilon \mathbf{c}_j$ and $\norm{\mathbf{P} \mathbf{a}_j} = \epsilon \ll 1 = \norm{\mathbf{a}_j}$ for all $j \in I_1$, and hence $\norm{\mathbf{A}_{I_1} \mathbf{P}}_F^2 = \sum_{j \in I_1} \norm{\mathbf{P} \mathbf{a}_j}^2 \leq \epsilon^2 \sum_{j=1}^m \norm{\mathbf{a}_j}^2 = \epsilon^2 \norm{\mathbf{A}}_F^2$.
\end{example}

The calculations in the preceding example, together with Theorem~\ref{thm:scrk_convergence} and~\eqref{eq:restricted_singval_ineq}, generalize to the following result:

\begin{corollary} \label{cor:scrk_coherence}
Consider the same setup as Theorem~\ref{thm:scrk_convergence}, and assume that for some $\delta \in [0, 1)$,
\begin{equation} \label{eq:scrk_coherence_assumption}
    \frac{\norm{\mathbf{P} \mathbf{a}_j}^2}{\norm{\mathbf{a}_j}^2} \leq 1 - \delta^2 \quad \text{for all } j \in I_1 .
\end{equation}
Then the SCRK iterates $\mathbf{x}^k$ converge to $\mathbf{x}^*$ in expectation with
\begin{equation} \label{eq:scrk_coherence_1}
    \mathbb{E} \norm{\mathbf{x}^k - \mathbf{x}^*}^2 \leq \left( 1 - \frac{1}{1 - \delta^2} \cdot \frac{\sigma_{\mathrm{min}}(\mathbf{A})^2}{\norm{\mathbf{A}}_F^2} \right)^k \cdot \norm{\mathbf{x}^0 - \mathbf{x}^*}^2 .
\end{equation}
\end{corollary}

\begin{remark} \label{rem:coherence_remark}
Since $\mathbf{P}$ is the orthogonal projection onto $\Range(\mathbf{A}_{I_0}^{\tran})^{\perp}$, $\norm{\mathbf{P} \mathbf{a}_j}^2 / \norm{\mathbf{a}_j}^2 = \sin^2 \theta_j$ where $\theta_j$ is the principal angle between the subspaces $\Range(\mathbf{a}_j)$ and $\Range(\mathbf{A}_{I_0}^{\tran})$. Therefore, the quantity $\delta$ in~\eqref{eq:scrk_coherence_assumption} measures the coherence between the row space of the fixed block $\mathbf{A}_{I_0}$ and each of the remaining rows $(\mathbf{a}_j)_{j \in I_1}$. A value of $\delta$ close to one means that the principal angles are uniformly small; i.e.\ all of the $\mathbf{a}_j$ are close to the row space of $\mathbf{A}_{I_0}$ and offer little new information by themselves. By projecting each row with $\mathbf{P}$, the shared information is effectively modded out, and thus each SCRK iteration is able to make more meaningful progress towards the solution.

In particular, if we take $\mathbf{A}_{I_0} = \mathbf{a}_i^{\tran}$ to be a single row and assume that $\norm{\mathbf{a}_i} = 1 = \norm{\mathbf{a}_j}$, then $\norm{\mathbf{P} \mathbf{a}_j}^2 = 1 - |\mathbf{a}_i^{\tran} \mathbf{a}_j|^2$, where $|\mathbf{a}_i^{\tran} \mathbf{a}_j|$ is the correlation between $\mathbf{a}_i$ and $\mathbf{a}_j$. It is shown in~\cite{NeedellWard2013} that the two-subspace Kaczmarz method, which iteratively projects onto the solution space associated with two random rows, significantly improves upon RK if $\mathbf{A}$ has highly correlated rows. Thus, Corollary~\ref{cor:scrk_coherence} quantifies a similar phenomenon for the SCRK method for higher dimensional subspaces.
\end{remark}

\subsubsection{Sampling rows to find a good subspace} \label{subsec:sampling_subspace}

Previously, we showed that bounds of the form $\norm{\mathbf{A}_{I_1} \mathbf{P}}_F^2 \ll \norm{\mathbf{A}}_F^2$ using a specific choice of rows $I_0$ imply significant improvements in the convergence rate of the SCRK method over randomized Kaczmarz. However, what if we are not explicitly given a good set $I_0$, even though there is latent low-rank structure in $\mathbf{A}$ -- in the sense that the matrix has $r \ll n$ dominant singular values -- that can be exploited? We begin by considering a motivating hypothetical example where the row span of $\mathbf{A}_{I_0}$ is able to align perfectly with the leading right singular subspace.

\begin{example}[$\norm{\mathbf{A}_{I_1} \mathbf{P}}_F^2 \ll \norm{\mathbf{A}}_F^2$] \label{eg:scrk_structure_frobnorm2}
Let $\mathbf{A}_{(r)} \coloneqq \mathbf{U}_{(r)} \boldsymbol{\Sigma}_{(r)} \mathbf{V}_{(r)}^{\tran}$ be the best rank-$r$ approximation of $\mathbf{A}$ (with respect to $\norm{\cdot}_F$), where $\boldsymbol{\Sigma}_{(r)} = \mathrm{diag}(\sigma_1(\mathbf{A}), \dots, \sigma_r(\mathbf{A}))$ is the diagonal matrix of the top $r$ singular values of $\mathbf{A}$, and the columns of $\mathbf{V}_{(r)} \in \reals^{n \times r}$ and $\mathbf{U}_{(r)} \in \reals^{m \times r}$ contain the corresponding right and left singular vectors. Suppose that the row span of $\mathbf{A}_{I_0}$ equals $\Range(\mathbf{V}_{(r)})$. Then
\[
    \norm{\mathbf{A}_{I_1} \mathbf{P}}^2_F = \norm{\mathbf{A} \mathbf{P}}^2_F = \norm{\mathbf{P} \mathbf{A}^{\tran}}^2_F = \norm{\mathbf{A}^{\tran} - \mathbf{P}^{\perp} \mathbf{A}^{\tran}}^2_F = \norm{\mathbf{A}^{\tran} - \mathbf{A}_{(r)}^{\tran}}^2_F = \sum_{i=r+1}^n \sigma_i(\mathbf{A})^2 .
\]
If the top $r$ singular values of $\mathbf{A}$ are much larger than the rest, then $\norm{\mathbf{A}_{I_1} \mathbf{P}}^2_F$ is much smaller than $\norm{\mathbf{A}}_F^2 = \sum_{i=1}^n \sigma_i(\mathbf{A})^2$.
\end{example}

Note that in general, such a subset of rows does not exist in $\mathbf{A}$. This raises the following question: can we efficiently find a small subset $I_0$ of rows of $\mathbf{A}$ so that the row span of $\mathbf{A}_{I_0}$ is a good approximation of the top $r$-dimensional right singular subspace of $\mathbf{A}$?
This is known as the problem of finding an approximate CX decomposition in the randomized numerical linear algebra literature. Algorithms have been proposed that sample rows of $\mathbf{A}$ according to their Euclidean norms~\cite{DrKaMa2006b} or their leverage scores $(\ell_j)_{j \in [m]}$~\cite{DrMaMu2008}, where $\ell_j$ is the squared $\ell_2$ norm of the $j$\textsuperscript{th} row of $\mathbf{U}_{(r)}$, using the same notation as the example above (for more details, see~\cite{Mahoney2011}).
The following summarizes guarantees for these two sampling schemes proved in~\cite{DrMaMu2008, BoMaDr2009} and~\cite{DrKaMa2006b}:

\begin{theorem}[Theorem~1~\cite{DrMaMu2008} and~\cite{BoMaDr2009}; Theorem~4~\cite{DrKaMa2006b}] \label{thm:sampling_guarantees}
Suppose that $c$ rows of $\mathbf{A}$ are independently sampled, where row $j$ is selected with probability $p_j$ in each trial. Let $I_0$ be the set of indices of the sampled rows, and $\mathbf{P}$ be the orthogonal projection onto $\Null(\mathbf{A}_{I_0})$.
\begin{enumerate}[label=\textnormal{(\roman*)}, leftmargin=*]
    \item If $p_j = \ell_j / r$ and $c = O(r \log r / \epsilon^2)$, then with probability at least 0.9,
    \begin{equation} \label{eq:sampling_guarantees_2}
        \norm{\mathbf{A} \mathbf{P}}_F^2 \leq (1 + \epsilon) \norm{\mathbf{A} - \mathbf{A}_{(r)}}_F^2 = (1 + \epsilon) \sum_{i=r+1}^n \sigma_i(\mathbf{A})^2 .
    \end{equation}

    \item If $p_j = \norm{\mathbf{a}_j}^2 / \norm{\mathbf{A}}_F^2$ and $c = O(r / \epsilon^2)$, then with probability at least 0.9,
    \begin{equation} \label{eq:sampling_guarantees_1}
        \norm{\mathbf{A} \mathbf{P}}_F^2 \leq \norm{\mathbf{A} - \mathbf{A}_{(r)}}_F^2 + \epsilon \norm{\mathbf{A}}_F^2 = (1 + \epsilon) \sum_{i=r+1}^n \sigma_i(\mathbf{A})^2 + \epsilon \sum_{i=1}^r \sigma_i(\mathbf{A})^2 .
    \end{equation}
\end{enumerate}
\end{theorem}

Theorem~\ref{thm:sampling_guarantees} implies that sampling $c \approx r \log r$ rows of $\mathbf{A}$ produces a subspace that tames the leading $r$ singular values of $\mathbf{A}$ with high probability. In practice, it has been observed that a modest oversampling factor (i.e.\ $c$ is a small constant times $r$) usually suffices~\cite{DrMaMu2008}. The relative-error bound~\eqref{eq:sampling_guarantees_2} is better than the additive-error bound~\eqref{eq:sampling_guarantees_1}; however it is more costly because it requires the estimation of the leverage scores (see, e.g.,~\cite{DrMaMaWo2012, HaMaTr2011}). By combining Theorem~\ref{thm:sampling_guarantees} and~\eqref{eq:restricted_singval_ineq} with the SCRK convergence result (Theorem~\ref{thm:scrk_convergence}), we deduce the following:

\begin{corollary}
Suppose that $I_0 \subset [m]$ contains $m_0 = O(r \log r / \epsilon^2)$ rows of $\mathbf{A}$, randomly sampled according to the leverage scores of $\mathbf{A}$ relative to its best rank-$r$ approximation as described in Theorem~\ref{thm:sampling_guarantees}, and partition $\mathbf{A}$ into blocks $\mathbf{A}_{I_0}$ and $\mathbf{A}_{I_1}$ with $I_1 = [m] \setminus I_0$. Then with probability at least $0.9$ over the sampling of $I_0$, the SCRK iterates $\mathbf{x}^k$ satisfy
\[
    \mathbb{E} \norm{\mathbf{x}^k - \mathbf{x}^*}^2 \leq \left( 1 - \frac{\sigma_{n}(\mathbf{A})^2}{(1 + \epsilon) \sum_{i=r+1}^n \sigma_i(\mathbf{A})^2} \right)^k \cdot \norm{\mathbf{x}^0 - \mathbf{x}^*}^2 .
\]
\end{corollary}

Thus, if the rank of $\mathbf{A}$ is effectively less than $r$ (in the sense $\sum_{i=r+1}^{n} \sigma_i(\mathbf{A})^2 \ll \norm{\mathbf{A}}_F^2$), then the SCRK method with iterates constrained to the solution space corresponding to $m_0 = O(r \log r)$ randomly sampled rows significantly improves upon RK.
Note that similar results are known for the sketch-and-project method~\cite{derezinski2024sharp}.
Moreover, the effective rank of a large-scale matrix $\mathbf{A}$ can be estimated in a data-driven manner by sketching~\cite{MeierNakatsukasa2024}.

\subsection{SCRK on random data and dimension reduction} \label{sec:random_matrix_analysis}

Previously, we discussed how the SCRK method accelerates the iterative solver when the matrix $\mathbf{A}$ has approximately low-rank structure. In this section, we consider a somewhat complementary setting to study the effect of the subspace constraint when \emph{$\mathbf{A}$ is unstructured and homogeneous}: namely, when $\mathbf{A}$ is drawn from a class of generic random matrices (precisely defined below) whose rows behave like independent standard Gaussian vectors.
Such a matrix is typically well-conditioned as long as its aspect ration $m/n$ is large enough, and hence the corresponding linear system is easily solved using randomized Kaczmarz. However, for almost-square systems with an aspect ratio close to one, the convergence rate is far from optimal.

First, we review some definitions from probability theory (we refer to~\cite{Vershynin2018} for more details). If $\mathbf{a} \in \reals^n$ is a random vector, we say that $\mathbf{a}$ is \emph{mean-zero} if $\ev{\mathbf{a}} = \mathbf{0}$, and $\mathbf{a}$ is \emph{isotropic} if $\ev{\mathbf{a} \mathbf{a}^{\tran}} = \mathbf{I}$. We say that a scalar random variable $X$ is \emph{$K$-subgaussian} if its subgaussian norm $\norm{X}_{\psi_2} \coloneqq \inf_{t > 0} \{ \ev{\exp(X^2 / t^2) \leq 2} \}$ is bounded by $K > 0$; informally, this means that $X$ concentrates around its mean with a light, exponentially decaying tail. Furthermore, a random vector $\mathbf{a}$ is $K$-subgaussian if all of its one-dimensional marginals are $K$-subgaussian: $\norm{\mathbf{a}}_{\psi_2} \coloneqq \sup_{\mathbf{z} \in \reals^n: \norm{\mathbf{z}} = 1} \norm{\innprod{\mathbf{z}}{\mathbf{a}}}_{\psi_2} \leq K$.

As before, we will continue to assume that $\mathbf{A}$ has full rank (almost surely). For our model, we will allow $\mathbf{A}_{I_0} \in \reals^{m_0 \times n}$ to be arbitrary, and we assume that $\mathbf{A}_{I_1} \in \reals^{(m - m_0) \times n}$ is a random matrix that satisfies the following:
\begin{enumerate}[label=\textbf{A\arabic*.}, ref=A\arabic*, leftmargin=*]
    \item \label{asmp:A1}
    The rows of $\mathbf{A}_{I_1}$ are independent, mean-zero, isotropic, and $K$-subgaussian random vectors.
\end{enumerate}

The canonical example for our model is a standard Gaussian matrix $\mathbf{A}_{I_1}$ whose entries are independent standard normal random variables. In this special case, exact computations are often possible.
More generically, Assumption~\ref{asmp:A1} models unstructured matrices containing homogeneous data (that is centered and isotropic) with light tails.

Our main result in this section shows that for such matrices, the subspace constraint imposed by $\mathbf{P}$ acts as a form of dimension reduction, typically resulting in a near-optimal convergence rate of approximately $1 - 1/(n - m_0)$ as long as the ``effective aspect ratio'' $(m - m_0) / (n - m_0)$, which may be much larger than $m / n$, is large enough.

\begin{theorem} \label{thm:subgauss_scrk_convergence}
Suppose that the rows of $\mathbf{A}$ are partitioned into two blocks $\mathbf{A}_{I_0}$ and $\mathbf{A}_{I_1}$ of sizes $m_0$ and $m - m_0$ respectively, where $\mathbf{A}_{I_0}$ is arbitrary and $\mathbf{A}_{I_1}$ is a random matrix that satisfies Assumption~\ref{asmp:A1}.
There exists constants $c, R > 0$ (only depending on $K$) such that if
\[
    r \coloneqq \frac{m - m_0}{n - m_0} \geq R,
\]
then for any $\epsilon \in (0, 1)$, with probability at least $1 - 3 \exp\left\{ -c \epsilon^2 \left(1 - \sqrt{\frac{R}{r}} \right) (m - m_0) \right\}$ over the randomness in $\mathbf{A}_{I_1}$, the SCRK iterates $\mathbf{x}^k$ satisfy
\begin{equation} \label{eq:subgauss_scrk_convergence_1}
    \mathbb{E} \norm{\mathbf{x}^{k} - \mathbf{x}^*}^2 \leq \left( 1 - \frac{(1 - \epsilon)^2}{1 + \epsilon} \left( 1 - \sqrt{\frac{R}{r}} \right)^2 \cdot \frac{1}{n - m_0} \right)^k \cdot \norm{\mathbf{x}^0 - \mathbf{x}^*}^2 .
\end{equation}
In the special case where $\mathbf{A}_{I_1}$ is standard Gaussian, this result holds with $R = 1$.
\end{theorem}

The values of $R$ and $c$ depend on the precise distributional properties of the random matrix, and are, importantly, independent of $m$ and $n$. Note that for tall, large-scale systems with $m, n \gg 1$ and $r \gg 1$, the requirement $r \geq R$ is not difficult to meet, and taking $\epsilon \approx 0$ shows that the convergence rate is approximately $1 - 1 / (n - m_0)$ with a probability guarantee that is exponentially close to one.

\subsubsection{Proof of Theorem~\ref{thm:subgauss_scrk_convergence}}
\label{subsec:subgauss_scrk_convergence_proof}

To study the typical convergence rate with a random matrix, we will obtain tail bounds for $\sigma_{\mathrm{min}}^+(\mathbf{A}_{I_1} \mathbf{P})$ and $\norm{\mathbf{A}_{I_1} \mathbf{P}}_F^2$. The first lemma is deterministic, and shows that instead of studying the non-zero singular values of the $(m - m_0) \times n$ matrix $\mathbf{A}_{I_1} \mathbf{P}$, we can study the singular values of a thinner $(m - m_0) \times (n - m_0)$ matrix after rotating.

\begin{lemma} \label{lem:proj_matrix_singvals}
Let $\mathbf{X} \in \reals^{m \times n}$ be a matrix, and $\mathbf{P}\in \reals^{n \times n}$ be an orthogonal projection onto a $d$-dimensional subspace of $\reals^n$. Suppose that the columns of $\bar{\mathbf{Q}} \in \reals^{n \times d}$ form an orthonormal basis for $\Range(\mathbf{P})$. Then the non-zero singular values of $\mathbf{X} \mathbf{P} \in \reals^{m \times n}$ and $\mathbf{X} \bar{\mathbf{Q}} \in \reals^{m \times d}$ are the same.
\end{lemma}

\begin{proof}
Note that $\mathbf{P} = \bar{\mathbf{Q}}\bar{\mathbf{Q}}^{\tran}$. Let $\mathbf{X} \bar{\mathbf{Q}} = \mathbf{U} \mathbf{\Sigma}\mathbf{V}^{\tran}$ be a compact singular value decomposition of $\mathbf{X} \bar{\mathbf{Q}}$, which means that $\mathbf{\Sigma}$ is a square diagonal matrix containing the $\rank(\mathbf{X})$ non-zero singular values of $\mathbf{X} \bar{\mathbf{Q}}$, and $\mathbf{U}, \mathbf{V}$ are rectangular matrices with orthonormal columns. Observe that $\mathbf{X} \mathbf{P} = \mathbf{U} \mathbf{\Sigma} (\mathbf{V}')^{\tran}$ where $\mathbf{V}' \coloneqq \bar{\mathbf{Q}} \mathbf{V}$ also has orthonormal columns. This allows us to conclude the desired result since the non-zero singular values of $\mathbf{X} \mathbf{P}$ are presented in the same matrix $\mathbf{\Sigma}$.
\end{proof}



The following probabilistic result shows that the smallest non-zero singular value of $\mathbf{A}_{I_1} \mathbf{P}$ can be lower bounded with very high probability.

\begin{lemma} \label{lem:subgauss_minmax_singvals}
Let $\mathbf{P}$ be an orthogonal projection onto a fixed $(n - m_0)$-dimensional subspace. Suppose that the random matrix $\mathbf{A}_{I_1} \in \reals^{(m - m_0) \times n}$ satisfies Assumption~\ref{asmp:A1}. Then there exists an absolute constant $C > 0$ such that for all $s > 0$, with probability at least $1 - 2e^{-s^2 (m - m_0)}$, the smallest and largest non-zero singular values of $\mathbf{A}_{I_1} \mathbf{P}$ satisfy
\begin{align*}
    &\sigma_{\mathrm{min}}^+(\mathbf{A}_{I_1} \mathbf{P})
    \geq \sqrt{m - m_0} - C K^2 (\sqrt{n - m_0} + s \sqrt{m - m_0})
    \quad  \text{ and } \nonumber \\
    &\sigma_{\mathrm{max}}(\mathbf{A}_{I_1} \mathbf{P})
    \leq \sqrt{m - m_0} + C K^2 (\sqrt{n - m_0} + s \sqrt{m - m_0}) .
\end{align*}
In the case where $\mathbf{A}_{I_1}$ is Gaussian, the inequalities hold with $CK^2$ replaced by one.
\end{lemma}

\begin{proof}
Let $\bar{\mathbf{Q}} \in \reals^{(m - m_0) \times (n - m_0)}$ be a matrix whose columns form an orthonormal basis for $\Range(\mathbf{P})$. By Lemma~\ref{lem:proj_matrix_singvals}, the smallest and largest non-zero singular values of $\mathbf{A}_{I_1} \mathbf{P} \in \reals^{(m - m_0) \times n}$ and $\mathbf{B} := \mathbf{A}_{I_1} \bar{\mathbf{Q}} \in \reals^{(m - m_0) \times (n - m_0)}$ are equal. It can be directly checked that the rows of $\mathbf{B}$ are also independent, mean-zero, isotropic, $K$-subgaussian random vectors in $\reals^{n - m_0}$. Hence, using a standard tail bound for the extremal singular values $\sigma_{\mathrm{min}}(\mathbf{B})$ and $\sigma_{\mathrm{max}}(\mathbf{B})$ of the random matrix $\mathbf{B}$ (see~\cite[Theorem~4.6.1]{Vershynin2018}) implies the claimed inequalities.
In the Gaussian case, the precise constants can be computed using Gaussian concentration tools (see~\cite[{Corollary 7.3.3, Exercise 7.3.4}]{Vershynin2018}).
\end{proof}

\begin{remark}
The minimum restricted singular value of random matrices has also been studied in the context of universality laws for randomized dimension reduction in~\cite{OymakTropp2018}. If the entries of $\mathbf{A}_{I_1}$ are independent random variables satisfying some mild regularity conditions, then \cite[Theorem~II]{OymakTropp2018} establishes that $\sigma_{\mathrm{min}}^+(\mathbf{A}_{I_1} \mathbf{P}) \approx \sqrt{m - m_0} - C \sqrt{n - m_0}$ with high probability since $\Range(\mathbf{P})$ is a $(n - m_0)$-dimensional subspace in $\reals^n$. Thus, $\sigma_{\mathrm{min}}^+(\mathbf{A}_{I_1} \mathbf{P})$ is of comparable order for a wide class of distributions. However, the maximum restricted singular value is not necessarily universal.
\end{remark}

Next, our goal is to obtain tail bounds for $\norm{\mathbf{A}_{I_1} \mathbf{P}}_F^2 = \sum_{j \in I_1} \norm{\mathbf{P} \mathbf{a}_j}^2$. In the setting where the rows $\mathbf{a}_j$ of $\mathbf{A}_{I_1}$ are mean-zero, $K$-subgaussian random vectors, it is proved in~\cite{JinEtAl2019} that the Euclidean norms $\norm{\mathbf{a}_j}$ are $O(K \sqrt{n})$-subgaussian. The next lemma states that the norms of the projected vectors $\norm{\mathbf{P} \mathbf{a}_j}$ are $O(K \sqrt{n - m_0})$-subgaussian.

\begin{lemma} \label{lem:projected_rownorm_subgauss}
Let $\mathbf{a}$ be a mean-zero, $K$-subgaussian random vector in $\reals^n$, and $\mathbf{P}$ be an orthogonal projection onto a fixed $d$-dimensional subspace. Then the subgaussian norm of $\norm{\mathbf{P} \mathbf{a}}$ is bounded by $C K \sqrt{d}$ for some absolute constant $C > 0$.
\end{lemma}

The proof uses the following geometric observation about unit spheres of subspaces, which we record for later reference. We say that $\mathcal{N}$ is an \emph{$\epsilon$-net} of a set $S \subseteq \reals^n$ if $\mathcal{N} \subseteq S$ and every point in $S$ is within distance $\epsilon$ of some point in $\mathcal{N}$. It is known that there exists an $\epsilon$-net of the $d$-dimensional unit sphere $\mathbb{S}^{d - 1} \coloneqq \{ \mathbf{x} \in \reals^d : \norm{\mathbf{x}} = 1 \}$ with cardinality bounded by $(1 + 2 / \epsilon)^d$ for any $d$ (see, e.g.,~\cite[Corollary 4.2.13]{Vershynin2018}). Thus, if $\mathcal{U}$ is a $d$-dimensional subspace of $\reals^n$, then by identifying $\mathcal{U} \cong \reals^d$ (using the fact that rotations are isometries) and obtaining a net of $\mathbb{S}^{d-1}$, we deduce the following:

\begin{lemma} \label{lem:subspace_unitsphere_covering}
Let $\mathcal{U}$ be a $d$-dimensional subspace of $\reals^n$. Then for any $\epsilon > 0$, there exists an $\epsilon$-net $\mathcal{N}$ of $\mathcal{U} \cap \mathbb{S}^{n-1}$ with cardinality $|\mathcal{N}| \leq (1 + 2 / \epsilon)^d$.
\end{lemma}

\begin{proof}[Proof of Lemma~\ref{lem:projected_rownorm_subgauss}]
The proof is similar to the proof of~\cite[Lemma~1]{JinEtAl2019}; we provide it for completeness. First, it can be checked that $\mathbf{P} \mathbf{a}$ is also a mean-zero $K$-subgaussian random vector. Next, by Lemma~\ref{lem:subspace_unitsphere_covering}, we can fix a $1/2$-net $\mathcal{N}$ of $\Range(\mathbf{P}) \cap \mathbb{S}^{n-1}$ with cardinality $|\mathcal{N}| \leq 5^{n - m_0}$. By using these observations, we will show that
\begin{equation} \label{eq:projected_rownorm_subgauss_pf1}
    \prob{\norm{\mathbf{P} \mathbf{a}} \geq t} \leq 2 e^{-t^2 / (10 K^2 (n - m_0))} \quad \text{for all } t \geq 0 .
\end{equation}
Indeed, for any realization of $\mathbf{P} \mathbf{a}$, there exists $\mathbf{v} \in \mathcal{N}$ such that $\normbig{\frac{\mathbf{P} \mathbf{a}}{\norm{\mathbf{P} \mathbf{a}}} - \mathbf{v}} \leq 1/2$, and we can write
\[
    \norm{\mathbf{P} \mathbf{a}}
    = \innprod{\mathbf{v}}{\mathbf{P} \mathbf{a}} + \innprod{\frac{\mathbf{P} \mathbf{a}}{\norm{\mathbf{P} \mathbf{a}}} - \mathbf{v}}{\mathbf{P} \mathbf{a}}
    \leq \innprod{\mathbf{v}}{\mathbf{P} \mathbf{a}} + \normbig{\frac{\mathbf{P} \mathbf{a}}{\norm{\mathbf{P} \mathbf{a}}} - \mathbf{v}} \cdot \norm{\mathbf{P} \mathbf{a}}
    \leq \innprod{\mathbf{v}}{\mathbf{P} \mathbf{a}} + \frac{\norm{\mathbf{P} \mathbf{a}}}{2}
\]
to deduce that $\norm{\mathbf{P} \mathbf{a}} \leq 2 \innprod{\mathbf{v}}{\mathbf{P} \mathbf{a}}$. Therefore, because $\innprod{\mathbf{v}}{\mathbf{P} \mathbf{a}}$ is a mean-zero, $K$-subgaussian random variable, a union bound implies that for all $t \geq 0$,
\[
   \prob{\norm{\mathbf{P} \mathbf{a}} \geq t}
   \leq \prob{\exists \mathbf{v} \in \mathcal{N} : \innprod{\mathbf{v}}{\mathbf{P} \mathbf{a}} \geq \frac{t}{2}}
   \leq 5^{n - m_0} \cdot e^{-t^2 / (4 K^2)} .
\]
We claim that this implies~\eqref{eq:projected_rownorm_subgauss_pf1}. If $t^2 \leq 4 \log(5) K^2 (n - m_0)$, then~\eqref{eq:projected_rownorm_subgauss_pf1} trivially holds. Otherwise, if $t^2 = 4 \log(5) K^2 (n - m_0) + s$ for $s > 0$, then
\[
   5^{n - m_0} \cdot e^{-t^2 / (4 K^2)}
   = e^{-s / (4K^2)}
   \leq e^{-s / (6 \log(5) K^2 (n - m_0))}
   \leq 2 e^{-t^2 / (10 K^2 (n - m_0))} .
\]
Thus, the tail bound~\eqref{eq:projected_rownorm_subgauss_pf1} holds for all $t \geq 0$, which implies that (by, e.g.,~\cite[Proposition~2.5.2]{Vershynin2018}), $\norm{\mathbf{P} \mathbf{a}}$ has subgaussian norm bounded by $C K \sqrt{n - m_0}$ for some absolute constant $C > 0$.
\end{proof}

\begin{lemma} \label{lem:projected_frobenius_concentration}
Consider the same setup as Lemma~\ref{lem:subgauss_minmax_singvals}. Then there exists an absolute constant $c > 0$ such that for all $\epsilon > 0$, with probability at least $1 - e^{-c \min\{ \epsilon, \epsilon^2 \} (m - m_0) / K^4}$,
\begin{equation} \label{eq:projected_frobenius_concentration_1}
    \norm{\mathbf{A}_{I_1} \mathbf{P}}_F^2 \leq (1 + \epsilon) (m - m_0)(n - m_0) .
\end{equation}
\end{lemma}

\begin{proof}
Since $\mathbf{P}$ is an orthogonal projection onto an $(n - m_0)$-dimensional subspace and $\mathbf{a}_j$ is isotropic, using the cyclic property of trace implies that for all $j \in I_1$,
\[
    \mathbb{E} \norm{\mathbf{P} \mathbf{a}_j}^2
    = \ev{\mathrm{tr}(\mathbf{a}_j^{\tran} \mathbf{P} \mathbf{a}_j)}
    = \mathrm{tr}(\ev{\mathbf{a}_j \mathbf{a}_j^{\tran}} \mathbf{P})
    = \mathrm{tr}(\mathbf{P}) = n - m_0.
\]
Therefore,
\[
    \mathbb{E}\norm{\mathbf{A}_{I_1} \mathbf{P}}_F^2
    = \mathbb{E}\left[ \sum\limits_{j \in I_1} \norm{\mathbf{P} \mathbf{a}_j}^2 \right]
    = |I_1| \cdot \mathbb{E}\norm{\mathbf{P} \mathbf{a}_1}^2
    = (m - m_0) (n - m_0) .
\]
Now, the random variables $\norm{\mathbf{P} \mathbf{a}_j}$ are independent and, by Lemma~\ref{lem:projected_rownorm_subgauss}, $O(K \sqrt{n - m_0})$-subgaussian. Hence, by centering and Bernstein's inequality~\cite[Theorem~2.8.1]{Vershynin2018}, there exists an absolute constant $c > 0$ such that
\[
    \prob{\norm{\mathbf{A}_{I_1} \mathbf{P}}_F^2 - (m - m_0)(n - m_0) \geq t} \leq \exp\left( -\frac{c}{K^4} \min\left\{ \frac{t^2}{(m - m_0) (n - m_0)^2} ,\, \frac{t}{n - m_0} \right\} \right)
\]
for all $t \geq 0$. By choosing $t = \epsilon(m - m_0)(n - m_0)$, we obtain~\eqref{eq:projected_frobenius_concentration_1}.
\end{proof}

The main result of this section now easily follows from the tail bounds for $\mathbf{A}_{I_1} \mathbf{P}$.

\begin{proof}[Proof of Theorem~\ref{thm:subgauss_scrk_convergence}]
Suppose that the random matrix $\mathbf{A}_{I_1} \mathbf{P}$ satisfies the events in Lemma~\ref{lem:subgauss_minmax_singvals}, using $s = \epsilon \left( 1 - \frac{CK^2}{\sqrt{r}} \right) \frac{1}{CK^2}$ and relabelling $C^2 K^4$ by $R$, and Lemma~\ref{lem:projected_frobenius_concentration}. If this occurs, which holds with the claimed probability after simplifying, the convergence result~\eqref{eq:subgauss_scrk_convergence_1} then directly follows from the SCRK convergence result, Theorem~\ref{thm:scrk_convergence}.
\end{proof}


\section{Analysis of the QuantileSCRK algorithm} \label{sec:corruptions_analysis}

In this section, we consider the QuantileSCRK method for solving corrupted linear systems (Algorithm~\ref{alg:quantilescrk_method}). Recall that in our model, we are given a corrupted measurement vector $\widetilde{\mathbf{b}} \coloneqq \mathbf{b} + \mathbf{b}_{\mathcal{C}}$, where $\mathbf{b}_{\mathcal{C}}$ is a sparse vector of arbitrary corruptions supported on $\mathcal{C} \subseteq [m]$, as well as a corruption-free subset $I_0 \subseteq [m]$ of size $m_0$ such that $(\mathbf{b}_{\mathcal{C}})_{I_0} = \mathbf{0}$. Our goal is to reconstruct the solution $\mathbf{x}^*$ of the linear system $\mathbf{A} \mathbf{x} = \mathbf{b}$.

Our main result in this section is Theorem~\ref{thm:quantilescrk_convergence}, which shows that the QuantileSCRK method is able to converge robustly and efficiently when $\mathbf{A}$ is an unstructured random matrix as long as the effective aspect ratio $(m - m_0) / (n - m_0)$ is tall enough and the proportion of corrupted measurements $|\mathcal{C}| / (m - m_0)$ is not too large. Specifically, we consider the class of ``Gaussian-like'' random matrices previously considered in Section~\ref{sec:random_matrix_analysis}, and assume that $\mathbf{A}_{I_1}$ is a random matrix that satisfies Assumption~\ref{asmp:A1} in addition to the following continuity assumption:
\begin{enumerate}[label=\textbf{A\arabic*.}, ref=A\arabic*, leftmargin=*]
    \setcounter{enumi}{1}  
    \item \label{asmp:A2}
    Each row of $\mathbf{A}_{I_1}$ either has a log-concave distribution\footnote{A log-concave distribution in $\reals^n$ has a probability density $f$ that satisfies $f(\lambda \mathbf{x} + (1 - \lambda) \mathbf{y}) \geq f(\mathbf{x})^\lambda f(\mathbf{y})^{1 - \lambda}$ for all $\lambda \in [0, 1]$ and $\mathbf{x}, \mathbf{y} \in \reals^n$.} or has independent entries with bounded probability densities.\footnote{By scaling, we may assume without loss of generality that the densities are bounded by one.}
\end{enumerate}
The class of log-concave distributions is a generalization of the standard Gaussian distribution that allows for some dependence between the entries of a random vector; for example, the uniform distribution over any convex body in $\reals^n$ is log-concave. For more details and examples, we refer to~\cite{SaumardWellner2014}.

Assumption~\ref{asmp:A2} is essentially needed for technical reasons for our proof of Theorem~\ref{thm:quantilescrk_convergence}. Empirically, convergence is observed even if $\mathbf{A}$ has random discrete entries~\cite{HaNeReSw2022}, or if $\mathbf{A}$ is a structured, sparse matrix in an imaging problem (see Section~\ref{subsec:experiments_ct}). The assumption of having independent coordinates with bounded densities in each row was previously considered in~\cite{HaNeReSw2022}, and we extend the model by allowing for log-concave distributions.

We can now state our main result:

\begin{theorem} \label{thm:quantilescrk_convergence}
Suppose that the rows of $\mathbf{A}$ are partitioned into blocks $\mathbf{A}_{I_0}$ and $\mathbf{A}_{I_1}$ of sizes $m_0$ and $m - m_0$ respectively, where $\mathbf{A}_{I_0}$ is arbitrary and $\mathbf{A}_{I_1}$ is a random matrix that satisfies Assumptions~\ref{asmp:A1} and~\ref{asmp:A2}. Suppose that the corrupted measurement vector $\widetilde{\mathbf{b}} = \mathbf{b} + \mathbf{b}_{\mathcal{C}}$ is observed, $(\mathbf{b}_{\mathcal{C}})_{I_0} = \mathbf{0}$, and a quantile parameter $q \in (0, 1)$ is fixed. There exist constants $\beta_0 \in (0, 1)$ and $R \geq 1$ (only depending on $q$ and $K$) such that if
\begin{equation} \label{eq:quantilescrk_convergence_1}
    \frac{m - m_0}{n - m_0} \geq R \quad \text{and} \quad \beta \coloneqq \frac{|\mathcal{C}|}{m - m_0} \leq \beta_0 ,
\end{equation}
then for some constants $c_1, c_2 > 0$ (only depending on $q$, $\beta$, and $K$), with probability at least $1 - 6 e^{-c_1 (m - m_0)}$ over the randomness in $\mathbf{A}_{I_1}$, the QuantileSCRK iterates $\mathbf{x}^k$ from Algorithm~\ref{alg:quantilescrk_method} converge to the solution $\mathbf{x}^*$ in expectation with
\begin{equation} \label{eq:quantilescrk_convergence_3}
    \mathbb{E} \norm{\mathbf{x}^k - \mathbf{x}^*}^2 \leq \left( 1 - \frac{c_2}{n - m_0} \right)^k \cdot \norm{\mathbf{x}^0 - \mathbf{x}^*}^2 .
\end{equation}
\end{theorem}

As mentioned previously, the values of the constants $c_1$ and $c_2$ are dominated in large-scale systems with $m, n \gg 1$, and the requirement $(m - m_0) / (n - m_0) \geq R$ is not difficult to meet if the system is tall and there is enough external knowledge (i.e.\ $m \gg n$, $m_0 \gg 1$). In addition, we believe that it should be possible to obtain sharper theoretical estimates for $\beta_0$ and $R$.

The strategy to prove Theorem~\ref{thm:quantilescrk_convergence} is to combine a deterministic sufficient condition for the convergence of QuantileSCRK, adapting a result for QuantileRK proved by~\cite{Steinerberger2023}, with probabilistic results for the spectra of the projected random matrix $\mathbf{A}_{I_1} \mathbf{P}$.
First, we define some spectral quantities that will be needed. For $\alpha \in (0, 1]$, define
\begin{equation} \label{eq:defn_uniform_min_singval}
    \sigma_{\alpha, \mathrm{min}}^+(\mathbf{A}_{I_1} \mathbf{P})
    \coloneqq \inf_{\substack{T \subseteq I_1 \\ |T| = \alpha (m - m_0)}} \sigma_{\mathrm{min}}^+((\mathbf{A}_{I_1} \mathbf{P})_T) .
\end{equation}
For simplicity, we will assume throughout that $\alpha (m - m_0)$ is an integer. This quantity, which represents the uniform minimum singular value over all row submatrices of $\mathbf{A}_{I_1} \mathbf{P}$ with $\alpha(m - m_0)$ rows, has appeared in previous analyses of the QuantileRK algorithm~\cite{HaNeReSw2022, Steinerberger2023}, and quantifies whether there are any poorly-conditioned row submatrices that are particularly susceptible to corruptions. Similarly, define
\begin{equation} \label{eq:defn_uniform_max_frobnorm}
    Z_{\alpha}
    \coloneqq \sup_{\substack{T \subseteq I_1 \\ |T| = \alpha (m - m_0)}} \norm{(\mathbf{A}_{I_1} \mathbf{P})_T}_F^2 .
\end{equation}
Together, \eqref{eq:defn_uniform_min_singval} and~\eqref{eq:defn_uniform_max_frobnorm} provide a uniform upper bound for the scaled condition numbers of all row submatrices of $\mathbf{A}_{I_1} \mathbf{P}$ possibly containing the uncorrupted, admissible rows (i.e.\ whose residuals are smaller than the quantile threshold). This is the key step which guarantees that the expected improvement from moving in an uncorrupted direction offsets the expected deterioration caused by a corruption.

\subsection{A deterministic condition for convergence}
\label{subsec:quantilescrk_convergence_condition}

The following lemma provides a deterministic condition that guarantees the convergence of the QuantileSCRK algorithm for any arbitrary sparse corruption vector $\mathbf{b}_{\mathcal{C}}$, which may be of independent interest. This is adapted from a similar condition for convergence of the QuantileRK algorithm that was proved by Steinerberger~\cite{Steinerberger2023}.

\begin{lemma} \label{lem:quantilescrk_convergence_condition}
Recall that $\mathcal{C} \subseteq I_1$ are the indices of the corrupted measurements, and $\beta = |\mathcal{C}| / (m - m_0)$. Suppose that $\beta < q < 1 - \beta$. Define
\begin{equation} \label{eq:quantilescrk_convergence_condition_1}
    C_{q, \beta} \coloneqq
    \frac{1}{Z_{q - \beta}} \left\{ \sigma_{q - \beta, \mathrm{min}}^+(\mathbf{A}_{I_1} \mathbf{P})^2
    - \sigma_{\mathrm{max}}(\mathbf{A}_{I_1} \mathbf{P})^2
    \left( \frac{\beta}{1 - q}
    + 2 \sqrt{\frac{\beta}{1 - q}} \right) \right\} ,
\end{equation}
where $\sigma_{q - \beta, \mathrm{min}}^+(\mathbf{A}_{I_1} \mathbf{P})$ is defined in~\eqref{eq:defn_uniform_min_singval}, and $Z_{q - \beta}$ in~\eqref{eq:defn_uniform_max_frobnorm}. If $C_{q, \beta} > 0$, or equivalently,
\begin{equation} \label{eq:quantilescrk_convergence_condition_3}
    \frac{\sigma^+_{q - \beta, \mathrm{min}}(\mathbf{A}_{I_1} \mathbf{P})^2}{\sigma_{\mathrm{max}}(\mathbf{A}_{I_1} \mathbf{P})^2}
    > \frac{\beta}{1 - q} + 2 \sqrt{\frac{\beta}{1 - q}} ,
\end{equation}
then the QuantileSCRK iterates $\mathbf{x}^k$ from Algorithm~\ref{alg:quantilescrk_method} converge to the solution $\mathbf{x}^*$ in expectation with
\begin{equation} \label{eq:quantilescrk_convergence_condition_4}
    \mathbb{E} \norm{\mathbf{x}^k - \mathbf{x}^*}^2 \leq \left( 1 - C_{q, \beta} \right)^k \cdot \norm{\mathbf{x}^0 - \mathbf{x}^*}^2 .
\end{equation}
\end{lemma}

The proof of Lemma~\ref{lem:quantilescrk_convergence_condition} follows the same strategy as~\cite{Steinerberger2023} with a minor improvement in the condition~\eqref{eq:quantilescrk_convergence_condition_3} for convergence. For completeness, we provide the full details.

\begin{proof}
Consider the iterate $\mathbf{x}^k$. Recall that $J = J(q, k)$ is the set of indices of the admissible rows that satisfy $|b_j - \mathbf{a}_j^{\tran} \mathbf{x}^k| \leq \gamma_q = q\text{-quantile}\left\{ |b_j - \mathbf{a}_j^{\tran} \mathbf{x}^k| : j \in I_1 \right\}$, with $|J| = q(m - m_0)$. Let $S \coloneqq \mathcal{C} \cap J$ be the indices of the corrupted yet admissible rows, which satisfies $0 \leq |S| \leq \beta(m - m_0)$. Recall that the row $j$ is sampled from $J$ with probability equal to $\norm{\mathbf{P} \mathbf{a}_j}^2 / Z_J$, where $Z_J \coloneqq \sum_{j \in J} \norm{\mathbf{P} \mathbf{a}_j}^2$ is the normalizing constant. Conditional on all the choices up to the $k$\textsuperscript{th} iteration, we have
\begin{align}
    \mathbb{E}_k \norm{\mathbf{x}^{k+1} - \mathbf{x}^*}^2
    &= \left[ 1 - \frac{\sum_{j \in S} \norm{\mathbf{P} \mathbf{a}_j}^2}{Z_J} \right] \mathbb{E}_{\{j \in J \setminus S\}} \norm{\mathbf{x}^{k+1} - \mathbf{x}^*}^2 \label{eq:quantilescrk_convergence_condition_pf2b} \\
    &\quad\quad+ \frac{\sum_{j \in S} \norm{\mathbf{P} \mathbf{a}_j}^2}{Z_J} \mathbb{E}_{\{j \in S\}} \norm{\mathbf{x}^{k+1} - \mathbf{x}^*}^2 \label{eq:quantilescrk_convergence_condition_pf2a},
\end{align}
where $\mathbb{E}_{\{j \in S\}}$ denotes the expectation further conditional on $j \in S$, and similarly for $\mathbb{E}_{\{j \in J \setminus S\}}$. We proceed to estimate the two summands~\eqref{eq:quantilescrk_convergence_condition_pf2b} and~\eqref{eq:quantilescrk_convergence_condition_pf2a} individually.

\textbf{Step 1: Lower bounding the improvement from selecting an uncorrupted equation.}
Conditional on sampling an admissible, uncorrupted equation $j \in J \setminus S$, the improvement is given by one step of the SCRK method applied to the row submatrix $(\mathbf{A}_{I_1} \mathbf{P})_{J \setminus S}$. Thus, by Theorem~\ref{thm:scrk_convergence},
\[
    \mathbb{E}_{\{j \in J \setminus S\}} \norm{\mathbf{x}^{k+1} - \mathbf{x}^*}^2
    \leq \left( 1 - \frac{\sigma_{\mathrm{min}}^+((\mathbf{A}_{I_1} \mathbf{P})_{J \setminus S})^2}{\sum_{j \in J \setminus S} \norm{\mathbf{P} \mathbf{a}_j}^2} \right) \cdot \norm{\mathbf{x}^k - \mathbf{x}^*}^2 .
\]
Since $|J \setminus S| \geq (q - \beta) (m - m_0)$, by using the definition of $\sigma_{q - \beta, \mathrm{min}}^+(\mathbf{A}_{I_1} \mathbf{P})$ in~\eqref{eq:defn_uniform_min_singval} together with the fact that adding rows to a matrix can only increase its minimum singular value, we obtain the following upper bound for the first term~\eqref{eq:quantilescrk_convergence_condition_pf2b}:
\begin{equation} \label{eq:quantilescrk_convergence_condition_pf4}
    \left( 1 - \frac{\sum_{j \in S} \norm{\mathbf{P} \mathbf{a}_j}^2}{Z_J} \right) \cdot \norm{\mathbf{x}^k - \mathbf{x}^*}^2
    - \frac{\sigma_{q - \beta, \mathrm{min}}^+(\mathbf{A}_{I_1} \mathbf{P})^2}{Z_J} \cdot \norm{\mathbf{x}^k - \mathbf{x}^*}^2 .
\end{equation}

\textbf{Step 2: Upper bounding the deterioration from selecting a corrupted equation.}
The second term~\eqref{eq:quantilescrk_convergence_condition_pf2a} represents the possible deterioration from selecting a corrupted yet admissible row that may take $\mathbf{x}^k$ further away from the solution $\mathbf{x}^*$. By expanding the square, it is equal to
\begin{align}\label{eq:quantilescrk_convergence_condition_pf3}
    &\frac{\sum_{j \in S} \norm{\mathbf{P} \mathbf{a}_j}^2}{Z_J} \sum_{j \in S} \frac{\norm{\mathbf{P} \mathbf{a}_j}^2}{\sum_{i \in S} \norm{\mathbf{P} \mathbf{a}_i}^2} \normbig{\mathbf{x}^k - \mathbf{x}^*
    + \frac{b_j - \mathbf{a}_j^{\tran} \mathbf{x}^k}{\norm{\mathbf{P} \mathbf{a}_j}^2} \mathbf{P} \mathbf{a}_j}^2 \nonumber\\
    &= \frac{\sum_{j \in S} \norm{\mathbf{P} \mathbf{a}_j}^2}{Z_J} \norm{\mathbf{x}^k - \mathbf{x}^*}^2
    + \frac{1}{Z_J} \sum_{j \in S} |b_j - \mathbf{a}_j^{\tran} \mathbf{x}^k|^2
    + \frac{2}{Z_J} \sum_{j \in S} (b_j - \mathbf{a}_j^{\tran} \mathbf{x}^k) (\mathbf{P} \mathbf{a}_j)^{\tran} (\mathbf{x}^k - \mathbf{x}^*) \nonumber\\
    &\leq \frac{\sum_{j \in S} \norm{\mathbf{P} \mathbf{a}_j}^2}{Z_J} \norm{\mathbf{x}^k - \mathbf{x}^*}^2
    + \frac{1}{Z_J} |S| \gamma_q^2 + \frac{2}{Z_J} \gamma_q \sqrt{|S|} \, \norm{(\mathbf{A}_{I_1} \mathbf{P})_C (\mathbf{x}^k - \mathbf{x}^*)} ,
\end{align}
where the definition of the quantile $\gamma_q$ and Cauchy-Schwarz is used for the inequality.

\textbf{Step 3: Bounding the $q$-quantile of a sample.}
Since any uncorrupted row $\mathbf{a}_j$ with $j \in I_1 \setminus \mathcal{C}$ satisfies $\mathbf{a}_j^{\tran} \mathbf{x}^* = b_j$, we have
\[
    b_j - \mathbf{a}_j^{\tran} \mathbf{x}^k = \mathbf{a}_j^{\tran}(\mathbf{x}^* - \mathbf{x}^k) = (\mathbf{P} \mathbf{a}_j)^{\tran} (\mathbf{x}^* - \mathbf{x}^k) ,
\]
recalling that $\mathbf{x}^* - \mathbf{x}^k \in \Null(\mathbf{A}_{I_0})$. Since there are at least $(1 - q)(m - m_0) - (\beta(m - m_0) - |S|) = (1 - q - \beta)(m - m_0) + |S|$ uncorrupted equations in $I_1$ whose residual is larger than $\gamma_q$, we have
\begin{align*}
    ((1 - q - \beta) (m - m_0) + |S|) \gamma_q^2
    &\leq \sum_{j \in I_1 \setminus \mathcal{C}} |b_j - \mathbf{a}_j^{\tran} \mathbf{x}^k|^2
    \leq \sum_{j \in I_1} \left| ( \mathbf{P} \mathbf{a}_j )^{\tran} (\mathbf{x}^* - \mathbf{x}^k) \right|^2 \\
    &= \norm{\mathbf{A}_{I_1} \mathbf{P} (\mathbf{x}^k - \mathbf{x}^*)}^2
    \leq \sigma_{\mathrm{max}}(\mathbf{A}_{I_1} \mathbf{P})^2 \cdot \norm{\mathbf{x}^k - \mathbf{x}^*}^2 .
\end{align*}
Therefore, the $q$-quantile of the sizes of the residuals can be bounded by
\begin{equation} \label{eq:quantilescrk_convergence_condition_pf1}
    \gamma_q
    \leq \frac{\sigma_{\mathrm{max}}(\mathbf{A}_{I_1} \mathbf{P})}{\sqrt{(m - m_0)(1 - q - \beta) + |S|}} \cdot \norm{\mathbf{x}^k - \mathbf{x}^*} .
\end{equation}

\textbf{Step 4: Conclude.}
Combining~\eqref{eq:quantilescrk_convergence_condition_pf4} and~\eqref{eq:quantilescrk_convergence_condition_pf3} with the bound on $\gamma_q$ shows that the expected relative improvement $\mathbb{E}_k \norm{\mathbf{x}^{k+1} - \mathbf{x}^*}^2 / \norm{\mathbf{x}^{k} - \mathbf{x}^*}^2$ is upper bounded by
\begin{align} \label{eq:quantilescrk_convergence_condition_pf5}
    &1 - \frac{1}{Z_J} \left( \sigma_{q - \beta, \mathrm{min}}^+(\mathbf{A}_{I_1} \mathbf{P})^2
    - \frac{|S| \gamma_q^2}{\norm{\mathbf{x}^{k} - \mathbf{x}^*}^2}  - \frac{2\gamma_q \sqrt{|S|} \, \norm{(\mathbf{A}_{I_1} \mathbf{P})_C (\mathbf{x}^k - \mathbf{x}^*)}}{\norm{\mathbf{x}^{k} - \mathbf{x}^*}^2} \right) \nonumber \\
    &\leq 1 - \frac{1}{Z_J} \left( \sigma_{q - \beta, \mathrm{min}}^+(\mathbf{A}_{I_1} \mathbf{P})^2
    - \left[ \frac{|S| \cdot \sigma_{\mathrm{max}}(\mathbf{A}_{I_1} \mathbf{P})^2}{\theta + |S|}
    + \frac{2 \sqrt{|S|} \cdot \sigma_{\mathrm{max}}(\mathbf{A}_{I_1} \mathbf{P})^2}{\sqrt{\theta + |S|}} \right] \right),
 \end{align}
where $\theta \coloneqq (m - m_0)(1 - q - \beta)$. Now, we can upper bound $Z_J$ by $Z_{q - \beta}$ from~\eqref{eq:defn_uniform_max_frobnorm}. Next, consider the function $f(x) = \frac{x}{\theta + x} + \frac{2 \sqrt{x}}{\sqrt{\theta + x}}$, and observe that $f(|S|)$ appears in the upper bound~\eqref{eq:quantilescrk_convergence_condition_pf5}. Since $f'(x) > 0$ for all $x > 0$, the upper bound is increasing in $|S|$. Because $|S| \leq \beta(m - m_0)$, we conclude that the most pessimistic bound, independent of $|S|$ and $J$ (and hence $k$), is obtained by setting $|S| = \beta(m - m_0)$, which implies that
\begin{equation*} \label{eq:quantilescrk_convergence_condition_pf6}
    \mathbb{E}_k \norm{\mathbf{x}^{k+1} - \mathbf{x}^*}^2
    \leq (1 - C_{q, \beta}) \cdot \norm{\mathbf{x}^k - \mathbf{x}^*}^2 , \quad \text{ where $C_{q, \beta}$ is defined in~\eqref{eq:quantilescrk_convergence_condition_1}.}
\end{equation*}
To ensure that the mean squared error contracts after each step, it suffices for $C_{q, \beta}$ to be positive: this is exactly secured by the condition~\eqref{eq:quantilescrk_convergence_condition_3}. By iterating, we obtain~\eqref{eq:quantilescrk_convergence_condition_4}.
\end{proof}

Note that Lemma~\ref{lem:quantilescrk_convergence_condition} only provides a sufficient condition for convergence in the worst case (see~\cite{Steinerberger2023} for further discussion). Empirically, convergence is observed for larger values of $\beta$ because the corruptions are quickly detected and trapped beyond the threshold. The dependence on $|S|$ in~\eqref{eq:quantilescrk_convergence_condition_pf5} shows that if the number of admissible, corrupted equations is small, then far less is demanded of the spectral quantities of $\mathbf{A}_{I_1} \mathbf{P}$ for the mean squared error to contract. For similar reasons, the QuantileRK method also empirically outperforms currently available theoretical convergence guarantees~\cite{HaNeReSw2022, ChJaNeRe2022}.

\subsection{Proof of Theorem~\ref{thm:quantilescrk_convergence}}
\label{subsec:quantilescrk_convergence_proof}

To prove Theorem~\ref{thm:quantilescrk_convergence}, our strategy will be to show that the ratio of $\sigma_{\mathrm{max}}(\mathbf{A}_{I_1} \mathbf{P})$ and $\sigma_{q - \beta, \mathrm{min}}^+(\mathbf{A}_{I_1} \mathbf{P})$ is of the same order with high probability. Together with the condition~\eqref{eq:quantilescrk_convergence_condition_3} for convergence in Lemma~\ref{lem:quantilescrk_convergence_condition}, this implies that the QuantileSCRK method will efficiently converge if the proportion of corruptions is small enough.

First, we show that $\sigma_{q - \beta, \mathrm{min}}^+(\mathbf{A}_{I_1} \mathbf{P})$ can be lower bounded with high probability as long as the effective aspect ratio $(m - m_0) / (n - m_0)$ is tall enough. This is proved using a similar technique as~\cite[Proposition~3.4]{HaNeReSw2022}.

\begin{lemma} \label{lem:subgauss_uniform_min_singval}
Let $\mathbf{P}$ be an orthogonal projection onto a fixed $(n - m_0)$-dimensional subspace, $\alpha \in (0, 1]$, and $\mathbf{A}_{I_1} \in \reals^{(m - m_0) \times n}$ be a random matrix that satisfies Assumptions~\ref{asmp:A1} and~\ref{asmp:A2}. Then there exist absolute constants $C, \theta > 0$ such that if
\begin{equation} \label{eq:subgauss_uniform_min_singval_1}
    \frac{m - m_0}{n - m_0} \geq \frac{24}{\alpha} \log \left( \frac{36 \theta (1 + C K^2)}{\alpha^{3/2}} \right) ,
\end{equation}
then with probability at least $1 - 3e^{-\alpha (m - m_0) / 24}$,
\begin{equation} \label{eq:subgauss_uniform_min_singval_2}
    \inf_{\substack{T \subseteq I_1 \\ |T| = \alpha (m - m_0)}} \sigma_{\mathrm{min}}^+((\mathbf{A}_{I_1} \mathbf{P})_T) \geq \frac{\alpha^{3/2}}{32} \sqrt{m - m_0} .
\end{equation}
\end{lemma}

\begin{proof}
Recall that $\mathbb{S}^{n-1} = \{ \mathbf{x} \in \reals^n : \norm{\mathbf{x}} = 1 \}$ denotes the unit sphere. By Lemma~\ref{lem:subspace_unitsphere_covering}, we can fix an $\epsilon$-net $\mathcal{N}$ of $\Range(\mathbf{P}) \cap \mathbb{S}^{n-1}$ with cardinality $|\mathcal{N}| \leq (3 / \epsilon)^{n - m_0}$ for some $\epsilon \in (0, 1]$ to be chosen later. Fix any $T \subseteq I_1$ with $|T| = \alpha (m - m_0)$. Since for any $\mathbf{z} \in \Range(\mathbf{P}) \cap \mathbb{S}^{n-1}$, there exists $\mathbf{x} \in \mathcal{N}$ such that $\norm{\mathbf{z} - \mathbf{x}} \leq \epsilon$, using the reverse triangle inequality and $\norm{(\mathbf{A}_{I_1} \mathbf{P})_T} \leq \norm{\mathbf{A}_{I_1} \mathbf{P}}$ implies that
\begin{equation} \label{eq:subgauss_uniform_min_singval_pf1}
    \sigma_{\mathrm{min}}^+((\mathbf{A}_{I_1} \mathbf{P})_T)
    \geq \inf_{\mathbf{z} \in \Range(\mathbf{P}) \cap \mathbb{S}^{n-1}} \norm{(\mathbf{A}_{I_1} \mathbf{P})_T) \mathbf{z}}
    \geq \inf_{\mathbf{x} \in \mathcal{N}} \norm{(\mathbf{A}_{I_1} \mathbf{P})_T) \mathbf{z}} - \epsilon \norm{\mathbf{A}_{I_1} \mathbf{P}} .
\end{equation}
Firstly, by Lemma~\ref{lem:subgauss_minmax_singvals} (with $s = 1$), we have that with probability at least $1 - 2e^{-(m - m_0)}$,
\begin{equation} \label{eq:subgauss_uniform_min_singval_pf2}
    \norm{\mathbf{A}_{I_1} \mathbf{P}} \leq (1 + CK^2) \sqrt{m - m_0} .
\end{equation}
Next, our goal is to define an event $\mathcal{E}$ on which a good bound for $\inf_{\mathbf{x} \in \mathcal{N}} \norm{(\mathbf{A}_{I_1} \mathbf{P})_T) \mathbf{z}}$ that is independent of $T$ holds. More precisely, for every $j \in I_1$ and $\mathbf{x} \in \mathcal{N}$, define the ``bad'' event
\[
    \mathcal{E}_j^{\mathbf{x}} \coloneqq \left\{ |\innprod{\mathbf{a}_j}{\mathbf{x}}| \leq \alpha / (4 \theta) \right\} ,
\]
where $\theta$ is some constant to be specified later. Let $\mathcal{E}^{\mathbf{x}}$ be the ``good'' event where less than $\alpha (m - m_0) / 2$ of the events $(\mathcal{E}_j^{\mathbf{x}})_{j \in I_1}$ occur, and $\mathcal{E} \coloneqq \bigcap_{\mathbf{x} \in \mathcal{N}} \mathcal{E}^{\mathbf{x}}$. Observe that $\innprod{\mathbf{P} \mathbf{a}_j}{\mathbf{x}} = \innprod{\mathbf{a}_j}{\mathbf{P} \mathbf{x}} = \innprod{\mathbf{a}_j}{\mathbf{x}}$ since $\mathbf{x} \in \Range(\mathbf{P})$. Therefore, on $\mathcal{E}$, at least half of the rows of $(\mathbf{A}_{I_1} \mathbf{P})_T$ have nontrivial correlation with any $\mathbf{x} \in \mathcal{N}$, which implies that
\begin{equation} \label{eq:subgauss_uniform_min_singval_pf3}
    \inf_{\mathbf{x} \in \mathcal{N}} \norm{(\mathbf{A}_{I_1} \mathbf{P})_T \mathbf{x}}
    = \inf_{\mathbf{x} \in \mathcal{N}} \sqrt{\sum_{j \in T} |\innprod{\mathbf{P} \mathbf{a}_j}{\mathbf{x}}|^2}
    \geq \sqrt{\frac{\alpha (m - m_0)}{2} \cdot \frac{\alpha^2}{16 \theta^2}}
    \geq \frac{\alpha^{3/2}}{6 \theta} \sqrt{m - m_0} .
\end{equation}
To balance~\eqref{eq:subgauss_uniform_min_singval_pf2} and~\eqref{eq:subgauss_uniform_min_singval_pf3}, we choose $\epsilon = \alpha^{3/2} / (12 \theta (1 + CK^2))$. Therefore, if both events $\mathcal{E}$ and~\eqref{eq:subgauss_uniform_min_singval_pf2} hold, then~\eqref{eq:subgauss_uniform_min_singval_pf1} implies that the desired bound~\eqref{eq:subgauss_uniform_min_singval_2} holds.

It remains to bound the probability of the event $\mathcal{E}$, for which we will combine an anti-concentration result with a Chernoff bound.
By using either~\cite[Theorem~1.2]{RudelsonVershynin2015} if the row $\mathbf{a}_j$ has independent entries with bounded densities (with $\theta \geq 2 \sqrt{2}$), or~\cite[Theorem~8]{CarberyWright2001} if $\mathbf{a}_j$ has a log-concave distribution (increasing the value of $\theta$ based on the absolute constant in this result), we deduce that $\prob{\mathcal{E}_j^{\mathbf{x}}} \leq \alpha / 4$ for all $j \in I_1$. Hence, a standard Chernoff bound implies that $\prob{\mathcal{E}^{\mathbf{x}}} \geq 1 - e^{-\alpha (m - m_0) / 12}$ for all $\mathbf{x} \in \mathcal{N}$, and a union bound shows that $\mathcal{E}$ fails to hold with probability less than
\[
    |\mathcal{N}| \cdot \exp\left( \frac{-\alpha (m - m_0)}{12} \right)
    \leq \exp\left( (n - m_0) \log\left( \frac{3}{\epsilon} \right) - \frac{\alpha(m - m_0)}{12} \right)
    \leq e^{-\alpha(m - m_0) / 24} ,
\]
where the condition~\eqref{eq:subgauss_uniform_min_singval_1} is used for the final inequality. Combining this with the probability bound for~\eqref{eq:subgauss_uniform_min_singval_pf2} to hold completes the proof.
\end{proof}

Next, the following lemma bounds $Z_{q - \beta}$ from above with high probability.

\begin{lemma} \label{lem:uniform_max_frobnorm}
Let $\mathbf{P}$ be an orthogonal projection onto a fixed $(n - m_0)$-dimensional subspace, $\alpha \in (0, 1]$, and $\mathbf{A}_{I_1} \in \reals^{(m - m_0) \times n}$ be a random matrix that satisfies Assumptions~\ref{asmp:A1} and~\ref{asmp:A2}. Then there exists an absolute constant $c > 0$ such that with probability at least $1 - e^{-c \alpha (m - m_0) / K^4}$,
\begin{equation} \label{eq:uniform_max_frobnorm_1}
    \sup_{\substack{T \subseteq I_1 \\ |T| = \alpha(m - m_0)}} \norm{(\mathbf{A}_{I_1} \mathbf{P})_T}_F^2
    \leq \left( 2 + \frac{K^4}{c} \log\left( \frac{e}{\alpha} \right) \right) \alpha (m - m_0)(n - m_0) .
\end{equation}
\end{lemma}

\begin{proof}
For all fixed $T \subseteq I_1$ with $|T| = \alpha(m - m_0)$, Lemma~\ref{lem:projected_frobenius_concentration} applied to the submatrix $(\mathbf{A}_{I_1} \mathbf{P})_T$ implies that there exists an absolute constant $c > 0$ such that for all $\epsilon > 0$,
\[
    \prob{\norm{(\mathbf{A}_{I_1} \mathbf{P})_T}_F^2 \geq (1 + \epsilon) \alpha (m - m_0)(n - m_0)} \leq e^{-c \min \{ \epsilon^2 , \epsilon \} \alpha (m - m_0) / K^4} .
\]
Hence, by a union bound over all $\binom{m - m_0}{\alpha(m - m_0)} < e^{\alpha (m - m_0) \log(e / \alpha)}$ such subsets $T$, we deduce that the probability that the event~\eqref{eq:uniform_max_frobnorm_1} does not hold is not greater than $\exp\left\{ -\alpha(m - m_0) \left( \frac{c \epsilon}{K^4} - \log\left( \frac{e}{\alpha} \right) \right) \right\}$ for $\epsilon \geq 1$.
In particular, choosing $\epsilon = 1 + \frac{K^4}{c} \log\left( \frac{e}{\alpha} \right)$ leads to the claimed probability guarantee.
\end{proof}

By combining our tail bounds for $\sigma_{\mathrm{max}}(\mathbf{A}_{I_1} \mathbf{P})$ and $\sigma_{q - \beta, \mathrm{min}}^+(\mathbf{A}_{I_1} \mathbf{P})$ as well as $Z_{q - \beta}$, we can now prove Theorem~\ref{thm:quantilescrk_convergence}.

\begin{proof}[Proof of Theorem~\ref{thm:quantilescrk_convergence}]
In this proof, the various constants of the form $c_1, C_1, \dots$ that appear only depend on $K$. By Lemma~\ref{lem:subgauss_minmax_singvals}, $\sigma_{\mathrm{max}}(\mathbf{A}_{I_1} \mathbf{P}) \leq C_1 \sqrt{m - m_0}$ with probability at least $1 - 2e^{-c_1 (m - m_0)}$.
By Lemma~\ref{lem:subgauss_uniform_min_singval}, $\sigma_{q - \beta, \mathrm{min}}^+(\mathbf{A}_{I_1} \mathbf{P}) \geq C_2 (q - \beta)^{3/2} \sqrt{m - m_0}$ with probability at least $1 - 3e^{-c_2 (q - \beta) (m - m_0)}$, given that the condition~\eqref{eq:quantilescrk_convergence_1} is satisfied.
Therefore, if both of these events hold, then
\[
    \frac{\sigma^+_{q - \beta, \mathrm{min}}(\mathbf{A}_{I_1} \mathbf{P})^2}{\sigma_{\mathrm{max}}(\mathbf{A}_{I_1} \mathbf{P})^2} \geq \left( \frac{C_2}{C_1} \right)^2 (q - \beta)^3 .
\]
Hence, by Lemma~\ref{lem:quantilescrk_convergence_condition} we deduce that the QuantileSCRK algorithm converges if
\[
    \left( \frac{C_2}{C_1} \right)^2 (q - \beta)^3 > \frac{\beta}{1 - q} + 2 \sqrt{\frac{\beta}{1 - q}} .
\]
Since $q$ is fixed and the right-hand side can be made arbitrarily small by decreasing $\beta$, it follows that this condition is satisfied as long as $\beta$ is sufficiently small. Finally, Lemma~\ref{lem:uniform_max_frobnorm} implies that $Z_{q - \beta} \leq C_{q, \beta} (m - m_0)(n - m_0)$ with probability at least $1 - e^{-c_3 (q - \beta) (m - m_0)}$ for some constant $C_{q, \beta} > 0$ that only depends on $q$, $\beta$, and $K$. If all of these events hold, then Lemma~\ref{lem:quantilescrk_convergence_condition} implies that QuantileSCRK converges with the claimed rate~\eqref{eq:quantilescrk_convergence_3}. The proof is completed after simplifying the probability bound.
\end{proof}

\section{Numerical experiments} \label{sec:experiments}

In this section, we present numerical experiments that demonstrate various features of the SCRK method (Algorithm~\ref{alg:scrk_method}) and the QuantileSCRK method (Algorithm~\ref{alg:quantilescrk_method}). For the plots with random, simulated data that follow, the lines represent the median over 200 trials, and the shaded regions indicate the 0.1- and 0.9-quantiles around the corresponding medians. The log relative error refers to the quantity $\log(\norm{\mathbf{x}^k - \mathbf{x}^*} / \norm{\mathbf{x}^0 - \mathbf{x}^*})$. The experiments were performed on a MacBook Air M1 with 8GB RAM using Python 3.11.

\subsection{SCRK method for systems with correlated rows}

In \textbf{Figure~\ref{fig:scrk_varym0_coherent_uniform}}, we compare the performance of the SCRK method on a system with highly correlated rows for various sizes $m_0$ of $I_0$. It is known that RK performs poorly in this setting~\cite{NeedellTropp2014, NeedellWard2013}. The entries of $\mathbf{A} \in \reals^{2,000 \times 1,000}$ are independently and uniformly distributed on $[0.9, 1.1]$, and the solution $\mathbf{x}^* \in \reals^{1,000}$ is a standard Gaussian vector. The same initial iterate starting in the solution space corresponding to the biggest block (i.e.\ $m_0 = 200$) is used for each variation.

As predicted by Corollary~\ref{cor:scrk_coherence}, the SCRK method with $I_0$ as the first $m_0$ rows of $\mathbf{A}$ outperforms RK for any $m_0 \geq 1$ since the pairwise row correlations of $\mathbf{A}$ are bounded from below. Moreover, increasing $m_0$ increases the rate of convergence (see Theorem~\ref{thm:subgauss_scrk_convergence}). However, since increasing $m_0$ leads to heavier iterations and a higher initial cost from computing $\mathbf{A}_{I_0}^{\dagger}$ (see Remark~\ref{rmk:complexity}), the optimal block size for a given target error and time budget is not necessarily the largest as highlighted by Figure~\ref{fig:scrk_varym0_coherent_uniform} (right).


\begin{figure}[!htb]
    \centering
    \begin{subfigure}[b]{0.505\textwidth}
        \centering
        \includegraphics[width=\textwidth]{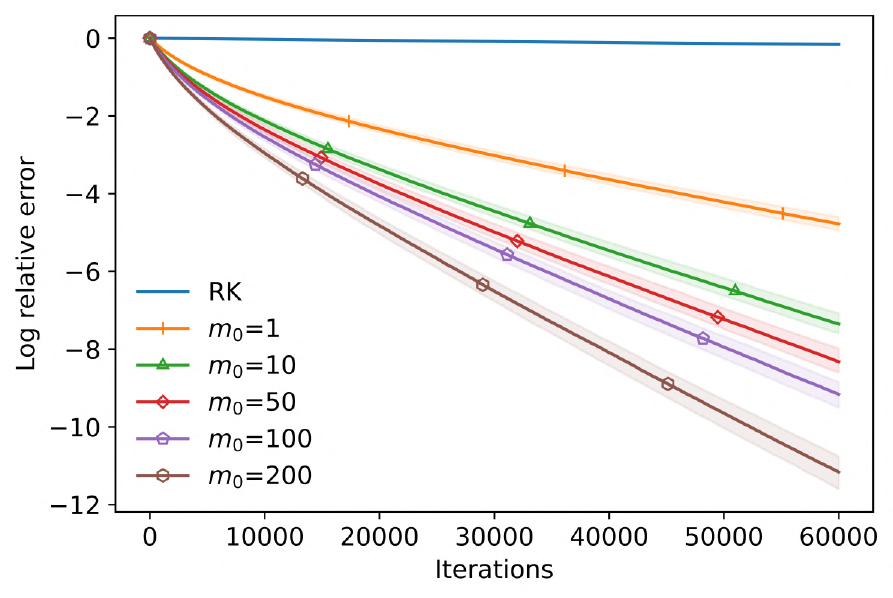}
    \end{subfigure}
    \hfill
    \begin{subfigure}[b]{0.485\textwidth}
        \centering
        \includegraphics[width=\textwidth, trim={0.7cm 0 0 0}, clip]{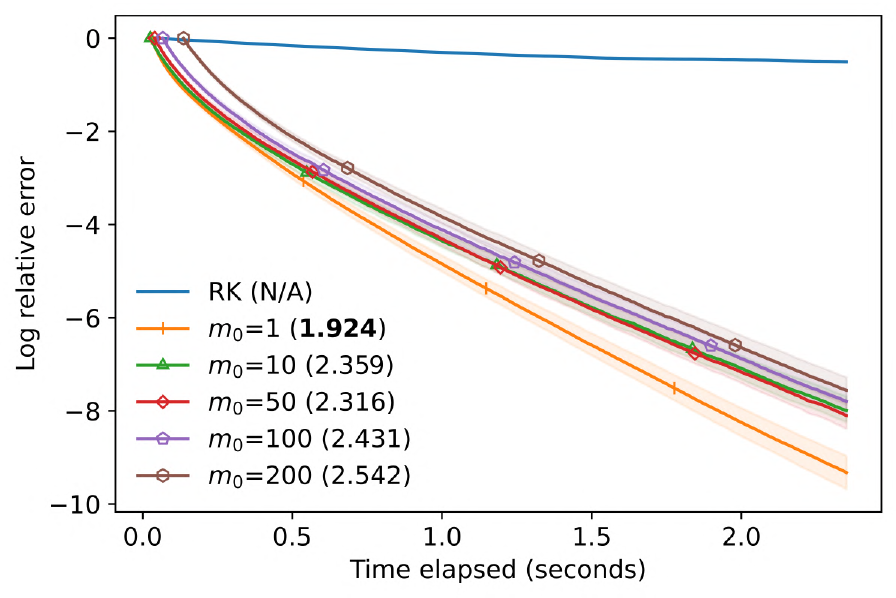}
    \end{subfigure}
    \caption{
    Performance of SCRK on a system with highly correlated rows for various sizes $m_0$ of $\mathbf{A}_{I_0}$.
    \textit{(Left)} Log relative error at each iteration.
    \textit{(Right)} Log relative error against time elapsed, including the initial cost of precomputing $\mathbf{A}_{I_0}^{\dagger}$ for each $m_0$. The time taken to reach a log relative error of less than $-8$ is reported in brackets (N/A indicates that this was not reached in 30 seconds).
    }
    \label{fig:scrk_varym0_coherent_uniform}
\end{figure}

\begin{figure}[!htb]
    \centering
    \begin{subfigure}[b]{0.505\textwidth}
        \centering
        \includegraphics[width=\textwidth]{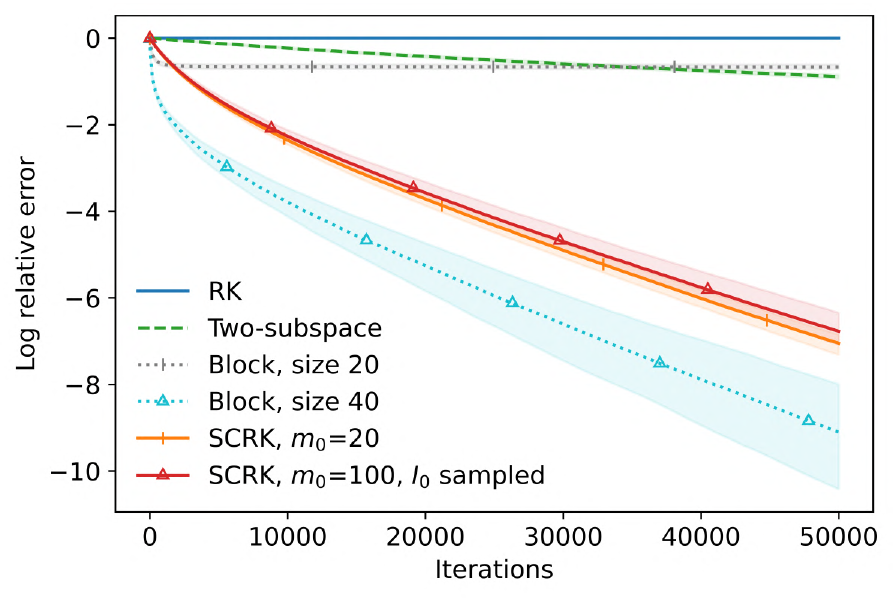}
    \end{subfigure}
    \hfill
    \begin{subfigure}[b]{0.485\textwidth}
        \centering
        \includegraphics[width=\textwidth, trim={0.7cm 0 0 0}, clip]{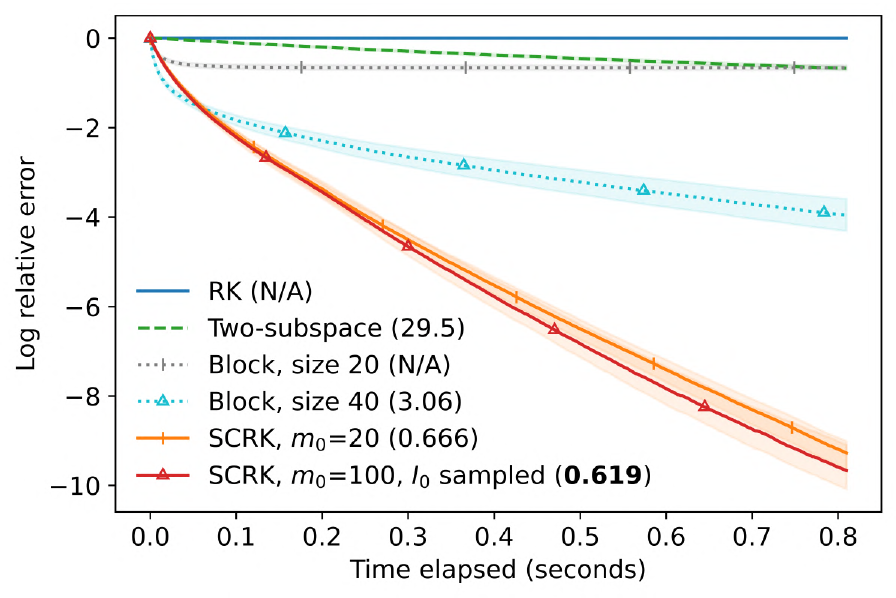}
    \end{subfigure}
    \caption{
    Performance of SCRK on a coherent system with low-rank structure using a ``perfect'' block (with $m_0 = 20$) and a randomly sampled block (with $m_0 = 100$) as described in the main text. The two-subspace Kaczmarz method~\cite{NeedellWard2013} and randomized block Kaczmarz method~\cite{NeedellTropp2014} (with two block sizes) are also included.
    \textit{(Left)} Log relative error at each iteration.
    \textit{(Right)} Log relative error against time elapsed, not including the initial costs of precomputing pseudoinverses for SCRK and block Kaczmarz. The time taken to reach a log relative error of less than $-8$ is reported in the brackets (N/A indicates that this was not reached in 30 seconds).
    }
    \label{fig:scrk_super_coherence}
\end{figure}

\subsection{SCRK method for systems with low-rank structure}

In \textbf{Figure~\ref{fig:scrk_super_coherence}}, we consider the performance of the SCRK method on a structured matrix $\mathbf{A} \in \reals^{2,000 \times 1,000}$, constructed as in Example~\ref{eg:scrk_structure_frobnorm}. The first $r = 20$ rows of $\mathbf{A}$ are normalized standard Gaussian vectors. The remaining $m - r$ rows $(\mathbf{a}_j)_{j > r}$ are equal to $\mathbf{a}_j \coloneqq (1 - \epsilon) \mathbf{a}'_j + \epsilon \mathbf{c}_j$, where $\epsilon = 0.1$, $\mathbf{a}'_j$ is sampled from $\{ \mathbf{a}_1, \dots, \mathbf{a}_r \}$, and $\mathbf{c}_j$ is sampled from $\mathrm{span}(\{ \mathbf{a}_1, \dots, \mathbf{a}_r \})^{\perp}$ and normalized; i.e.\ $\mathbf{a}_j$ mainly consists of a row from the special top block plus some noise in the orthogonal direction. The solution $\mathbf{x}^* \in \reals^{1,000}$ is a standard Gaussian vector.

The SCRK algorithm is run with two choices of $I_0$: the first uses the ``perfect'' block of size $m_0 = 20$ with the rows $\{ \mathbf{a}_1, \dots, \mathbf{a}_r \}$ that generate the coherence structure. The second variant uses a block of $m_0 = 5r = 100$ rows of $\mathbf{A}$ sampled (without replacement) uniformly at random. This represents the case where the source of coherence is unknown, but sampling the effective rank of $\mathbf{A}$ (with an appropriate oversampling factor) should find a good block $\mathbf{A}_{I_0}$ as predicted by Theorem~\ref{thm:sampling_guarantees}.
Indeed, Figure~\ref{fig:scrk_super_coherence} shows that both choices of $I_0$ converge effectively: the dramatic improvement in the per-iteration convergence rate of SCRK over RK shown by the left plot is explained by the (inverse) scaled condition number $\sigma_{\mathrm{min}}^+(\mathbf{A}_{I_1} \mathbf{P}) / \norm{\mathbf{A}_{I_1} \mathbf{P}}_F = 9.33 \times 10^{-3}$ of $\mathbf{A}_{I_1} \mathbf{P}$ with $m_0 = 20$ (and similarly $9.56 \times 10^{-3}$ with $m_0 = 100$) being significantly larger than the corresponding quantity $\sigma_{\mathrm{min}}(\mathbf{A}) / \norm{\mathbf{A}}_F = 3.29 \times 10^{-5}$ for $\mathbf{A}$ (see Section~\ref{subsec:scrk_structure_discussion}).

The two-subspace Kaczmarz method~\cite{NeedellWard2013} and randomized block Kaczmarz method \cite{NeedellTropp2014} (using equally-sized blocks of size 20 and 40 chosen uniformly at random, and precomputed pseudoinverses) are also included. The same initial iterate as SCRK with $m_0 = 20$ is used. It is known that these algorithms perform well in systems with highly correlated rows, such as the one previously considered in Figure~\ref{fig:scrk_varym0_coherent_uniform}. However, Figure~\ref{fig:scrk_super_coherence} shows that the effectiveness of two-subspace Kaczmarz and block Kaczmarz with blocks of size 20 that are ``too small'' is impeded by the coherence structure of $\mathbf{A}$.

On the other hand, block Kaczmarz with blocks of size 40 that are ``large enough'' (relative to $r = 20$ for this problem) converges effectively. While Figure~\ref{fig:scrk_super_coherence} (left) shows that it converges with a greater per-iteration rate than SCRK (since it effectively uses 40 new rows in each iteration instead of just one), Figure~\ref{fig:scrk_super_coherence} (right) shows that the lighter iterations of the SCRK method actually makes it more efficient on a time basis.

\subsection{SCRK method for noisy systems}

In \textbf{Figure~\ref{fig:scrk_noisy_horizon}}, we consider the performance of the SCRK algorithm on a noisy system to demonstrate the validity of the error horizon predicted by Theorem~\ref{thm:noisy_scrk_convergence}. The rows of $\mathbf{A} \in \reals^{300 \times 100}$ are independent normalized standard Gaussian vectors, the solution $\mathbf{x}^* \in \reals^{100}$ is a standard Gaussian vector, and the entries of the noise vector $\mathbf{r}$ are independently and uniformly distributed on $[-0.01, 0.01]$.

\begin{figure}[!htb]
    \centering
    \includegraphics[width=0.495\textwidth]{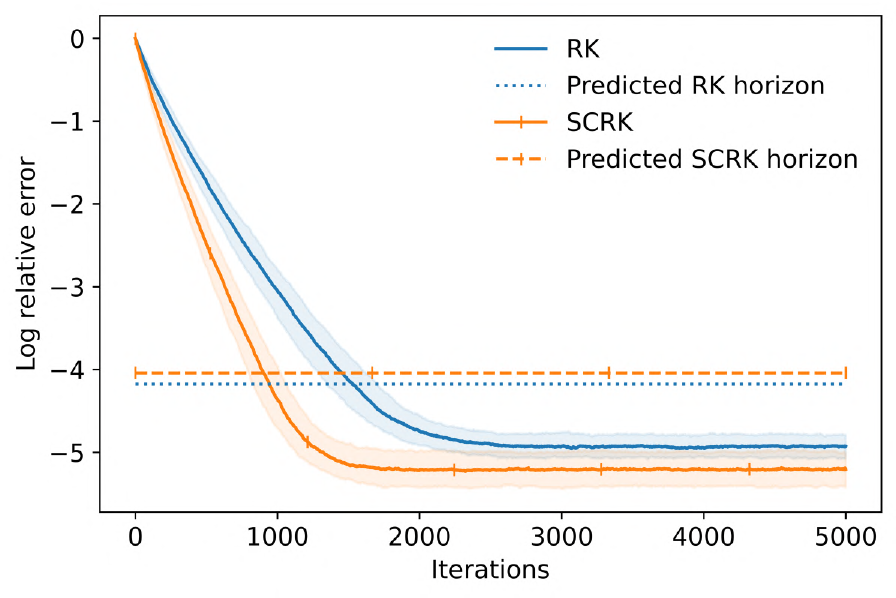}
    \caption{
    Convergence paths for the SCRK (with $I_0$ equal to the first $m_0 = 25$ rows) and RK methods on a noisy system. The dashed/dotted lines indicate the predicted error horizons $\gamma_0 + \gamma_1$ from Theorem~\ref{thm:noisy_scrk_convergence} and $\gamma = \norm{\mathbf{r}}^2 / \sigma_{\mathrm{min}}(\mathbf{A})^2$ from~\cite{Needell2010} respectively.
    }
    \label{fig:scrk_noisy_horizon}
\end{figure}

\subsection{QuantileSCRK algorithm} \label{subsec:experiments_quantilescrk}

\begin{figure}[!htb]
    \centering
    \begin{subfigure}[b]{\textwidth}
    \centering
    \begin{subfigure}[b]{0.505\textwidth}
        \centering
        \includegraphics[width=\textwidth]{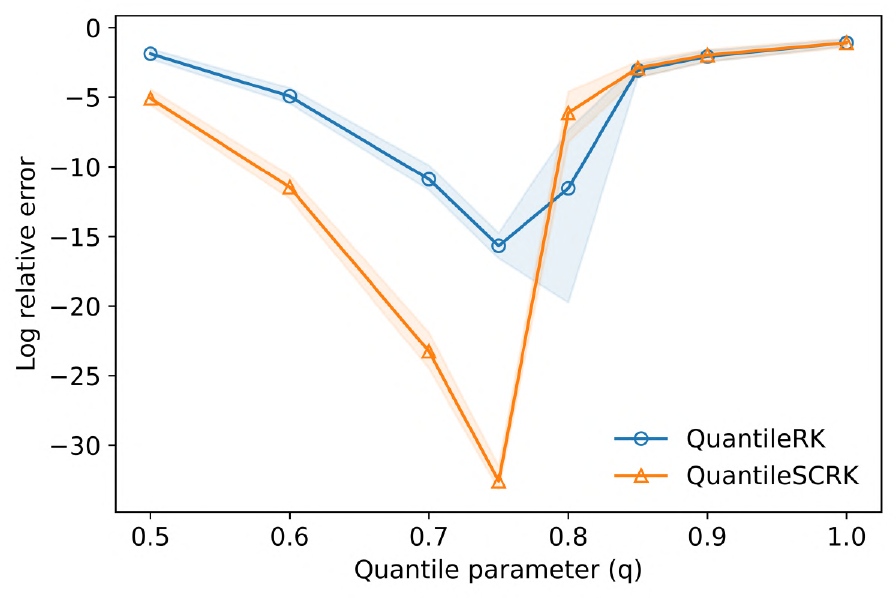}
    \end{subfigure}
    \hfill
    \begin{subfigure}[b]{0.485\textwidth}
        \centering
        \includegraphics[width=\textwidth, trim={0.7cm 0 0 0}, clip]{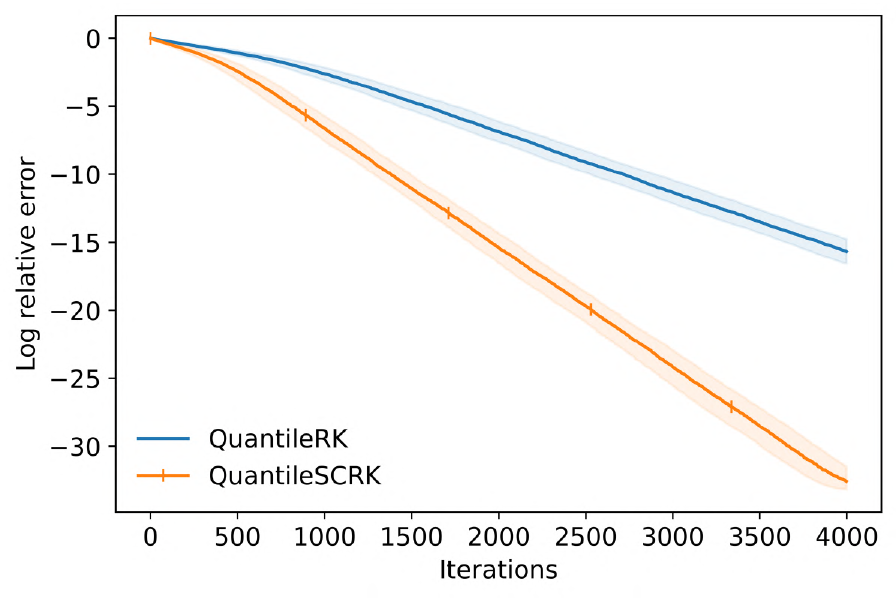}
    \end{subfigure}
    \subcaption{
    Tall system $\mathbf{A}^{500 \times 50}$ with $c = 100$, $m_0 = 20$, $k = 4,000$, and $q_{\mathrm{RK}} = q_{\mathrm{SCRK}} = 0.75$.
    }
    \label{fig:quantilescrk_tall}
    \end{subfigure}

    \begin{subfigure}[b]{\textwidth}
    \centering
    \begin{subfigure}[b]{0.505\textwidth}
        \centering
        \includegraphics[width=\textwidth]{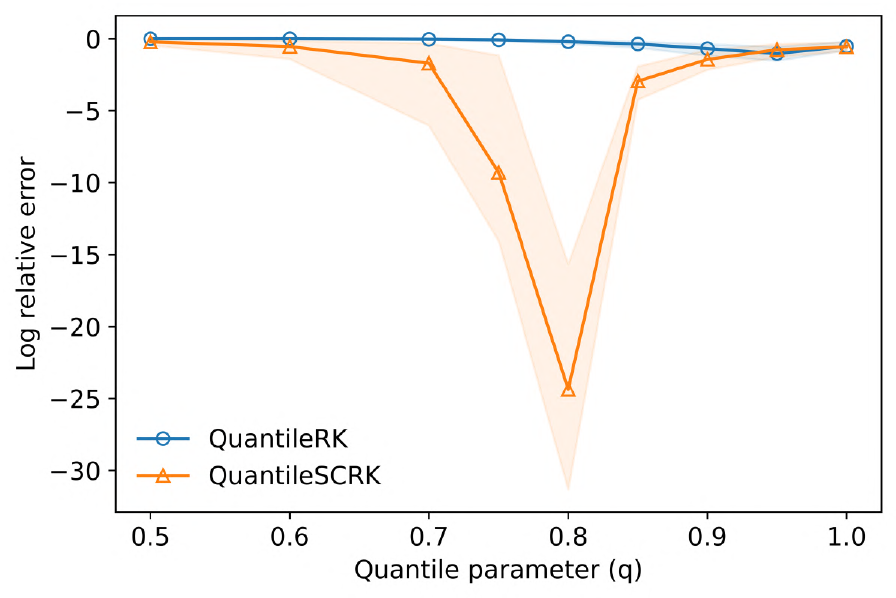}
    \end{subfigure}
    \hfill
    \begin{subfigure}[b]{0.485\textwidth}
        \centering
        \includegraphics[width=\textwidth, trim={0.7cm 0 0 0}, clip]{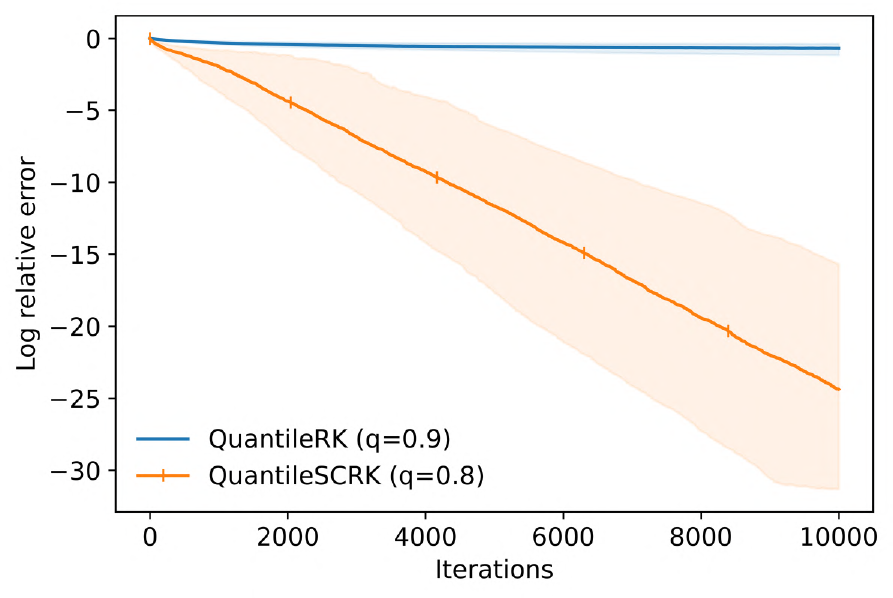}
    \end{subfigure}
    \subcaption{
    Almost-square systems $\mathbf{A}^{130 \times 100}$ with $c = 10$, $m_0 = 75$, $k = 10,000$, $q_{\mathrm{RK}} = 0.9$, $q_{\mathrm{SCRK}} = 0.8$.
    }
    \label{fig:quantilescrk_square}
    \end{subfigure}

    \caption{
    Performance of the QuantileSCRK method, given a corruption-free block of size $m_0$, compared to the QuantileRK method~\cite{HaNeReSw2022} on Gaussian systems with different aspect ratios and $c$ corrupted measurements.
    \textit{(Left)} Log relative error after $k$ iterations for various values of the quantile parameter $q$.
    \textit{(Right)} Convergence paths using the best quantile parameters $q_{\mathrm{RK}}$ and $q_{\mathrm{SCRK}}$.
    }
    \label{fig:quantilescrk_tall_square}
\end{figure}

In \textbf{Figure~\ref{fig:quantilescrk_tall_square}}, we compare the performance of the QuantileSCRK and QuantileRK~\cite{HaNeReSw2022} methods on Gaussian systems $\mathbf{A}$ with different aspect ratios, where the measurements are corrupted by a sparse vector with $c$ non-zero entries independently and uniformly distributed on $[-1, 1]$. The rows of $\mathbf{A}$ are independent normalized standard Gaussian vectors, and the solution $\mathbf{x}^*$ is a standard Gaussian vector.

Tall systems are considered in \textbf{Figure~\ref{fig:quantilescrk_tall}}, where $100 / 500 = 20\%$ (resp. $100 / 480 \approx 20.8\%$) of the rows of $\mathbf{A}$ (resp. $\mathbf{A}_{I_1}$) correspond to corrupted measurements. These plots replicate the finding that the QuantileRK method converges effectively for tall, Gaussian-like matrices even in the presence of numerous corruptions~\cite{HaNeReSw2022}, and also show that exploiting information about corruption-free measurements using the QuantileSCRK method accelerates convergence (see Theorem~\ref{thm:quantilescrk_convergence}).

Almost-square systems are considered in \textbf{Figure~\ref{fig:quantilescrk_square}}, where $10 / 130 \approx 7.7\%$ (resp. $10 / 55 \approx 18.2\%$) of the rows of $\mathbf{A}$ (resp. $\mathbf{A}_{I_1}$) correspond to corrupted measurements. It is clear that the QuantileRK method is unable to make any progress in this setting. On the other hand, the QuantileSCRK method converges for $q$ around $0.8$, which demonstrates that exploiting external knowledge in the form of a large block $\mathbf{A}_{I_0}$ corresponding to corruption-free measurements can enable convergence in such challenging settings.

\subsection{Systems of differential equations with inconsistent initial conditions} \label{subsec:experiments_diffeqns}

We consider the problem of numerically solving a system of differential equations given competing data for the initial conditions as another application of the QuantileSCRK method.
After discretization via a finite difference scheme, two types of equations arise: the first describe the underlying law and can be considered to be known exactly, and the second type encode the initial conditions, which can be obtained from real data with potentially faulty measurements.
Thus, the problem can be viewed as one about detecting and disregarding the ``corrupted'' equations coming from inconsistent initial conditions, given that the majority of the equations of the first type can be ``trusted''.

In \textbf{Figure~\ref{fig:quantilescrk_pde_lines}}, we consider the linear system obtained from discretizing the differential equation $y'' = 0$ for a line illustrate this idea. The top $98 \times 100$ block is a Toeplitz matrix with entries $1, -2, 1$ along the diagonal before normalization, which we take to be $\mathbf{A}_{I_0}$. We consider two sets of initial conditions corresponding to two lines: Line~1 being $y = x$ with 10 initial conditions, and Line~2 being $y = 25 - x/2$ with 5 initial conditions.

The plots show that solving this system using least squares or QuantileRK produces poor solutions. However, using QuantileSCRK with a careful choice of the quantile parameter enables convergence to one line or the other as the algorithm is able to find a set of consistent initial conditions: when $q = 0.65$, QuantileSCRK converges to Line~1 a majority of the time (left), and when $q = 0.3$, QuantileSCRK converges to Line~2 instead (right). We also observed that the initial iterate $\mathbf{x}^0$ has a significant biasing effect on which solution is preferred for convergence.

\begin{figure}[!htb]
    \centering
    \begin{subfigure}[b]{0.495\textwidth}
        \centering
        \includegraphics[width=\textwidth]{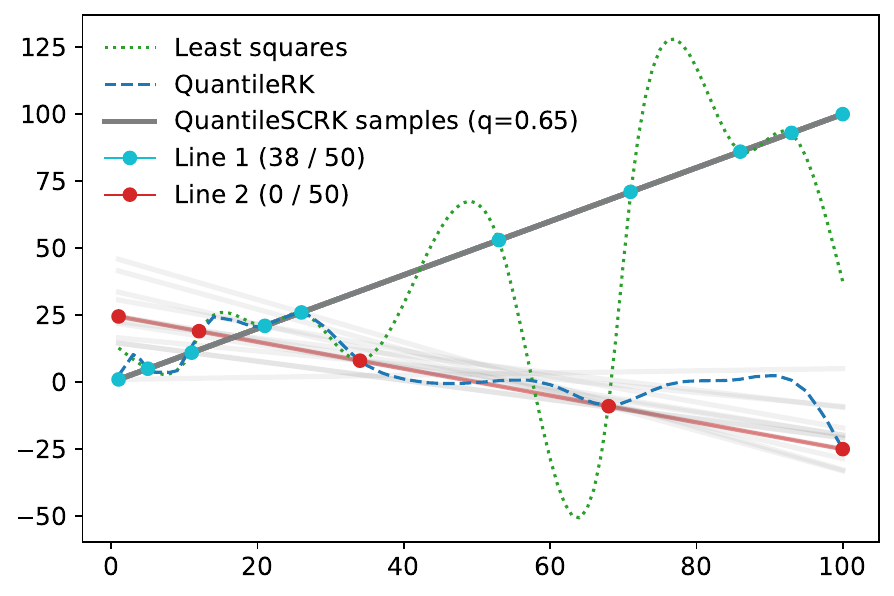}
    \end{subfigure}
    \hfill
    \begin{subfigure}[b]{0.495\textwidth}
        \centering
        \includegraphics[width=\textwidth]{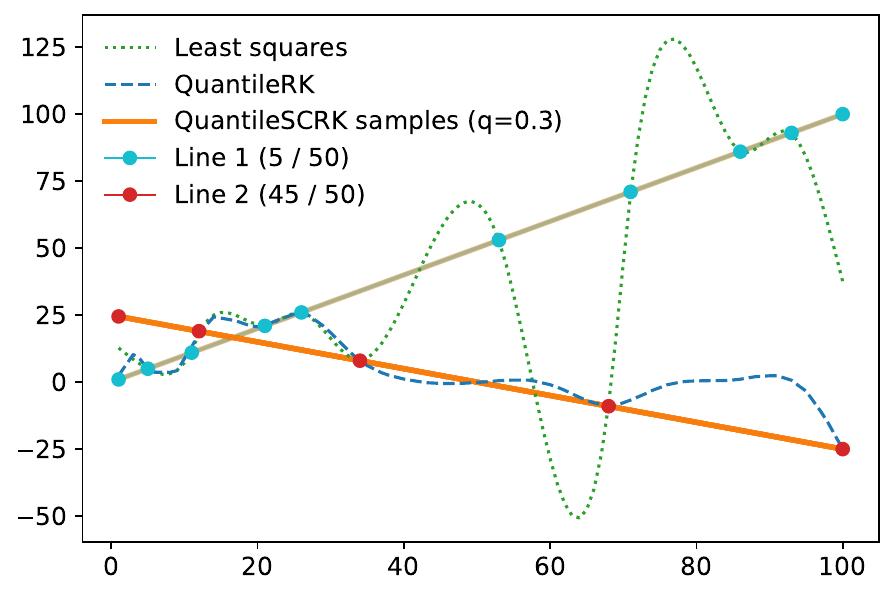}
    \end{subfigure}
    \caption{
    Solving a discretized differential equation for a line in the plane given two sets of inconsistent initial conditions as described in the main text.
    \textit{(Left)} 50 outputs after 10,000 iterations of QuantileSCRK with $q = 0.65$ (translucent grey lines); 38 out of the 50 converged to Line~1 (in the sense $\norm{\mathbf{A}_{\text{Line 1}} \mathbf{x}^k - \mathbf{b}_{\text{Line 1}}}_2 < 10^{-3}$).
    \textit{(Right)} 50 outputs after 10,000 iterations of the QuantileSCRK algorithm with $q = 0.3$ (translucent orange lines); 45 out of the 50 converged to Line~2.
    }
    \label{fig:quantilescrk_pde_lines}
\end{figure}

\subsection{CT image reconstruction} \label{subsec:experiments_ct}

\begin{figure}[!htb]
    \centering
    \includegraphics[width=\textwidth]{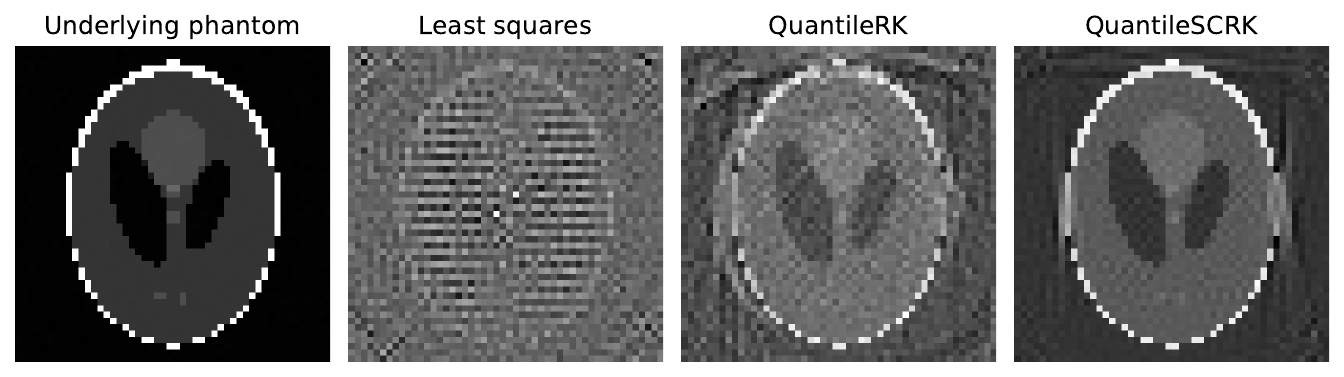}
    \caption{
    Reconstructions of the Shepp-Logan phantom from $\mathbf{A} \in \reals^{4,500 \times 2,500}$ and $\widetilde{\mathbf{b}} \in \reals^{4,500}$ with $c = 1,125$ corruptions as described in the main text. The QuantileSCRK method, given a corruption-free block $\mathbf{A}_{I_0}$ of size $m_0 = 500$, and the QuantileRK method were both run using $q = 0.7$ for $60m = 270,000$ iterations, obtaining a final $\ell_2$ error $\norm{\mathbf{x}^k - \mathbf{x}}_2$ of 3.47 and 6.85 respectively.
    }
    \label{fig:quantilescrk_phantom}
\end{figure}

Finally, we investigate the performance of QuantileSCRK on a realistic dataset. We consider the Shepp-Logan phantom, generated using the Air Tools II package~\cite{HansenJorgensen2018} with parameters $N = 50$ (the image is $N \times N$), $\theta = \{ 0, 2, 4, \dots, 178 \}$ (angles used), and $\rho = 50$ (number of parallel rays). The image is encoded by the measurement matrix $\mathbf{A} \in \reals^{4,500 \times 2,500}$ and measurements $\mathbf{b} \in \reals^{4,500}$. A subset $I_0$ of $m_0 = 500$ rows of $\mathbf{A}$ was randomly chosen to be corruption-free (e.g.\ corresponding to trustworthy measurements), and a random set of $c = 1,125$ of the remaining measurements were corrupted by quantities uniformly distributed in $[2, 6]$ to produce the corrupted measurements $\widetilde{\mathbf{b}}$.

In \textbf{Figure~\ref{fig:quantilescrk_phantom}}, we show various reconstructions given $\mathbf{A}$ and the corrupted measurements $\widetilde{\mathbf{b}}$. It is clear that the least squares solution is very poor. The QuantileRK method, initialized from zero, achieves a noisy reconstruction that does not recover the fine details. Using QuantileSCRK with the corruption-free block $\mathbf{A}_{I_0}$ achieves the best reconstruction, even though a significant proportion (25\%) of the measurements have been corrupted.

\section{Conclusion and future directions} \label{sec:conclusion}

In this paper, we introduced the subspace constrained randomized Kaczmarz (SCRK) method for solving consistent, overdetermined systems of linear equations $\mathbf{A} \mathbf{x} = \mathbf{b}$, which provides a framework for studying the dynamics of the randomized Kaczmarz algorithm when the iterates are confined to a selected solution space $\mathbf{A}_{I_0} \mathbf{x} = \mathbf{b}_{I_0}$. We described the convergence rate of the SCRK method in terms of the spectral properties of the matrix $\mathbf{A}_{I_1} \mathbf{P}$, where $\mathbf{P}$ is the orthogonal projector onto $\Null(\mathbf{A}_{I_0})$. We also demonstrated, both theoretically and empirically, how the SCRK method can exploit approximately low-rank structure to accelerate convergence.

We also proposed the QuantileSCRK method for solving corrupted linear systems, which is able to exploit external knowledge about corruption-free subsystems. In addition to theoretical convergence analysis, we demonstrated numerically that it is able to converge for almost-square corrupted linear systems, where existing iterative methods are ineffective, and that it can be useful for solving differential equations with inconsistent initial conditions and image reconstruction from significantly corrupted measurements.

The framework of subspace constrained iterations raises many possible future directions. For example, since our analysis showed that the SCRK updates simplify to a version of the usual Kaczmarz updates with skewed step directions and the projector $\mathbf{P}$ acts as a right preconditioner for $\mathbf{A}$ to improve the convergence rate, it seems plausible that similar ideas could be applied to related solvers such as the sketch-and-project algorithm~\cite{GowerRichtarik2015} or iterative projection methods for solving systems of linear inequalities~\cite{LeventhalLewis2010, BriskmanNeedell2015}.
It would also be interesting to develop and analyze a QuantileSCRK method in which the trusted solution space is built up adaptively in a data-driven way, based on the information accumulated during the iteration process, which could lead to an effective way for solving corrupted linear systems even in the absence of external knowledge.

\section*{Acknowledgements}

The authors acknowledge partial support by NSF DMS-2309685. ER was also partially funded by NSF DMS-2108479. ER is grateful to Y. Kevrekidis and M. Derezi{\'n}ski for the helpful and inspiring discussions.
We would like to thank the referees for their valuable comments and suggestions, which have improved the clarity and exposition of the paper.


\phantomsection
\printbibliography[title=References]

@article{StrohmerVershynin2009,
	author      = "Strohmer, Thomas and Vershynin, Roman",
	title 	    = "{A Randomized Kaczmarz Algorithm with Exponential Convergence}",
	journal     = "J. Fourier Anal. Appl.",
	year        = "2009",
	volume      = "15",
	number      = "",
	pages       = "262-278",
	doi         = "10.1007/s00041-008-9030-4",
	eprinttype  = "arXiv",
	eprint      = "math/0702226",
	note 	    = ""
}

@article{derezinski2024sharp,
	author      = "Derezi\'{n}ski, Micha\l{} and Rebrova, Elizaveta",
	title       = "{Sharp Analysis of Sketch-and-Project Methods via a Connection to Randomized Singular Value Decomposition}",
        journal     = "SIAM Journal on Mathematics of Data Science",
	year        = "2024",
	volume      = "6",
	number      = "1",
	pages       = "127-153",
	doi         = "10.1137/23M1545537",
	eprinttype  = "arXiv",
	eprint      = "2208.09585",
	eprintclass = "math.OC",
	note 	    = ""
}

@article{rebrova2021block,
	author      = "Rebrova, Elizaveta and Needell, Deanna",
	title       = "{On block Gaussian sketching for the Kaczmarz method}",
	journal     = "Numerical Algorithms",
	year        = "2021",
	volume      = "86",
	number      = "1",
	pages       = "443-473",
	doi         = "10.1007/s11075-020-00895-9",
	eprinttype  = "arXiv",
	eprint      = "1905.08894",
	eprintclass = "math.PR",
	note        = ""
}

@article{NeSrWa2016,
	author      = "Needell, Deanna and Srebo, Nathan and Ward, Rachel",
	title       = "{Stochastic gradient descent, weighted sampling, and the randomized Kaczmarz algorithm}",
	journal     = "Math. Program.",
	year        = "2016",
	volume      = "155",
	number      = "",
	pages       = "549-573",
	doi         = "10.1007/s10107-015-0864-7",
	eprinttype  = "arXiv",
	eprint      = "1310.5715",
	eprintclass = "math.NA",
	note        = ""
}

@article{Needell2010,
	author      = "Needell, Deanna",
	title       = "{Randomized Kaczmarz solver for noisy linear systems}",
	journal     = "BIT Numer. Math.",
	year        = "2010",
	volume      = "50",
	number      = "",
	pages       = "395-403",
	doi         = "10.1007/s10543-010-0265-5",
	eprinttype  = "arXiv",
	eprint      = "0902.0958",
	eprintclass = "math.NA",
	note        = ""
}

@article{HaNeReSw2022,
	author      = "Haddock, Jamie and Needell, Deanna and Rebrova, Elizaveta and Swartworth, William",
	title       = "{Quantile-Based Iterative Methods for Corrupted Systems of Linear Equations}",
	journal     = "SIAM Journal on Matrix Analysis and Applications",
	year        = "2022",
	volume      = "43",
	number      = "2",
	pages       = "605-637",
	doi         = "10.1137/21M1429187",
	eprinttype  = "arXiv",
	eprint      = "2009.08089",
	eprintclass = "math.NA",
	note        = ""
}

@article{NeedellTropp2014,
	author      = "Needell, Deanna and Tropp, Joel A.",
	title       = "{Paved with good intentions: Analysis of a randomized block Kaczmarz method}",
	journal     = "Linear Algebra and its Applications",
	year        = "2014",
	volume      = "441",
	number      = "",
	pages       = "199-221",
	doi         = "10.1016/j.laa.2012.12.022",
	eprinttype  = "arXiv",
	eprint      = "1208.3805",
	eprintclass = "math.NA",
	note        = ""
}

@inproceedings{candes2005error,
	author      = "Candes, Emmanuel J. and Rudelson, Mark and Tao, Terence and Vershynin, Roman",
	title 	    = "{Error Correction via Linear Programming}",
	booktitle   = "IEEE Symposium on Foundations of Computer Science (FOCS)",
	pages       = "668-681",
	year 	    = "2006",
	doi         = "10.1109/SFCS.2005.5464411",
	eprinttype  = "",
	eprint      = "",
	note 	    = ""
}

@article{GowerRichtarik2015,
	author      = "Gower, Robert M. and Richt\'{a}rik, Peter",
	title 	    = "{Randomized Iterative Methods for Linear Systems}",
	journal     = "SIAM Journal on Matrix Analysis and Applications",
	year 	    = "2015",
	volume      = "36",
	number      = "4",
	pages       = "1660-1690",
	doi         = "10.1137/15M1025487",
	eprinttype  = "arXiv",
	eprint      = "1506.03296",
	eprintclass = "math.NA",
	note 	    = ""
}

@inproceedings{JarmanNeedell2021,
	author      = "Jarman, Benjamin and Needell, Deanna",
	title 	    = "{QuantileRK: Solving large-scale linear systems with corrupted, noisy data}",
	booktitle   = "55th Asilomar Conference on Signals, Systems, and Computers",
	editor      = "",
	publisher   = "",
	pages       = "1312-1316",
	year 	    = "2021",
	volume      = "",
	doi         = "",
	eprinttype  = "arXiv",
	eprint      = "2108.02304",
	eprintclass = "math.NA",
	note 	    = ""
}

@article{LoHaNe2017,
	author      = "De Loera, Jes\'{u}s A. and Haddock, Jamie and Needell, Deanna",
	title       = "{A Sampling Kaczmarz--Motzkin Algorithm for Linear Feasibility}",
	journal     = "SIAM Journal on Scientific Computing",
	year        = "2017",
	volume      = "39",
	number      = "5",
	pages       = "S66-S87",
	doi         = "10.1137/16M1073807",
	eprinttype  = "arXiv",
	eprint      = "1605.01418",
	eprintclass = "math.OC",
	note 	    = ""
}

@article{LeventhalLewis2010,
	author      = "Leventhal, Dennis and Lewis, Adrian S.",
	title 	    = "{Randomized Methods for Linear Constraints: Convergence Rates and Conditioning}",
	journal     = "Mathematics of Operations Research",
	year 	    = "2010",
	volume      = "35",
	number      = "3",
	pages       = "641-654",
	doi         = "10.1287/moor.1100.0456",
	eprinttype  = "arXiv",
	eprint      = "0806.3015",
	eprintclass = "math.OC",
	note 	    = ""
}

@article{Necoara2019,
	author      = "Necoara, Ion",
	title       = "{Faster Randomized Block Kaczmarz Algorithms}",
	journal     = "SIAM Journal on Matrix Analysis and Applications",
	year 	    = "2019",
	volume      = "40",
	number      = "4",
	pages       = "1425-1452",
	doi         = "10.1137/19M1251643",
	eprinttype  = "arXiv",
	eprint      = "1902.09946",
	eprintclass = "math.OC",
	note 	    = ""
}

@article{Steinerberger2023,
	author      = "Steinerberger, Stefan",
	title       = "{Quantile-based Random Kaczmarz for corrupted linear systems of equations}",
	journal     = "Information and Inference: A Journal of the IMA",
	year 	    = "2023",
	volume      = "12",
	number      = "1",
	pages       = "448-465",
	doi         = "10.1093/imaiai/iaab029",
	eprinttype  = "arXiv",
	eprint      = "2107.05554",
	eprintclass = "math.NA",
	note 	    = ""
}

@article{HaddockNeedell2019,
	author      = "Haddock, Jamie and Needell, Deanna",
	title       = "{Randomized Projection Methods for Linear Systems with Arbitrarily Large Sparse Corruptions}",
	journal     = "SIAM Journal on Scientific Computing",
	year        = "2019",
	volume      = "41",
	number      = "5",
	pages       = "S19-S36",
	doi         = "10.1137/18M1179213",
	eprinttype  = "arXiv",
	eprint      = "1803.08114",
	eprintclass = "math.NA",
	note        = ""
}

@article{Steinerberger2021,
	author      = "Steinerberger, Stefan",
	title       = "{Randomized Kaczmarz converges along small singular vectors}",
	journal     = "SIAM Journal on Matrix Analysis and Applications",
	year 	    = "2021",
	volume      = "42",
	number      = "2",
	pages       = "608-615",
	doi         = "10.1137/20M1350947",
	eprinttype  = "arXiv",
	eprint      = "2006.16978",
	eprintclass = "math.NA",
	note 	    = ""
}

@article{NeZhZo2015,
	author      = "Needell, Deanna and Zhao, Ran and Zouzias, Anastasios",
	title       = "{Randomized Block Kaczmarz Method with Projection for Solving Least Squares}",
	journal     = "Linear Algebra and its Applications",
	year        = "2015",
	volume      = "484",
	number      = "",
	pages       = "322-343",
	doi         = "10.1016/j.laa.2015.06.027",
	eprinttype  = "arXiv",
	eprint      = "1403.4192",
	eprintclass = "math.NA",
	note        = ""
}

@article{ZouziasFreris2013,
	author      = "Zouzias, Anastasios and Freris, Nikolaos M.",
	title       = "{Randomized Extended Kaczmarz for Solving Least Squares}",
	journal     = "SIAM Journal on Matrix Analysis and Applications",
	year        = "2013",
	volume      = "34",
	number      = "2",
	pages       = "773-793",
	doi         = "10.1137/120889897",
	eprinttype  = "arXiv",
	eprint      = "1205.5770",
	eprintclass = "math.NA",
	note        = ""
}

@article{MaNeRa2015,
	author      = "Ma, Anna and Needell, Deanna and Ramdas, Aaditya",
	title       = "{Convergence Properties of the Randomized Extended Gauss--Seidel and Kaczmarz Methods}",
	journal     = "SIAM Journal on Matrix Analysis and Applications",
	year        = "2015",
	volume      = "36",
	number      = "4",
	pages       = "1590-1604",
	doi         = "10.1137/15M1014425",
	eprinttype  = "arXiv",
	eprint      = "1503.08235",
	eprintclass = "math.NA",
	note        = ""
}

@article{RudelsonVershynin2015,
	author      = "Rudelson, Mark and Vershynin, Roman",
	title       = "{Small Ball Probabilities for Linear Images of High-Dimensional Distributions}",
	journal     = "International Mathematics Research Notices",
	year        = "2015",
	volume      = "2015",
	number      = "19",
	pages       = "9594-9617",
	doi         = "10.1093/imrn/rnu243",
	eprinttype  = "arXiv",
	eprint      = "1402.4492",
	eprintclass = "math.PR",
	note        = ""
}

@book{Vershynin2018,
    title       = "{High-Dimensional Probability: An Introduction with Applications in Data Science}",
    author      = "Vershynin, Roman",
    year        = "2018",
    publisher   = "Cambridge University Press",
    address     = "Cambridge",
    series      = "",
    doi         = "10.1017/9781108231596"
}

@article{HaMaTr2011,
	author      = "Halko, Nathan and Martinsson, Per-Gunnar and Tropp, Joel A.",
	title       = "{Finding structure with randomness: Probabilistic algorithms for constructing approximate matrix decompositions}",
	journal     = "SIAM Review",
	year        = "2011",
	volume      = "53",
	number      = "2",
	pages       = "217-288",
	doi         = "10.1137/090771806",
	eprinttype  = "arXiv",
	eprint      = "0909.4061",
	eprintclass = "math.NA",
	note        = ""
}

@article{ChJaNeRe2022,
	author      = "Cheng, Lu and Jarman, Benjamin and Needell, Deanna and Rebrova, Elizaveta",
	title       = "{On block accelerations of quantile randomized Kaczmarz for corrupted systems of linear equations}",
	journal     = "Inverse Problems",
	year        = "2023",
	volume      = "39",
	number      = "2",
	pages       = "024002",
	doi         = "10.1088/1361-6420/aca78a",
	eprinttype  = "arXiv",
	eprint      = "2206.12554",
	eprintclass = "math.NA",
	note        = ""
}

@inproceedings{amaldi2005randomized,
	author      = "Amaldi, Edoardo and Belotti, Pietro and Hauser, Raphael",
	title 	    = "{Randomized Relaxation Methods for the Maximum Feasible Subsystem Problem}",
	booktitle   = "Integer Programming and Combinatorial Optimization",
	editor      = "",
	publisher   = "Springer",
        address     = "Berlin, Heidelberg",
	pages       = "249-264",
	year 	    = "2005",
	volume      = "",
	doi         = {10.1007/11496915_19},
	eprinttype  = "",
	eprint      = "",
	eprintclass = "",
	note 	    = ""
}

@inproceedings{dalalyan2019outlier,
	author      = "Dalalyan, Arnak and Thompson, Philip",
	title 	    = "{Outlier-robust estimation of a sparse linear model using $\ell_1$-penalized Huber's $M$-estimator}",
	booktitle   = "Advances in Neural Information Processing Systems",
	editor      = "",
	publisher   = "",
        address     = "",
	pages       = "",
	year 	    = "2019",
	volume      = "32",
	doi         = "",
	eprinttype  = "arXiv",
	eprint      = "1904.06288",
	eprintclass = "math.ST",
	note 	    = ""
}

@article{amaldi1995complexity,
	author      = "Amaldi, Edoardo and Kann, Viggo",
	title       = "{The complexity and approximability of finding maximum feasible subsystems of linear relations}",
	journal     = "Theoretical Computer Science",
	year 	    = "1995",
	volume      = "147",
	number      = "1-2",
	pages       = "181-210",
	doi         = "10.1016/0304-3975(94)00254-G",
	eprinttype  = "",
	eprint      = "",
	eprintclass = "",
	note 	    = ""
}

@article{OymakTropp2018,
	author      = "Oymak, Samet and Tropp, Joel A",
	title       = "{Universality laws for randomized dimension reduction, with applications}",
	journal     = "Information and Inference: A Journal of the IMA",
	year 	    = "2018",
	volume      = "7",
	number      = "3",
	pages       = "337-446",
	doi         = "10.1093/imaiai/iax011",
	eprinttype  = "arXiv",
	eprint      = "1511.09433",
	eprintclass = "math.PR",
	note 	    = ""
}

@article{BriskmanNeedell2015,
	author      = "Briskman, Jonathan and Needell, Deanna",
	title       = "{Block Kaczmarz Method with Inequalities}",
	journal     = "Journal of Mathematical Imaging and Vision",
	year 	    = "2015",
	volume      = "52",
	number      = "",
	pages       = "385-396",
	doi         = "10.1007/s10851-014-0539-7",
	eprinttype  = "arXiv",
	eprint      = "1406.7339",
	eprintclass = "math.NA",
	note 	    = ""
}

@article{Kaczmarz1937,
	author      = "Kaczmarz, Stefan",
	title       = "{Angen\"{a}herte Aufl\"{o}sung von Systemen linearer Gleichungen}",
	journal     = "Bulletin International de l'Academie Polonaise des Sciences et des Lettres",
	year 	    = "1937",
	volume      = "35",
	number      = "",
	pages       = "355-357",
	doi         = "",
	eprinttype  = "",
	eprint      = "",
	eprintclass = "",
	note 	    = ""
}

@article{HaddockMa2021,
	author      = "Haddock, Jamie and Ma, Anna",
	title       = "{Greed Works: An Improved Analysis of Sampling Kaczmarz--Motzkin}",
	journal     = "SIAM Journal on Mathematics of Data Science",
	year        = "2021",
	volume      = "3",
	number      = "1",
	pages       = "342-368",
	doi         = "10.1137/19M1307044",
	eprinttype  = "arXiv",
	eprint      = "1912.03544",
	eprintclass = "math.NA",
	note        = ""
}

@inproceedings{GKLR2019,
	author      = "Gower, Robert M. and Kovalev, Dmitry and Lieder, Felix and Richt\'{a}rik, Peter",
	title 	    = "{RSN: Randomized Subspace Newton}",
	booktitle   = "Advances in Neural Information Processing Systems",
	editor      = "",
	publisher   = "",
	pages       = "",
	year 	    = "2019",
	volume      = "32",
	doi         = "",
	eprinttype  = "arXiv",
	eprint      = "1905.10874",
	eprintclass = "math.OC",
	note 	    = ""
}

@article{Mahoney2011,
	author      = "Mahoney, Michael W.",
	title       = "{Randomized Algorithms for Matrices and Data}",
	journal     = "Foundations and Trends in Machine Learning",
	year        = "2011",
	volume      = "3",
	number      = "2",
	pages       = "123-224",
	doi         = "10.1561/2200000035",
	eprinttype  = "arXiv",
	eprint      = "1104.5557",
	eprintclass = "cs.DS",
	note        = ""
}

@article{NeedellWard2013,
	author      = "Needell, Deanna and Ward, Rachel",
	title       = "{Two-Subspace Projection Method for Coherent Overdetermined Systems}",
	journal     = "Journal of Fourier Analysis and Applications",
	year        = "2013",
	volume      = "19",
	number      = "",
	pages       = "256-269",
	doi         = "10.1007/s00041-012-9248-z",
	eprinttype  = "arXiv",
	eprint      = "1204.0277",
	eprintclass = "math.NA",
	note        = "Associated technical report arXiv:1204.0279"
}

@misc{JinEtAl2019,
	author      = "Jin, Chi and Netrapalli, Praneeth and Ge, Rong and Kakade, Sham M. and Jordan, Michael I.",
	title 	    = "{A Short Note on Concentration Inequalities for Random Vectors with SubGaussian Norm}",
	year 	    = "2019",
	eprinttype  = "arXiv",
	eprint      = "1902.03736",
	eprintclass = "math.PR",
	note 	    = "arXiv preprint arXiv:1902.03736"
}

@inproceedings{derezinski2023solving,
	author      = "Derezi{\'n}ski, Micha{\l} and Yang, Jiaming",
	title 	    = "{Solving Dense Linear Systems Faster than via Preconditioning}",
        booktitle   = "ACM Symposium on Theory of Computing (STOC)",
        pages       = "",
	year 	    = "2024",
	eprinttype  = "arXiv",
	eprint      = "2312.08893",
	eprintclass = "cs.DS",
	note 	    = ""
}

@misc{derezinski2024fine,
	author      = "Derezi{\'n}ski, Micha{\l} and LeJeune, Daniel and Needell, Deanna and Rebrova, Elizaveta",
	title 	    = "{Fine-grained Analysis and Faster Algorithms for Iteratively Solving Linear Systems}",
	year 	    = "2024",
	eprinttype  = "arXiv",
	eprint      = "2405.05818",
	eprintclass = "cs.DS",
	note 	    = "arXiv preprint arXiv:2405.05818"
}

@article{du2021kaczmarz,
	author      = "Du, Yi-Shu and Hayami, Ken and Zheng, Ning and Morikuni, Keiichi and Yin, Jun-Feng",
	title       = "{Kaczmarz-type inner-iteration preconditioned flexible GMRES methods for consistent linear systems}",
	journal     = "SIAM Journal on Scientific Computing",
	year        = "2021",
	volume      = "43",
	number      = "5",
	pages       = "S345-S366",
	doi         = "10.1137/20M1344937",
	eprinttype  = "arXiv",
	eprint      = "2006.10818",
	eprintclass = "math.NA",
	note        = ""
}

@article{chen2021regularized,
	author      = "Chen, Xuemei and Qin, Jing",
	title       = "{Regularized Kaczmarz algorithms for tensor recovery}",
	journal     = "SIAM Journal on Imaging Sciences",
	year        = "2021",
	volume      = "14",
	number      = "4",
	pages       = "1439-1471",
	doi         = "10.1137/21M1398562",
	eprinttype  = "arXiv",
	eprint      = "2102.06852",
	eprintclass = "math.OC",
	note        = ""
}

@book{Natterer2001,
    author      = "Natterer, Frank",
    title       = "The Mathematics of Computerized Tomography",
    year        = "2001",
    publisher   = "Society for Industrial and Applied Mathematics",
    address     = "",
    series      = "Classics in Applied Mathematics",
    doi         = "10.1137/1.9780898719284",
    note        = ""
}

@book{Herman2009,
    author      = "Herman, Gabor T.",
    title       = "{Fundamentals of Computerized Tomography}",
    year        = "2009",
    publisher   = "Springer-Verlag London",
    address     = "",
    series      = "Advances in Computer Vision and Pattern Recognition",
    doi         = "10.1007/978-1-84628-723-7",
    note        = ""
}

@inproceedings{CenkerEtAl1992,
	author      = "Cenker, C. and Feichtinger, Hans Georg and Mayer, M. and Steier, H. and Strohmer, Thomas",
	title 	    = "{New variants of the POCS method using affine subspaces of finite codimension with applications to irregular sampling}",
	booktitle   = "Visual Communications and Image Processing '92",
	volume      = "1818",
	pages       = "299-310",
	year 	    = "1992",
	doi         = "",
	eprinttype  = "",
	eprint      = "",
	note 	    = ""
}

@article{LiuWright2016,
	author      = "Liu, Ji and Wright, Stephen J.",
	title 	    = "{An Accelerated Randomized Kaczmarz Algorithm}",
	journal     = "Mathematics of Computation",
	year 	    = "2016",
	volume      = "85",
	number      = "297",
	pages       = "153-178",
	doi         = "10.1090/mcom/2971",
	eprinttype  = "arXiv",
	eprint      = "1310.2887",
	eprintclass = "math.NA",
	note 	    = ""
}

@article{SchopferLorenz2019,
	author      = "Sch{\"{o}}pfer, F. and Lorenz, D.A.",
	title 	    = "{Linear convergence of the randomized sparse Kaczmarz method}",
	journal     = "Mathematical Programming",
	year 	    = "2019",
	volume      = "173",
	number      = "",
	pages       = "509-536",
	doi         = "10.1007/s10107-017-1229-1",
	eprinttype  = "arXiv",
	eprint      = "1610.02889",
	eprintclass = "math.OC",
	note 	    = ""
}

@inproceedings{LorenzEtAl2014,
	author      = "Lorenz, Dirk A. and Wenger, Stephan and Sch{\"{o}}pfer, Frank and Magnor, Marcus A.",
	title 	    = "{A sparse Kaczmarz solver and a linearized Bregman method for online compressed sensing}",
	booktitle   = "IEEE International Conference on Image Processing (ICIP)",
	editor      = "",
	publisher   = "",
	pages       = "1347-1351",
	year 	    = "2014",
	volume      = "",
	doi         = "10.1109/ICIP.2014.7025269",
	eprinttype  = "arXiv",
	eprint      = "1403.7543",
	eprintclass = "math.OC",
	note 	    = ""
}

@book{HornJohnson2012,
    author      = "Horn, Roger A. and Johnson, Charles R.",
    title       = "{Matrix Analysis}",
    edition     = "2",
    year        = "2012",
    publisher   = "Cambridge University Press",
    address     = "Cambridge",
    series      = "",
    doi         = "10.1017/CBO9781139020411"
}

@article{Elfving1980,
	author      = "Elfving, Tommy",
	title 	    = "{Block-Iterative Methods for Consistent and Inconsistent Linear Equations}",
	journal     = "Numerische Mathematik",
	year        = "1980",
	volume      = "35",
	number      = "",
	pages       = "1-12",
	doi         = "10.1007/BF01396365",
	eprinttype  = "",
	eprint      = "",
	note 	    = ""
}

@inproceedings{KOSZ2013,
	author      = "Kelner, Jonathan A. and Orecchia, Lorenzo and Sidford, Aaron and Zhu, Zeyuan Allen",
	title       = "{A Simple, Combinatorial Algorithm for Solving SDD Systems in Nearly-Linear Time}",
	booktitle   = "ACM Symposium on Theory of Computing (STOC)",
        series      = "",
	year        = "2013",
	volume      = "",
	number      = "",
	pages       = "911–920",
	doi         = "10.1145/2488608.2488724",
	eprinttype  = "arXiv",
	eprint      = "1301.6628",
	eprintclass = "cs.DS",
	note        = ""
}

@article{TanVershynin2019,
	author      = "Tan, Yan Shuo Tan and Vershynin, Roman",
	title       = "{Phase retrieval via randomized Kaczmarz: theoretical guarantees}",
	journal     = "Information and Inference: A Journal of the IMA",
	year        = "2019",
	volume      = "8",
	number      = "1",
	pages       = "97-123",
	doi         = "10.1093/imaiai/iay005",
	eprinttype  = "arXiv",
	eprint      = "1706.09993",
	eprintclass = "math.NA",
	note        = ""
}

@article{DrMaMu2008,
	author      = "Drineas, Petros and Mahoney, Michael W. and Muthukrishnan, S.",
	title       = "{Relative-Error $CUR$ Matrix Decompositions}",
	journal     = "SIAM Journal on Matrix Analysis and Applications",
	year        = "2008",
	volume      = "30",
	number      = "2",
	pages       = "844-881",
	doi         = "10.1137/07070471X",
	eprinttype  = "arXiv",
	eprint      = "0708.3696",
	eprintclass = "cs.DS",
	note        = ""
}

@article{DrKaMa2006b,
	author      = "Drineas, Petros and Kannan, Ravi and Mahoney, Michael W.",
	title       = "{Fast Monte Carlo Algorithms for Matrices II: Computing a Low-Rank Approximation to a Matrix}",
	journal     = "SIAM Journal on Computing",
	year        = "2006",
	volume      = "36",
	number      = "1",
	pages       = "158-183",
	doi         = "10.1137/S0097539704442696",
	eprinttype  = "",
	eprint      = "",
	eprintclass = "",
	note        = ""
}

@inproceedings{BoMaDr2009,
	author      = "Boutsidis, Christos and Mahoney, Michael W. and Drineas, Petros",
	title 	    = "{An Improved Approximation Algorithm for the Column Subset Selection Problem}",
	booktitle   = "ACM-SIAM Symposium on Discrete Algorithms (SODA)",
	editor      = "",
	publisher   = "",
	pages       = "968-977",
	year 	    = "2009",
	volume      = "",
	doi         = "10.1137/1.9781611973068.105",
	eprinttype  = "arXiv",
	eprint      = "0812.4293",
	eprintclass = "cs.DS",
	note 	    = ""
}

@article{HansenJorgensen2018,
	author      = "Hansen, P.C. and Jørgensen, J.S.",
	title       = "{AIR Tools II: algebraic iterative reconstruction methods, improved implementation}",
	journal     = "Numerical Algorithms",
	year        = "2018",
	volume      = "79",
	number      = "",
	pages       = "107-137",
	doi         = "10.1007/s11075-017-0430-x",
	eprinttype  = "",
	eprint      = "",
	eprintclass = "",
	note        = ""
}

@article{DrMaMaWo2012,
	author      = "Drineas, Petros and Magdon-Ismail, Malik and Mahoney, Michael W. and Woodruff, David P.",
	title       = "{Fast Approximation of Matrix Coherence and Statistical Leverage}",
	journal     = "The Journal of Machine Learning Research",
	year        = "2012",
	volume      = "13",
	number      = "1",
	pages       = "3475-3506",
	doi         = "10.5555/2503308.2503352",
	eprinttype  = "arXiv",
	eprint      = "1109.3843",
	eprintclass = "cs.DS",
	note        = ""
}

@article{NecoaraTakac2021,
	author      = "Necoara, Ion and Tak\'{a}\v{c}, Martin",
	title       = "{Randomized sketch descent methods for non-separable linearly constrained optimization}",
	journal     = "IMA Journal of Numerical Analysis",
	year        = "2021",
	volume      = "41",
	number      = "2",
	pages       = "1056-1092",
	doi         = "10.1093/imanum/draa018",
	eprinttype  = "arXiv",
	eprint      = "1808.02530",
	eprintclass = "math.OC",
	note        = ""
}

@article{SaumardWellner2014,
	author      = "Saumard, Adrien and Wellner, Jon A.",
	title       = "{Log-concavity and strong log-concavity: A review}",
	journal     = "Statistics Surveys",
	year        = "2014",
	volume      = "8",
	number      = "",
	pages       = "45-114",
	doi         = "10.1214/14-SS107",
	eprinttype  = "arXiv",
	eprint      = "1404.5886",
	eprintclass = "math.ST",
	note        = ""
}

@article{Wu2022,
	author      = "Wu, Wen-Ting",
	title       = "{On two-subspace randomized extended Kaczmarz method for solving large linear least-squares problems}",
	journal     = "Numerical Algorithms",
	year        = "2022",
	volume      = "89",
	number      = "",
	pages       = "1-31",
	doi         = "10.1007/s11075-021-01104-x",
	eprinttype  = "",
	eprint      = "",
	eprintclass = "",
	note        = ""
}

@article{CarberyWright2001,
	author      = "Carbery, Anthony and Wright, James",
	title       = "{Distributional and $L^q$ norm inequalities for polynomials over convex bodies in $\mathbb{R}^n$}",
	journal     = "Mathematical Research Letters",
	year        = "2001",
	volume      = "8",
	number      = "",
	pages       = "233-248",
	doi         = "10.4310/MRL.2001.v8.n3.a1",
	eprinttype  = "",
	eprint      = "",
	eprintclass = "",
	note        = ""
}

@article{SchopferEtAl2022,
	author      = "Sch{\"{o}}pfer, Frank and Lorenz, Dirk A. and Tondji, Lionel and Winkler, Maximilian",
	title       = "{Extended randomized Kaczmarz method for sparse least squares and impulsive noise problems}",
	journal     = "Linear Algebra and its Applications",
	year        = "2022",
	volume      = "652",
	number      = "",
	pages       = "132-154",
	doi         = "10.1016/j.laa.2022.07.003",
	eprinttype  = "arXiv",
	eprint      = "2201.08620",
	eprintclass = "math.NA",
	note        = ""
}

@article{LorenzEtAl2018,
	author      = "Lorenz, Dirk A. and Rose, Sean and Sch{\"{o}}pfer, Frank",
	title       = "{The randomized Kaczmarz method with mismatched adjoint}",
	journal     = "BIT Numerical Mathematics",
	year        = "2018",
	volume      = "58",
	number      = "",
	pages       = "1079-1098",
	doi         = "10.1007/s10543-018-0717-x",
	eprinttype  = "arXiv",
	eprint      = "1803.02848",
	eprintclass = "math.NA",
	note        = ""
}

@article{MeierNakatsukasa2024,
	author      = "Meier, Maike and Nakatsukasa, Yuji",
	title       = "{Fast Randomized Numerical Rank Estimation for Numerically Low-Rank Matrices}",
	journal     = "Linear Algebra and its Applications",
	year        = "2024",
	volume      = "686",
	number      = "",
	pages       = "1-32",
	doi         = "10.1016/j.laa.2024.01.001",
	eprinttype  = "arXiv",
	eprint      = "2105.07388",
	eprintclass = "math.NA",
	note        = ""
}

@inproceedings{AgWaLu2014,
	author      = "Agaskar, A. and Wang, C. and Lu, Y. M.",
	title       = "{Randomized Kaczmarz algorithms: Exact MSE analysis and optimal sampling probabilities}",
	booktitle   = "IEEE Global Conference on Signal and Information Processing (GlobalSIP)",
        series      = "",
	year        = "2014",
	volume      = "",
	number      = "",
	pages       = "389-393",
	doi         = "10.1109/GlobalSIP.2014.7032145",
	eprinttype  = "",
	eprint      = "",
	eprintclass = "",
	note        = ""
}

@article{StKuPoBo2012,
	author      = "Studer, Christoph and Kuppinger, Patrick and Pope, Graeme and Bolcskei, Helmut",
	title       = "{Recovery of Sparsely Corrupted Signals}",
	journal     = "IEEE Transactions on Information Theory",
	year        = "2012",
	volume      = "58",
	number      = "5",
	pages       = "3115-3130",
	doi         = "10.1109/TIT.2011.2179701",
	eprinttype  = "arXiv",
	eprint      = "1102.1621",
	eprintclass = "cs.IT",
	note        = ""
}
\addcontentsline{toc}{section}{References}

\end{document}